%% file: main.tex
\documentclass[12pt]{amsart}
\usepackage[margin=1in]{geometry} 

\usepackage[dvipsnames]{xcolor}
\usepackage{hyperref}
\hypersetup{colorlinks=true, citecolor=Blue, linkcolor=Blue} 

\usepackage{lipsum}
\usepackage{tikz-cd}
\usepackage{graphicx}
\usepackage{amssymb,amsmath}
\usepackage{mathtools}
\usepackage{mathrsfs}
\usepackage{enumerate}
\usepackage{ragged2e}
\usepackage{enumitem,pifont}
\usepackage{tikz}
\usetikzlibrary{arrows,decorations.pathmorphing,automata,backgrounds}
\usetikzlibrary{backgrounds,positioning}

\usepackage{caption}
\usepackage{subcaption}
\usepackage{float}

\usepackage{egothic}
\usepackage[T1]{fontenc}
\usepackage{yfonts}

\usepackage{calc}

\definecolor{darkblue}{rgb}{0,0,0.7} 
\definecolor{green}{RGB}{57,181,74} 
\definecolor{violet}{RGB}{147,39,143} 

\newcommand{\defn}[1]{{\emph{#1}}}

\newcommand{\T}{\mathbb{T}}

\newcommand{\Sym}{\mathbb{S}}

\newcommand{\eqdef}{:=}
\newcommand{\defeq}{\mbox{~\ensuremath{=}\raisebox{0.2ex}{\scriptsize\ensuremath{\mathrm:}} }} 

\newcommand{\calC}{\mathcal{C}}

\newcommand{\calI}{\mathcal{I}}
\newcommand{\calJ}{\mathcal{J}}
\newcommand{\calL}{\mathcal{L}}
\newcommand{\calF}{\mathcal{F}}
\newcommand{\calG}{\mathcal{G}}

\newcommand{\calQ}{\mathcal{Q}}
\newcommand{\calR}{\mathcal{R}}
\newcommand{\calT}{\mathcal{T}}

\newcommand{\calX}{\mathcal{X}}
\newcommand{\calY}{\mathcal{Y}}

\newcommand{\rmC}{\mathrm{C}}

\newcommand{\rmI}{\mathrm{I}}
\newcommand{\rmO}{\mathrm{O}}
\newcommand{\rmP}{\mathrm{P}}
\newcommand{\rmQ}{\mathrm{Q}}
\newcommand{\rmT}{\mathrm{T}}

\newcommand{\id}{\operatorname{id}}

\newcommand{\End}{\operatorname{End}}

\newcommand{\Lie}{\operatorname{Lie}}
\newcommand{\cLie}{\operatorname{cLie}}

\newcommand{\ev}{\operatorname{ev}}

\newcommand{\forget}{\calF}
\newcommand{\free}{\calG}

\newcommand{\Curv}{\operatorname{Curv}_\infty}
\newcommand{\coCurv}{\operatorname{Curv}^\as}
\newcommand{\nsCurv}{\operatorname{Curv}_\infty^{\mathrm{ns}}}

\newcommand{\coTwist}{\operatorname{Twist}^\as}
\newcommand{\Twist}{\operatorname{Twist}_\infty}
\newcommand{\nsTwist}{\operatorname{Twist}_\infty^{\mathrm{ns}}}

\newcommand{\kurv}{\kappa}
\newcommand{\diff}{m}
\newcommand{\Vect}{\operatorname{Vect}}
\newcommand{\Op}{\operatorname{Op}}
\newcommand{\coOp}{\operatorname{coOp}}
\newcommand{\dgOp}{\operatorname{dg-Op}}
\newcommand{\Mult}{(\Linf \downarrow \Op)}
\newcommand{\nsMult}{(\Ainf \downarrow \Op)}
\newcommand{\MultSym}{(\Linf \downarrow \Op)}

\newcommand{\cMult}{(\cLinf \downarrow \Op)}
\newcommand{\nscMult}{(\cAinf \downarrow \Op)}
\newcommand{\cMultSym}{(\cLinf \downarrow \Op)}

\newcommand{\coCom}{(\uComd \downarrow \coOp)}

\newcommand{\form}{\varphi}
\newcommand{\counit}{\xi}

\usepackage[capitalise]{cleveref}

\newcounter{dummy} \numberwithin{dummy}{section}
\newtheorem{theorem}[dummy]{Theorem}
\newtheorem{thm}[dummy]{Theorem}
\newtheorem{prop}[dummy]{Proposition}
\newtheorem{lemma}[dummy]{Lemma}

\newtheorem{corollary}[dummy]{Corollary}

\newtheorem*{thm*}{Theorem}
\newtheorem*{prop*}{Proposition}

\theoremstyle{definition}
\newtheorem{definition}[dummy]{Definition}

\newtheorem{notation}[dummy]{Notation}
\newtheorem{def-prop}[dummy]{Definition-Proposition}
\newtheorem{lemma-def}[dummy]{Lemma-Definition}
\numberwithin{equation}{section}

\theoremstyle{remark}

\newtheorem{remark}[dummy]{Remark}

\usepackage{multicol}

\newcommand{\As}{\ensuremath{\mathrm{As}}}

\newcommand{\Ainf}{\ensuremath{\mathrm{A}_\infty}}

\newcommand{\cAinf}{\ensuremath{\mathrm{cA}_\infty}}

\newcommand{\Linf}{\ensuremath{\mathrm{L}_\infty}}

\newcommand{\cLinf}{\ensuremath{\mathrm{cL}_\infty}}

\newcommand{\Gerinf}{\ensuremath{\mathrm{Ger}_\infty}}

\newcommand{\Br}{\ensuremath{\mathrm{Br}}}

\newcommand{\BT}{\ensuremath{\mathrm{BT}}}

\newcommand{\MC}{\ensuremath{\mathrm{MC}}}

\newcommand{\Tw}{\ensuremath{\mathrm{Tw}}}

\newcommand{\cTw}{\ensuremath{\mathrm{cTw}}}

\newcommand{\Com}{\ensuremath{\mathrm{Com}}}
\newcommand{\Ass}{\ensuremath{\mathrm{Ass}}}
\newcommand{\uAss}{\ensuremath{\mathrm{uAss}}}

\newcommand{\uCom}{\ensuremath{\mathrm{uCom}}}
\newcommand{\uComd}{\ensuremath{\mathrm{uCom}^{*}}}
\newcommand{\PreLie}{\ensuremath{\mathrm{PreLie}}}
\newcommand{\uPreLie}{\ensuremath{\mathrm{uPreLie}}}

\newcommand{\HypCom}{\ensuremath{\mathrm{HyperCom}}}
\newcommand{\Ger}{\ensuremath{\mathrm{Ger}}}
\newcommand{\ncGer}{\ensuremath{\mathrm{ncGer}}}
\newcommand{\uGer}{\ensuremath{\mathrm{uGer}}}

\newcommand{\Ginf}{\ensuremath{\mathrm{Ger}_\infty}}

\newcommand{\Grav}
{\ensuremath{\mathrm{Grav}}}

\newcommand{\Grainf}
{\ensuremath{\mathrm{Grav}_\infty}}

\newcommand{\Perm}{\ensuremath{\mathrm{Perm}}}

\newcommand{\Graph}{\ensuremath{\mathrm{Gra}}}

\newcommand{\Perinf}{\ensuremath{\mathrm{Perm}_\infty}}

\newcommand{\Gra}{\ensuremath{\mathrm{Gra}}}
\newcommand{\FM}{\ensuremath{\mathrm{FM}}}

\newcommand{\Pinf}{\ensuremath{\rmP_\infty}}
\newcommand{\Qinf}{\ensuremath{\rmQ_\infty}}
\newcommand{\cPinf}{\ensuremath{\mathrm{c}\rmP_\infty}}

\newcommand{\cP}{\ensuremath{\mathrm{c}\rmP}}
\newcommand{\cQ}{\ensuremath{\mathrm{c}\rmQ}}
\newcommand{\uP}{\ensuremath{\mathrm{u}\rmP}}

\newcommand{\Sh}{\ensuremath{\mathrm{Sh}}}

\newcommand{\im}{\ensuremath{\mathrm{im}}}

\newcommand{\op}{\ensuremath{\mathrm{op}}}

\newcommand{\as}{{\scriptstyle \text{\rm !`}}}

\usepackage{todonotes}

\title{Operadic twisting as an adjunction}

\author{Guillaume Laplante-Anfossi}

\author{Adrian Petr}

\author{Vivek Shende}

\begin{document}

\begin{abstract} 
For operads with a map from the curved homotopy Lie operad, we introduce a corresponding curved variant `cTw' of Willwacher's operadic twisting comonad `Tw'.  
We show that cTw-coalgebra structures on such an operad are in bijection with certain splittings (not respecting the differential) of the projection to its quotient by the curvature operation.   
We derive a similar classification of Tw-coalgebras.  

For the class of operads whose Koszul dual admits a unital extension, we give explicit formulas for the cTw-coalgebra structures on their curved homotopy resolutions, recovering the convolution Lie algebra's ``gauge group action'' of Dotsenko, Shadrin, and Vallette.
\end{abstract}

\maketitle


\input{uncurving2}


\newpage

\bibliographystyle{plain}
\bibliography{uncurving}

\end{document}

%% file: uncurving2.tex

\section{Introduction}

Given an $\Linf$ algebra in chain complexes $(A, d)$, and an element
$\alpha \in A$, one may define new  operations 
by the following formula:
\begin{equation}
\label{eq:ell-n-alpha}
    \ell_n^\alpha \eqdef \sum_{k \geq 0} \frac{1}{k!}\ell_{k+n} (\alpha^k,-,\cdots,-)  .
\end{equation}
The new operations determine a $\cLinf$ algebra structure on the underlying chain complex $(A,d)$.  If $\ell_0^\alpha = 0$ vanishes, i.e.\ if $\alpha$ is a Maurer-Cartan element,  the $\ell_n^\alpha$ for $n \geqslant 2$ define an $\Linf$ structure on $(A, d^\alpha := d + \ell_1^\alpha)$. 

By `operadic twisting', we mean, following Willwacher \cite{Willwacher-KGT}, the following general problem. 
Given a morphism of operads 
$\Linf \to \rmP$, a $\rmP$-algebra $A$, and a Maurer-Cartan element $\alpha$ for the $\Linf$ structure on $A$, construct / classify $\rmP$-algebra structures extending the $\Linf$ structure on $A^\alpha$. 

Willwacher's approach to this question was to introduce an operad $\Tw \rmP$ encoding all the natural operations on a $\rmP$-algebra equipped with a Maurer-Cartan element for the $\Linf$ structure \cite{Willwacher-KGT, Willwacher-homotopybraces, DolgushevWillwacher15}. 
There are morphisms of operads: 
    \begin{center}
    \begin{tikzcd}
    \Linf \arrow[r] \arrow[rd] \arrow[rrd, dashed] & \rmP \arrow[r]  & \End_{(A, d)} \\
    \, & \Tw \rmP \arrow[r, dashed] \arrow[u] & \End_{(A, d^\alpha)}
    \end{tikzcd}
    \end{center}
The dashed arrows are  induced by the choice of Maurer-Cartan element $\alpha$.

A choice of section $i: \rmP \to \Tw\rmP$ (as $\Linf$ operads) provides an operadic twisting for $\rmP$.   
In fact, there is also a natural map $\Tw \rmP \to \Tw (\Tw \rmP)$ which, together with the structure map~$\Tw \rmP \to \rmP$, determines a comonad, with the property that when $i: \rmP \to \Tw\rmP$ is a coalgebra for this comonad, the resulting operadic twisting is linear in the Maurer-Cartan elements. 

The comonad $\Tw$ plays various roles in the context of Kontsevich formality,  Deligne's conjecture on the Hochschild complex and related problems \cite{Willwacher-KGT, DolgushevWillwacher15, Willwacher-homotopybraces}.
Two problems which recur in the latter references are to construct $\Tw$-coalgebra structures and to identify $\Tw$ homotopy fixed points; we note that the calculations involved in doing so have been rather non-trivial (see for instance the proof that $\Gerinf$ is a $\Tw$-coalgebra \cite[Thm.~4.12, pp.\ 1382--1387]{DolgushevWillwacher15} or the proof that any operad obtained via distributive law with $\Lie$ is a $\Tw$-coalgebra and homotopy fixed point \cite[Thm.~5.10, pp.\ 1390--1396]{DolgushevWillwacher15}).

\vspace{2mm}

Let us recall the basic relationship between co/monads and adjoints. 
For any adjoint pair of functors 
\begin{tikzcd}
L: \calX \arrow[r,shift left=.5ex]
&
\calY : R \arrow[l,shift left=.5ex]
\end{tikzcd}
with unit $\eta: 1_\calX \to RL$ and counit $\varepsilon: LR \to 1_\calY$, the endofunctor $LR: \calY \to \calY$ carries the structure of a comonad, with counit $\varepsilon$ and cocomposition given by $L\eta R: LRLR \to LR$. 
For $x \in \calX$ the object $L(x)$ carries a natural $LR$-coalgebra structure, given by $L(\eta_x): L(x) \mapsto LRL(x)$. 
When this map from $\calX$ to $LR$-coalgebras is an equivalence, the adjoint pair $(L, R)$ is said to be comonadic.
We say a functor is comonadic when it admits a right adjoint and the adjunction is comonadic. 
Likewise, $RL$ carries a monad structure, there is a map from $\calY$ to the category of $RL$-algebras, etc. 

Tautologically, $\Tw$ is the comonad associated to the comonadic forgetful functor from $\Tw$-coalgebras to $\MultSym$.
The basic purpose of our article is to give a rather different description of this adjunction.
Said differently, we will give an explicit and simple characterization of $\Tw$-coalgebras.
We will also give a criterion for identifying homotopy fixed points of $\Tw$.
Our results will, in particular, recover corresponding results of \cite{DolgushevWillwacher15}.   

\subsection*{The comonad $\cTw$ as an adjunction}

In fact, the corresponding problem is somewhat simpler for the analogue of operadic twisting in the curved setting, that is, for operads $\cLinf \to \rmQ$ under the curved homotopy Lie operad.  There is a comonadic endofunctor
$\cTw$ on the category of such operads such that $\cTw(\rmQ)$ encodes all operations which can be built from $\rmQ$ and some chosen element of a $\rmQ$-algebra (see \cref{def of cTw} below). A $\cTw$-coalgebra structure on $\rmQ$ provides a way to twist a $\rmQ$-algebra $A$ by any element of $A$ compatibly with the corresponding $\cLinf$ twisting (\cref{curved operadic twisting}). 

We will write $[x]$ for the operad generated by a single element $x$ (of some specified arity). Given an operad $\rmP$, the coproduct $\rmP \vee [x]$ is naturally equipped with a ``$x$-grading'' given by the number of $x$ which appear in a given expression.

\begin{definition} \label{curv redefined}
    A  {\em curved twisting structure} is an operad morphism $\cLinf \to \rmQ$, and a non-dg suboperad $\rmQ_0 \subset \rmQ$ such that:
    \begin{itemize}
        \item\label{coproduct curv} The natural map $[\ell_0] \vee \rmQ_0 \to \rmQ$ is an isomorphism of non-dg operads. 
        \item\label{image curv} $\rmQ_0$ contains the image of the $\ell_i$ for $i > 0$. 
        \item\label{differential curv} $d_\rmQ(\rmQ_0)$ consists of elements of degree at most one in $\ell_0$.
    \end{itemize} 
    A morphism of curved twisting structures $(\rmQ, \rmQ_0) \to (\rmQ', \rmQ_0')$ is a morphism under $\cLinf$ carrying $\rmQ_0 \to \rmQ'_0$. 
    We denote the category of curved twisting structures by $\Curv$, and write $\nsCurv$ for the similar notion for non-symmetric operads over $\cAinf$. 
\end{definition} 

Here we prove: 

\begin{theorem} 
\label{cTw comonadicity}
    The forgetful functor $\Curv \to \cMult$ is comonadic and the corresponding comonad is $\cTw$. 
    The same holds for $\nsCurv \to \nscMult$.  
\end{theorem}

\begin{remark}
    Often co/monadicity results serve to access objects in a difficult-to-understand category $Y$ in terms of co/algebras in an easier-to-understand category $X$.  
    The present result is perhaps more useful in the other direction: the data and checking involved in the definition of a $\cTw$-coalgebra structure seems rather more involved than the (equivalent per Theorem \ref{cTw comonadicity}) data of a lift to $\Curv$, i.e., identifying $\rmQ_0$.   
\end{remark}

In fact, the construction of an adjoint to this forgetful functor was the second main result of our previous article~\cite{Laplante-Petr-Shende2025}.  (The first main result showed that the forgetful map from $\Curv$ to tuples $(\rmQ, d_\rmQ, \ell_0, \rmQ_0)$ is fully faithful, and characterized the image. That is, the  map from $\cAinf$ or $\cLinf$ is determined from just the data of $\ell_0$ and $\rmQ_0$.  In fact, in \cite{Laplante-Petr-Shende2025}, we used said image as the definition of $\Curv$.)  Given  the adjoint, we establish the asserted comonadicity by checking Beck's criterion~\cite{Beck25}.  The main new content in the result is the identification of the comonad with $\cTw$, which is an explicit calculation. 

\subsection*{The comonad $\Tw$ as an adjunction}

We return now to the original context of operads over $\Linf$ (rather than $\cLinf$) and $\Tw$ (rather than $\cTw$).  
\begin{definition} \label{twist new}
A {\em twisting structure} is an operad morphism $\cLinf \to \rmQ$, and a non-dg suboperad $\rmQ_{0,0} \subset \rmQ$, such that: 
    \begin{itemize}
        \item\label{coproduct} The natural map $[\ell_0] \vee [\ell_1] \vee \rmQ_{0,0} \to \rmQ$ is an isomorphism of non-dg operads. 
        \item\label{image} $\rmQ_{0,0}$ contains the images of the $\ell_i$ for $i > 1$.
        \item\label{differential} $(d_\rmQ + [\ell_1, \cdot])(\rmQ_{0,0})$ consists of elements of degree $0$ in $\ell_1$ and degree $\le 1$ in $\ell_0$.   
    \end{itemize} 
We write $\Twist$ for the category of such diagrams, and $\nsTwist$ for the analogous notion for nonsymmetric operads and $\Ainf, \cAinf$ in place of $\Linf, \cLinf$. 
\end{definition}

As for $\Curv$, the forgetful functor from $\Twist$ to tuples $(\rmQ, d_{\rmQ}, \mu_0, \mu_1, \rmQ_{0,0})$ is fully faithful -- i.e., we can reconstruct the $\cAinf$ or $\cLinf$ structure from this data -- and we can explicitly characterize the image.
This fact is not used in the present article, but is presented in \cref{sec:reconstruction}.  

It is immediate from the definition that the map $\cLinf \to \rmQ$ extends uniquely to a diagram 
    \begin{center}
    \begin{tikzcd}
    \cLinf \arrow[r] \arrow[d, two heads]  & \rmQ \arrow[d, two heads] \\
    \Linf \arrow[bend right=60, dashed, hookrightarrow]{u} \arrow[r] & \rmQ_{0,0} \arrow[bend right=60, dashed, hookrightarrow]{u}
    \end{tikzcd}
    \end{center}
where we give $\rmQ_{0,0}$ a dg structure via $\rmQ_{0,0} \cong \rmQ/\left( \ell_0, \ell_1 \right)$, the solid arrows are morphisms of dg operads,  the upward dashed arrows are non-dg sections of the downward projections, and both squares commute. 

We establish the following relation between $\Twist$ and $\Tw$: 

\begin{theorem}
\label{thm:main-thm} 
    The forgetful functor $\Twist \to \MultSym$ sending a diagram as above to $\Linf \to \rmQ_{0,0}$ is comonadic, and the corresponding comonad is $\Tw$.  
    The same holds for $\nsTwist \to \nsMult$.  
\end{theorem}

Our result in \cite{Laplante-Petr-Shende2025} {\em did not} come with an explicit formula for the unit of the adjunction (though does provide a recursive procedure to determine it), and correspondingly \cref{cTw comonadicity} does not come with a formula for the $\cTw$-coalgebra structure. 

\subsection*{The Maurer--Cartan monad and uncurving}

For an $\Linf$-operad $\rmP$, the data of a $\rmP$-algebra~$A$ and a Maurer--Cartan element $a \in A$ is encoded by an operad $\MC(\rmP)=\rmP \hat{\vee} [\alpha]$ where $d\alpha = -\ell_0^\alpha$.
This can be regarded as an endofunctor $\MC$ on $\MultSym$, which admits a natural lift to $\Twist$ (\cref{definition:cMC}).

\begin{theorem}
\label{thm:co-main-thm} 
    The functor $\MultSym \to \Twist$ sending $\rmP$ to $(\cTw(\rmP)^+,\Tw(\rmP))$ is monadic, and the corresponding monad is $\MC$.  
    The same holds for $\nsMult \to \nsTwist$.  
\end{theorem}

In particular, for a $\cLinf$-operad $\rmP$, we have a bijection between Maurer--Cartan elements in $\rmP(0)$ and $\MC$-algebra structures on the image of $\rmP$ under the right adjoint functor $\cMult \to \Twist$ (\cref{MC element = MC-algebra structures}).
Using this, we can formalize the procedure of uncurving, e.g.\ the fact that twisting a $\cLinf$-algebra by a Maurer--Cartan element gives a $\Linf$-algebra, as follows.
\begin{prop}
    Let $(\rmQ,\rmQ_{0,0})$ be an object of $\Twist$, let $\rmP$ be an object of $\cMult$, and consider a morphism $\rmQ \to \rmP$.
    Then, any Maurer--Cartan element $a \in \rmP(0)$ induces a morphism of $\Linf$-operads $\rmQ_{0,0} \to \rmP^{\ell_1^\alpha}$.
\end{prop}

\subsection*{A Koszul predual adjunction}
In \cref{Koszul preduals}, we define categories $\coCurv$ and $\coTwist$, whose objects are cooperads under~$\uCom^{*}$, the linear dual of the operad $\uCom$ encoding unital commutative algebras, with extra structure, and we show that the cobar construction induces functors $\coCurv \to \Curv$ and $\coTwist \to \Twist$.
We prove that the forgetful functor $\coCurv \to \coCom$ admits a right adjoint and give an explicit description for the unit of this adjunction.
This allows us to give an explicit formula for the units of the adjunctions $\Curv \leftrightarrow \cMult$ (\cref{explicit formula unit intro}) and $\Twist \leftrightarrow \Mult$ (\cref{explicit formula unit Twist intro}) on objects which come from the cobar construction.

We also explain how the extendable operads of Dotsenko, Shadrin, and Vallette \cite[Sec.~4.5]{DotsenkoShadrinVallette23} fit in our framework.
More precisely, we show that any unital extension $\rmP \hookrightarrow u_\chi \rmP$ of a quadratic operad $\rmP$ defines a lift of $(u_\chi \rmP)^*$ to $\coCurv$ and $\coTwist$ (\cref{dual of Ext is in coCurv intro}).
When~$\rmP$ is moreover Koszul, we explain that any unital extension $\rmP^! \hookrightarrow u_\chi \rmP^!$ of its Koszul dual induces 
\begin{enumerate}
    \item an explicit $\cTw$-coalgebra structure on $\cPinf^\chi \eqdef \Omega(u_\chi \rmP^!)^*$ (\cref{DSV formula intro}) related to the gauge group action on the Lie algebra $\hom_\Sym((u_\chi \rmP^!)^*, \End_A)$, and
    \item an explicit $\Tw$-coalgebra structure on $\Pinf \eqdef \Omega \overline{(\rmP^!)^*}$ (\cref{extendable operads are Tw-coalgebras intro}).
\end{enumerate}
In particular, using the terminology of \cite{DotsenkoShadrinVallette23}, our results imply that $\Pinf$ admits a $\Tw$-coalgebra structure if the category of $\Pinf$-algebras is twistable.

In fact, we show in \cref{sec:group-action}, using the Yoneda lemma, that given an operad $\rmQ \in \cMult$ and a complete filtered vector space~$A$, a $\cTw$-coalgebra structure on $\rmQ$ is equivalent to an action of $\calF_1 A$ on $\hom_{\Op}(\rmQ,\End_A)$. 
In the case where $\rmQ=\cPinf^\chi$, we moreover show in \cref{rem:DSV-formula} that this action coincides with the Dotsenko--Shadrin--Vallette gauge group action~\cite[Section~4.5]{DotsenkoShadrinVallette23}.

\subsection*{Homotopy fixed points}
We turn to the problem of determining homotopy fixed points of $\Tw$, i.e.\ objects $\rmP$ in $\Mult$ such that the counit $\Tw(\rmP) \to \rmP$ is a quasi-isomorphism. 
This problem was first studied by Dolgushev and Willwacher \cite{DolgushevWillwacher15}, in order to 
establish the existence of solutions to the Deligne conjecture compatible with the twisting procedure. 

Here we approach the question beginning with the observation that the corresponding problem for $\cTw$ is trivial:   every operad under $\cLinf$ is a homotopy fixed point for~$\cTw$ (\cref{lemma:homology-coproduct}). 
By elaborating on this fact and carefully studying the monad for the adjunction $\Twist \leftrightarrow \Mult$ -- see in particular \cref{thm:monad} -- we arrive at a new sufficient condition for an object to be a homotopy fixed point of $\Tw$.  

\begin{thm}
\label{thm:homotopy-fixed-point intro}
    Let $\rmP$ be an object of $\Mult$. 
    Assume that 
    \begin{enumerate}
        \item\label{cond:Tw-coalgebra} $\rmP$ admits a $\Tw$-coalgebra structure,
        \item\label{cond:zero-divisor} the Lie bracket $\ell_2$ is not a right zero divisor in $H_*(\rmP)$, i.e.\ $d_\rmP \nu = 0$ and $\nu \circ_1 \ell_2 \in \im(d_\rmP)$ imply $\nu \in \im(d_\rmP)$.
    \end{enumerate}
    Then $\rmP$ is a homotopy fixed point of $\Tw$.
\end{thm}

We apply this result to prove that any object of $\Mult$ obtained from an operad and $\Lie$ using a distributive law is a homotopy fixed point of $\Tw$ (\cref{fixed-point-distributive-law}), thus recovering \cite[Theorem 5.10]{DolgushevWillwacher15} (without the degree boundedness assumption).

\begin{remark}
    Neither condition (\ref{cond:Tw-coalgebra}) nor $(\ref{cond:zero-divisor})$ alone is sufficient.
    In \cref{examples}, we show that the permutative operad admits a $\Tw$-coalgebra structure but is not a homotopy fixed point for~$\Tw$. 
    On the other hand we show that the Lie bracket is not a right zero divisor in the pre-Lie operad, which is known not to be a homotopy fixed point for $\Tw$  \cite[Thm.~5.1]{DotsenkoKoroshkin24}. 
\end{remark}

\subsection*{Examples}
In \cref{examples}, we apply our results to systematically recover known $\Tw$-coalgebra structures and homotopy fixed points for~$\Tw$. 


\subsection*{Acknowledgements}
We would like to thank Alexander Berglund, Coline Emprin, Alyosha Latyntsev, Hugo Pourcelot, Alex Takeda, and Bruno Vallette for useful discussions.
The work of the authors is supported by Novo Nordisk Foundation grant NNF20OC0066298, Villum Fonden Villum Investigator grant 37814, and Danish National Research Foundation grant DNRF157. 
GLA was supported by the Andrew Sisson Fund and the Australian Research Council Future Fellowship FT210100256.



\subsection*{Conventions}

Unless explicitly stated, the word ``operad'', respectively ``cooperad'', always means ``non-unital symmetric operad'', respectively ``non-counital symmetric cooperad'', in chain complexes over a field of characteristic $0$.
We always assume that our cooperads admit a complete decreasing filtration so that we can apply the filtered cobar construction~$\Omega$ to them, and we denote by $\iota : \rmC \to \Omega \rmC$ the natural degree $(-1)$ map.
We use shifted homological degree conventions: for a Koszul operad $\rmP$ concentrated in degree zero, the generators in the minimal resolutions~$\Omega(\rmP^!)^{*}$ all have degree $-1$.
One can recover usual degree and signs conventions by de-suspension, see e.g.\ \cite[Sec.~4.1]{DotsenkoShadrinVallette23} for formulas.
In particular, we work with the shifted \defn{curved L-infinity dg operad} $\cLinf$, which is the free operad on generators $\ell_n$, $n\geq 0$, of arity $n$ and degree $(-1)$, endowed with the differential
\begin{equation}\label{clinf differential}
    d_{\cLinf}(\ell_n) \eqdef -\sum_{\substack{p+q=n \\ p,q \geqslant 0}} \sum_{\sigma \in \Sh_{p,q}^{-1}}(\ell_{p+1} \circ_{1} \ell_q)^\sigma \ ,
\end{equation}
where $\Sh_{p,q}^{-1} \subset \Sym_{p+q}$ denotes the set of $(p,q)$-unshuffles.
The symmetric group action on the generators is given by $\ell_n^\sigma=\ell_n$ for any $\sigma \in \mathbb{S}_n$.


\subsection*{Plan of the paper}
In \cref{sec:cTw-comonadicity} we define the comonad $\cTw$ and characterize it via its associated comonadic adjunction.
In \cref{section:comparison} we relate the categories $\Curv$ and $\Twist$ by an adjunction, which allows us to give an analogous characterization of the comonad $\Tw$ in \cref{adjunction twist}.
Symmetrically, we characterize the Maurer--Cartan and the uncurving procedure in \cref{section:monad}.
We give in \cref{homotopy-fixed-points} a criterion to test whether an operad is a homotopy fixed point for $\Tw$.
In \cref{Koszul preduals} we give a Koszul predual adjunction, and an explicit formula for $\Tw$-coalgebras whose underlying operad is in the image of the cobar functor.
We apply the preceding results to identify known $\Tw$-coalgebras and homotopy fixed points in \cref{examples}.

In \cref{sec:reconstruction} we give the equivalence between the present definitions and the ones from our previous work~\cite{Laplante-Petr-Shende2025}.
We explain in \cref{section distributivity} how to recover the~$\Tw$ comonad from a distributive law between~$\cTw$ and another comonad, using a monad from \cref{section:comparison}.
In \cref{bar-cobar} we show how to lift the bar construction to~$\coCurv$. 
In \cref{sec:group-action} we show that $\cTw$-coalgebra structures on a $\cLinf$-operad are equivalent to monoid actions. 
Finally in \cref{sec:twistable-Tw-stable} we show that the gauge group action of \cite{DotsenkoShadrinVallette23} gives $\cTw$-coalgebra structures.


\tableofcontents


\section{Tw,  cTw, and the proof of Theorem \ref{cTw comonadicity}}
\label{sec:cTw-comonadicity}

We recall Willwacher's twisting comonad~$\Tw$ \cite{Willwacher-KGT}, in the later formulation of  \cite[Sec.~5.3]{DotsenkoShadrinVallette23}. 
Given an operad~$\rmP$ in~$\Mult$, define the operad $\MC(\rmP) \eqdef \rmP \, \hat \vee \, [\alpha]$, where $\hat \vee$ denotes the completed coproduct with respect to the $\alpha$-filtration, with differential $d\alpha\eqdef -\ell_0^\alpha$.
The natural operad morphism $\Linf \to \MC(\rmP)$ sending $\ell_i \mapsto \ell_i^\alpha$ intertwines $d_{\Linf}$ not with $d_{\MC(\rmP)}$ but instead with 
$d_{\MC(\rmP)} + [\ell_1^{\alpha}, \cdot]$. We write $\Tw(\rmP)$ for $\MC(\rmP)$ with this new differential. 

\begin{remark}
    The fact that $d_{\MC(\rmP)} + [\ell_1^{\alpha}, \cdot]$ squares to zero amounts to $d \ell_1^{\alpha} + \ell_1^{\alpha} \circ \ell_1^{\alpha} = 0$.
    This situation is described in the literature by saying that $\ell_1^\alpha$ is an operadic Maurer-Cartan element.
    Further, the notation $\MC(\rmP)^{\ell_1^\alpha}$ is used for the operad $\MC(\rmP)$ equipped with the new differential $d_{\MC(\rmP)} + [\ell_1^{\alpha}, \cdot]$.  That is, by definition,
    $\Tw(\rmP) := \MC(\rmP)^{\ell_1^\alpha}$. 
\end{remark}

$\Tw$ is an endofunctor on $\MultSym$, acting on morphisms by $f \mapsto f \vee 1_\alpha$. 
The comonad structure is given by the natural transformations
 \[
    \begin{array}{ccccccc}
    \Tw(\rmP) & \xrightarrow{\Delta_\rmP} & \Tw\Tw(\rmP), & & \Tw(\rmP) & \xrightarrow{\varepsilon_\rmP} & \rmP \\
    \nu & \mapsto & \nu & & \nu & \mapsto & \nu \\
    \alpha & \mapsto & \alpha' + \alpha & & \alpha & \mapsto & 0. \\
    \end{array}
\]
for $\nu \in \rmP$, and $\alpha'$ the copy of $\alpha$ adjoined by the second instance of $\Tw$. 

We now introduce the endofunctor $\cTw$ on $\cMult$.  

\begin{definition}
\label{def of cTw}
    Given an operad $\cLinf \to \rmQ$, define an operad $\cTw(\rmQ) := \rmQ  \vee [\kurv] \, \hat \vee \, [\alpha] $ with differential acting as $d_\rmQ$ on $\rmQ$ and by $d \alpha = \kurv$.
    There is a morphism of operads $\cLinf \to \cTw(\rmQ)$ sending $\ell_0$ to $(\kurv + \ell_0^\alpha)$ and $\ell_n$ to $\ell_n^\alpha$ for $n \geq 1$.
    The action of $\cTw$ on morphisms is given by~$f \mapsto f \vee 1_{[\kurv] \vee [\alpha]}$.     
\end{definition}

There is a comonad structure given by the natural transformations: 
\[
\begin{array}{ccccccc}
\cTw(\rmQ) & \to & \cTw(\cTw(\rmQ)), & & \cTw(\rmQ) & \to & \rmQ \\
\nu & \mapsto & \nu & & \nu & \mapsto & \nu \\
\alpha & \mapsto & \alpha' + \alpha & & \alpha & \mapsto & 0 \\
\kurv & \mapsto & \kurv' + \kurv & & \kurv & \mapsto & 0. \\
\end{array}
\]
for $\nu \in \rmQ$, and $\alpha', \kurv'$ the copies of $\alpha, \kurv$ adjoined by the second instance of $\cTw$.

A $\rmQ$-algebra $(A, d)$ and element $a \in A$ determine a diagram
   \begin{center}
    \begin{tikzcd}
    \cLinf \arrow[r] \arrow[rd] \arrow[rrd, dashed] & \rmQ \arrow[r, "\rho"]  & \End_{(A, d)} \\
    \, & \cTw(\rmQ) \arrow[r, dashed] \arrow[u] & \End_{(A, d)}
    \end{tikzcd}
    \end{center}
where the map $\cTw(\rmQ) \to \End_{(A, d)}$  is given by $\rho$ on $\rmQ$ and sends $\alpha \mapsto a$. 

Unlike for $\Tw$, there is an obvious operad morphism $\rmQ \to \cTw(\rmQ)$ giving a section of the projection, but it is not generally a morphism of $\cLinf$-operads.  
Such a $\cLinf$ section is by definition a curved operadic twisting for $\rmQ$.
In fact we have the following result as a direct consequence of the construction of $\cTw$.

\begin{prop}
\label{curved operadic twisting}
    Let $\rmQ$ be a $\cLinf$-operad with a $\cTw$-coalgebra structure $\rmQ \to \cTw(\rmQ)$.
    Let $(A,d)$ be a complete filtered chain complex with a $\rmQ$-algebra structure $\rmQ \to \End_{(A,d)}$ and corresponding $\cLinf$-operations $\ell_n$.
    Then any element $a \in \calF_1 A$ gives rise to a new $\rmQ$-algebra structure on $A$ with corresponding $\cLinf$-operations $\ell_n^a$.
    This procedure is moreover linear in~$a \in \calF_1 A$.
\end{prop}

\begin{remark}
    The fact that $\cTw$ is structurally simpler than $\Tw$ corresponds to the fact that there is no need to impose the Maurer-Cartan equation or twist the differential.
\end{remark}

Before turning to the proof of \cref{cTw comonadicity}, we recall the second main result of our previous paper \cite{Laplante-Petr-Shende2025}.

\begin{theorem}[{\cite[Theorem~1.8~\&~Proposition~4.4]{Laplante-Petr-Shende2025}}]
\label{thm:adjunction-Curv}
    The forgetful functor $\calL: \Curv \to \cMult$ admits a right adjoint $\calR$ given on objects by $\rmQ \mapsto \cTw(\rmQ)$ with suboperad $\cTw(\rmQ)_0 \eqdef \rmQ \, \hat \vee \, [\alpha]$.
    The counit $\varepsilon^{\calL \calR}_{\rmQ} : \rmQ \vee [\kurv] \, \hat \vee \, [\alpha] \to \rmQ$ is the identity on $\rmQ$ and sends $\alpha, \kappa$ to~$0$. 
\end{theorem}

\subsection{Proof of Theorem \ref{cTw comonadicity}}
Theorem \ref{thm:adjunction-Curv} already asserts that 
for the adjunction $(\calL \dashv \calR)$, we have $\calL \calR = \cTw$ as endofunctors.

Now we check that the comonad structures of $\calL \calR$ and $\cTw$ also agree. 
According to \cref{thm:adjunction-Curv}, we know that the counit $\calL \calR \to 1$ agrees with the map $\cTw \to 1$.
Therefore it remains to check that, for an object $\calR(\rmQ) \in \Curv$, the unit $\eta : \calR(\rmQ) \to (\calR \calL)(\calR(\rmQ))$ is given by
\[\eta(\nu) = \nu, \quad \eta(\alpha) = \alpha' + \alpha, \quad \eta(\kurv) = \kurv' + \kurv. \]
It is straightforward to check that the latter formula defines a morphism of dg operads, therefore it remains to check that it preserves the distinguished element.
In other words, writing $\lambda_0 := \kurv + \ell_0^\alpha$ and $\lambda_n := \ell_n^\alpha$ for the $\cLinf$-operations on $\calL \calR (\rmQ)$,
we need to check that $\eta(\kurv+ \ell_0^{\alpha}) = \kurv' + \lambda_0^{\alpha'}$.
We compute 
\begin{align*}
    \kurv' + \lambda_0^{\alpha'} = & \kurv' + \kurv + \sum_{k \geq 0} \frac{1}{k!} \ell_k^{\alpha} ((\alpha')^k) \\
    = & \kurv' + \kurv + \sum_{k \geq 0} \sum_{r \geq 0} \frac{1}{k!} \frac{1}{r!} \ell_{k+r} (\alpha^r, (\alpha')^k) \\ 
    = & \kurv' + \kurv + \sum_{n \geq 0} \frac{1}{n!} \ell_n ( (\alpha + \alpha')^n) = \kurv' + \kurv + \ell_0^{\alpha+\alpha'} = \eta(\kurv+ \ell_0^{\alpha}). 
\end{align*}

Finally, we check that the adjunction is comonadic.
According to the comonadicity theorem (see \cite[Sec.~3.5]{barr2000toposes} and \cite{Beck25}), it is enough to check that 
\begin{enumerate}
    \item \label{item:conservative2} $\calL$ reflects isomorphisms,
    \item \label{item:equalizers} $\Curv$ has equalizers, and they are preserved by $\calL$. 
\end{enumerate}
Condition~(\ref{item:conservative2}) is obvious.  Regarding~(\ref{item:equalizers}), recall that equalizers in the category of operads are computed by the arity-wise equalizers of the underlying vector space of operations; the same evidently holds for $\cMultSym$. 
It will now suffice to check that, in $\Curv$, we have
$$\rm{Eq}((\rmQ, \rmQ_0) \substack{\rightarrow\\[-1em] \rightarrow} (\rmQ', \rmQ_0')) \stackrel{?}{=} 
(\rm{Eq}(\rmQ \substack{\rightarrow\\[-1em] \rightarrow} \rmQ'), \rm{Eq}(\rmQ_0 \substack{\rightarrow\\[-1em] \rightarrow} \rmQ'_0)).$$
It is obvious that the right hand side will be the equalizer if it is in $\Curv$. It remains to check that as (not dg) operads, 
$$\rm{Eq}((\rmQ_0 \vee [\kappa]) \substack{\rightarrow\\[-1em] \rightarrow} (\rmQ'_0 \vee [\kappa])) = \rm{Eq}(\rmQ \substack{\rightarrow\\[-1em] \rightarrow} \rmQ') \stackrel{?}{=} \rm{Eq}(\rmQ_0 \substack{\rightarrow\\[-1em] \rightarrow} \rmQ'_0) \vee [\kappa]$$
Now there is an obvious injective map from the right side to the left; surjectivity follows from freeness in $\kappa$.  $\qed$

\subsection{Nonsymmetric case}
Recall that the shifted \defn{curved A-infinity operad}, denoted~$\cAinf$, is the free operad on generators $\mu_n$, $n\geqslant 0$, of arity $n$ and degree $(-1)$, endowed with the differential
$d_{\cAinf}(\mu_n) \eqdef - \sum_{p+q+r=n} \mu_{p+1+r} \circ_{p+1} \mu_q$.
Analogous constructions can be made for non-symmetric operads by replacing $\cLinf$ by $\cAinf$ and the operations~$\ell_n^\alpha$~\eqref{eq:ell-n-alpha} by
\begin{equation}
\label{mu-n-alpha}
    \mu_n^\alpha \eqdef \sum_{r_0,\ldots,r_n \geqslant 0} \mu_{n+r_0+\cdots+r_n} (\alpha^{r_0},-,\alpha^{r_1},-,\ldots,\alpha^{r_n}) \in \rmP \, \hat{\vee} \, [\alpha].
\end{equation}
The corresponding results hold with the same proofs.


\section{Comparison between $\Twist$ and $\Curv$}
\label{section:comparison}

Here we study the relationship between the categories $\Twist$ and  $\Curv$.
We first introduce an endofunctor of $\cMult$ following a standard construction, see for instance \cite[Section~2.3]{Merkulov23} or \cite[Section~5.2]{ChuangLazarev13} where it is called the ``hat construction''.

\begin{def-prop}
\label{def:+} 
    Given $\rmQ \in \cMult$, we denote by $\rmQ^+$ the object of $\cMult$ defined as follows
    \begin{itemize}
        \item the underlying operad is the coproduct of $\rmQ$ and the free operad $[\diff]$ generated by a single operation $\diff$ of arity $1$ and degree $(-1)$, 
        \item the differential is given by $d_{\rmQ^+} \diff \eqdef - \diff \circ_1 \diff$ and $d_{\rmQ^+} \nu \eqdef d_\rmQ \nu - [\diff, \nu]$ for $\nu \in \rmQ$,
        \item the map $\cLinf \to \rmQ^+$ sends $\ell_1 \mapsto (\diff + \ell_1)$ and $\ell_n \mapsto \ell_n$ for $n \ne 1$.
    \end{itemize}
    The assignment $\rmQ \mapsto \rmQ^+$ defines an endofunctor of $\cMult$ acting on morphisms as $\psi \mapsto \psi \vee 1_{[\diff]}$.
\end{def-prop}
\begin{proof}
    To check that the map $f: \cLinf \to \rmQ^+$ is a dg morphism, we note that $(d_{\cLinf} \ell_n + [\ell_1, \ell_n])$ does not contain any $\ell_1$ for $n \ne 1$.
    Therefore for $n \ne 1$ we compute
    \begin{align*}
        f(d_{\cLinf} \ell_n) & = f(d_{\cLinf} \ell_n + [\ell_1, \ell_n]) - f([\ell_1, \ell_n]) \\
        & = d_\rmQ \ell_n + [\ell_1, \ell_n] - [\diff + \ell_1, \ell_n] \\
        & = d_\rmQ \ell_n - [\diff, \ell_n] = d_{\rmQ^+} \ell_n.
    \end{align*}
    For $n=1$ we have 
    \begin{align*}
        -f(d_{\cLinf} \ell_1) & = f(\ell_1 \circ_1 \ell_1 + \ell_2 \circ_1 \ell_0) \\
        & = (\diff + \ell_1) \circ_1 (\diff + \ell_1) + \ell_2 \circ_1 \ell_0 \\
        & = \diff \circ_1 \diff + \ell_1 \circ_1 \ell_1 + \ell_2 \circ_1 \ell_0 + [\diff, \ell_1] = -d_{\rmQ^+}(\diff + \ell_1).
    \end{align*}
\end{proof}

\begin{definition}
\label{definition:functor-Twist-Curv}
    We denote by $\calI : \Twist \to \Curv$ the functor that acts on objects by 
    $(\rmQ, \rmQ_{0,0}) \mapsto (\rmQ, \rmQ_0 := \rmQ_{0,0}^+)$, where $\rmQ_{0,0}^+ \subset \rmQ$ via $m \mapsto \ell_1$, and acts as `the identity' on morphisms.
\end{definition}

\begin{def-prop}
\label{def-prop:functor-Curv-Twist}
    There is a functor $\calJ : \Curv \to \Twist$ acting as 
    \[
    (\rmQ, \rmQ_0) \mapsto (\rmQ^+, \rmQ^+_{0,0} := \rmQ_0).
    \] 
\end{def-prop}

\begin{proof}
    We check that $\calJ(\rmQ,\rmQ_0)$ satisfies the conditions in \cref{twist new}.
    The map $\rmQ_0 \vee [\ell_0] \vee [\ell_1] \to \rmQ \vee [\diff]$ (sending $\ell_1 \mapsto (\diff + \ell_1)$) is an isomorphism by freeness of $\diff$, and since $\rmQ_0 \vee [\ell_0] \to \rmQ$ is an isomorphism according to the definition of $\Curv$.

    Moreover, the image of $\Linf$ under $f : \cLinf \to \rmQ^+$ is contained in $\rmQ^+_{0,0} = \rmQ_0$ since the map $\cLinf \to \rmQ$ sends $\Linf$ to $\rmQ_0$ according to the definition of $\Curv$.

    Finally, for $\nu \in \rmQ^+_{0,0} = \rmQ_0$ we have $d_{\rmQ^+} \nu + [\diff + \ell_1, \nu] = d_{\rmQ} \nu + [\ell_1, \nu]$. The latter is of degree $\leq 0$ in $(\diff + \ell_1)$ and of degree at most $1$ in $\ell_0$ according to the definition of $\Curv$.
    This finishes the proof.
\end{proof}

\begin{thm}
\label{thm:adjunction-Twist-Curv}
    The two functors $\calI$ and $\calJ$ form an adjoint pair with $\calI$ left adjoint to $\calJ$.
    The unit $\eta : (\rmQ, \rmQ_{0,0}) \to (\rmQ^+, \rmQ_{0,0}^+)$ is given by 
    \begin{equation}
    \label{eq:IJ-unit}
        \eta_{|\rmQ_{0,0}} = \id_{\rmQ_{0,0}}, \quad \eta(\ell_0) = \ell_0, \quad \eta(\ell_1) = \diff + \ell_1. 
    \end{equation}
    The counit $\varepsilon : (\rmQ^+, \rmQ_0^+) \to (\rmQ, \rmQ_0)$ is given by 
    \begin{equation}
    \label{eq:IJ-counit}
    \varepsilon_{|\rmQ} = \id_{\rmQ}, \quad \varepsilon(\diff) = 0. 
    \end{equation}
\end{thm}
\begin{proof}  
From the \cref{def-prop:functor-Curv-Twist}, we see that $\eta$ commutes with the morphisms $\cLinf \to \rmQ$ and $\cLinf \to \rmQ^+$. In particular, $\eta$ commutes with the differentials on $\ell_0$ and $\ell_1$.
We now check that $\eta$ commutes with the differentials on any $\nu \in \rmQ_{0,0}$. By \cref{twist new}, $d_{\rmQ} \nu + [\ell_1, \nu]$ does not contain any $\ell_1$, therefore
\begin{align*}
    \eta(d_\rmQ \nu) & = \eta(d_{\rmQ} \nu + [\ell_1, \nu]) - \eta([\ell_1, \nu]) \\
    & = d_\rmQ \nu + [\ell_1, \nu] - [\diff + \ell_1, \nu] \\
    & = d_\rmQ \nu - [\diff, \nu] = d_{\rmQ^+} \nu.
\end{align*}
We conclude that $\eta$ commutes with the differentials. 
Moreover, given that that both~$\calI$ and~$\calJ$ act trivially on morphisms, it is clear that $\eta$ is a well-defined natural transformation.

Given that $d_{\rmQ^+}(\diff)=-\diff \circ_1 \diff$ and $d_{\rmQ^+}(\nu)=d_\rmQ(\nu)-[\diff,\nu]$ for $\nu \in \rmQ$, it is immediate to check that $\varepsilon$ commutes with the differentials.
Here again, it is clear that the condition with respect to morphisms in~$\Curv$ is fulfilled, therefore $\varepsilon$ is a well-defined natural transformation.

It remains to check that the two unit-counit equations 
\[
\id_\calI = (\varepsilon \calI)(\calI \eta)
\quad \text{and} \quad
\id_\calJ = (\calJ\varepsilon)(\eta\calJ)
\]
are satisfied. 
The object $\calI \calJ \calI(\rmQ,\rmQ_{0,0})$ in $\Curv$ is
\[\calI \calJ \calI(\rmQ,\rmQ_{0,0}) = (\rmQ^+,\rmQ_{0,0}^{++}).\]
Using the formulas for the unit~\eqref{eq:IJ-unit} and for the counit~\eqref{eq:IJ-counit} it is straightforward to check that the composite $(\varepsilon \calI)(\calI \eta)$ is the identity. 
Similarly, the object $\calJ\calI\calJ(\rmQ,\rmQ_0)$ in $\Twist$ is 
\[\calJ\calI\calJ(\rmQ,\rmQ_0) =(\rmQ^{++}, \rmQ_0^+),\]
and checking that $(\calJ\varepsilon)(\eta\calJ)$ is the identity is once again a direct verification.
\end{proof}

\begin{remark}
\label{monad IJ}
    Observe that the monad $\calI \calJ$ on $\Curv$ is given by $(\rmQ, \rmQ_0) \mapsto (\rmQ^+, \rmQ_0^+)$.
    We will elaborate on this fact in \cref{section distributivity}.
\end{remark}

The same results (with the same proofs) hold for the corresponding categories of non-symmetric operads after replacing the operations $\ell_n$ in $\cLinf$ by the operations $\mu_n$ in $\cAinf$. 


\section{Adjunction for $\Tw$ and the proof of \cref{thm:main-thm}}
\label{adjunction twist}

Consider the quotient map $\pi: \cLinf \to \cLinf/(\ell_0,\ell_1) = \Linf$.  
There is an adjunction $\pi_!: \cMult \leftrightarrow \Mult: \pi^!$
given by
\begin{eqnarray*}
    \pi_!: (\cLinf \to \rmP) & \mapsto & (\Linf \to \rmP/(\ell_1,\ell_0)), \quad \text{and} \\
    \pi^! : (\Linf \to \rmP) &\mapsto& (\cLinf \xrightarrow{\pi} \Linf \to \rmP).
\end{eqnarray*}
Combining it with the adjunctions $(\calL \dashv \calR)$ of \cref{thm:adjunction-Curv} and $(\calI \dashv \calJ)$ of \cref{thm:adjunction-Twist-Curv}, we get the following three adjunctions:
\[ 
\begin{tikzcd}
    \Twist \ar[r, bend left, "\calI"] & \Curv \ar[r, bend left, "\calL"] \ar[l, bend left, "\calJ"] & \cMult \ar[l, bend left, "\calR"] \ar[r, bend left, "\pi_!"] & \Mult \ar[l, bend left, "\pi^!"].
\end{tikzcd}
\]

\begin{theorem}
\label{thm:adjoint}
    The forgetful functor 
    \[\forget : \Twist \to \Mult, \quad (\rmQ, \rmQ_{0,0}) \mapsto \rmQ_{0,0} \]
    has for right adjoint the functor $\free := \calJ \calR \pi^! : \rmP \mapsto (\cTw(\pi^! \rmP)^+, \Tw(\rmP))$.
    The counit is
    $1_{\rmP} \vee 0 : \Tw(\rmP) = (\rmP \, \hat{\vee} \, [\alpha], d_{\Tw(\rmP)}) \to (\rmP, d_\rmP)$.
\end{theorem}

\begin{proof}
    Using the adjunctions $(\pi_!, \pi^!)$, $(\calL, \calR)$ from \cite[Theorem 1.7]{Laplante-Petr-Shende2025}, and $(\calI, \calJ)$ from \cref{thm:adjunction-Twist-Curv}, we know that $\calJ \calR \pi^!$ is the right adjoint of $\pi_! \calL \calI : (\rmQ, \rmQ_{0,0}) \mapsto \rmQ / (\ell_0, \ell_1)$, which is naturally isomorphic to $\forget$.
    Moreover, we have $(\calJ \calR \pi^!)(\rmP) = (\cTw(\pi^! \rmP)^+, \rmP \, \hat{\vee} \, [\alpha])$.
    
    We now show that the induced differential $d$ on $\free(\rmP)_{0,0} = \rmP \, \hat{\vee} \, [\alpha]$ is the same as $d_{\Tw}$, i.e.
    \[\left\{
    \begin{array}{llll}
    d \nu & = & d_\rmP \nu + [\ell_1^{\alpha}, \nu] & \text{for } \nu \in \rmP \\
    d \alpha  & = & - \ell_0^{\alpha}+[\ell_1^{\alpha},\alpha]. & \\
    \end{array}
    \right.\]
    By the definitions of $\cTw$ (\cref{def of cTw}) and the symbol $(-)^+$ (\cref{def:+}), the differential $D$ on $\cTw(\pi^! \rmP)^+$ satisfies 
    \begin{enumerate}
        \item $D \nu = d_\rmP \nu - [\diff, \nu] = d_\rmP \nu + [\ell_1^{\alpha}, \nu] - [\diff + \ell_1^\alpha, \nu]$ for $\nu \in \rmP$,
        \item $D \alpha = - [\diff, \alpha] + \kurv = - \ell_0^{\alpha} + [\ell_1^{\alpha},\alpha] - [\diff + \ell_1^\alpha, \alpha] + (\kurv+ \ell_0^\alpha)$.
    \end{enumerate}
    Therefore $d = d_\Tw$.

    It remains to show that the morphism $(\cLinf, \Linf) \to \calG(\rmP)$ induces the map $\Linf \to \Tw(\rmP)$ given by $\ell_n \mapsto \ell_n^\alpha$. 
    By definition of $\cTw$ (\cref{def of cTw}), the morphism $\cLinf \to (\calR \pi^!)(\rmP) = \cTw(\pi^!\rmP)$ sends $\ell_0 \mapsto (\kurv+\ell_0^\alpha)$ and $\ell_n \mapsto \ell_n^\alpha$ for $n\geq 1$.
    The result follows by applying the functor $\calJ$ of \cref{def-prop:functor-Curv-Twist}.
\end{proof}

Any adjunction gives rise to a monad on the source category, and a comonad on the target category. 
We unravel the comonad structure on the category $\Mult$ associated to the adjunction $\forget \dashv \free$ between $\Twist$ and $\Mult$ of \cref{thm:adjoint}.

\begin{prop}
\label{prop:comonad-structure} 
    If we write $(\rmP \, \hat{\vee} \, [\alpha]) \, \hat{\vee} \, [\alpha']$ for the underlying operad of $(\forget \free \forget \free)(\rmP)$, then the comonad structure of the endofunctor $\forget \free$ is given by the maps
    \[
    \begin{array}{ccccccc}
    \forget \free & \xrightarrow{\forget(\eta_{\free})} & \forget \free \forget \free, & & \forget \free & \xrightarrow{\varepsilon} & \id \\
    \nu & \mapsto & \nu & & \nu & \mapsto & \nu \\
    \alpha & \mapsto & \alpha' + \alpha & & \alpha & \mapsto & 0. 
    \end{array}
    \]
\end{prop}

\begin{proof}
    Let $\rmP \in \Mult$.
    According to \cref{thm:adjoint}, we already know the formula for the counit $\varepsilon : \forget \free \to \id$.
    Therefore it suffices to show that for the object $\free(\rmP)$, the unit $\eta : \free(\rmP) \to (\free \forget) (\free(\rmP))$ is given by 
    \[
    \eta(\nu) = \nu, \quad \eta(\alpha) = \alpha' + \alpha, \quad \eta(\kurv) = \kurv'-(\ell_0^\alpha+[\ell_1^\alpha, \alpha']), \quad \eta(\diff) = \diff'-\ell_1^\alpha. \]
    It is straightforward to check that the latter formula defines a morphism of operads, therefore it remains to check that it preserves the distinguished elements.
    In other words, writing $\lambda_n := \ell_n^\alpha$ for the $\Linf$-operations on $\forget \free (\rmP, d_\rmP)$,
    we need to check that $\eta(\diff+\ell_1^{\alpha}) = \diff' + \lambda_1^{\alpha'}$ and $\eta(\kurv+ \ell_0^{\alpha}) = \kurv' + \lambda_0^{\alpha'}$.
    The equality $\eta(\diff+\ell_1^{\alpha}) = \diff' + \lambda_1^{\alpha'}$ follows from the identity $\ell_1^{\alpha+\alpha'}=\ell_1^\alpha + \lambda_1^{\alpha'}$.  
    An argument for this can be found in e.g.\ \cite[Lemma~5.24]{DotsenkoShadrinVallette23}; we reproduce it in our context for completeness
    \begin{align*}
        \ell_1^{\alpha} + \lambda_1^{\alpha'} = & \sum_{n \geq 1} \frac{1}{n!} \ell_{n+1} (\alpha^n, -,) + \sum_{k \geq 1} \frac{1}{k!} \ell_{k+1}^{\alpha} ((\alpha')^k, -) \\
        = & \sum_{n \geq 1} \frac{1}{n!} \ell_{n+1} (\alpha^n, -) + \sum_{k \geq 1} \sum_{r \geq 0} \frac{1}{k!} \frac{1}{r!} \ell_{k+r+1} (\alpha^r, (\alpha')^k,-) \\ 
        = & \sum_{n \geq 1} \frac{1}{n!} \ell_{n+1} ( (\alpha + \alpha')^k, -) = \ell_1^{\alpha+\alpha'}.
    \end{align*}
    Similarly, the equality $\eta(\kurv+ \ell_0^{\alpha}) = \kurv' + \lambda_0^{\alpha'}$ follows from the identity $\ell_0^{\alpha+\alpha'} = \ell_0^\alpha + \lambda_0^{\alpha'} + [\ell_1^\alpha, \alpha']$ which is proved as follows
    \begin{align*}
        \ell_0^\alpha + \lambda_0^{\alpha'} + [\ell_1^\alpha, \alpha'] = & \sum_{n \geq 2} \frac{1}{n!} \ell_n (\alpha^n) + \sum_{k \geq 2} \frac{1}{k!} \ell_k^{\alpha} ((\alpha')^k) + \sum_{n \geq 1} \frac{1}{n!} \ell_{n+1} (\alpha^n, \alpha') \\
        = & \sum_{n \geq 2} \frac{1}{n!} \ell_n (\alpha^n) + \sum_{k \geq 2} \sum_{r \geq 0} \frac{1}{k!} \frac{1}{r!} \ell_{k+r} (\alpha^r, (\alpha')^k) + \sum_{n \geq 1} \frac{1}{n!} \ell_{n+1} (\alpha^n, \alpha') \\ 
        = & \sum_{n \geq 2} \frac{1}{n!} \ell_n ( (\alpha + \alpha')^n) = \ell_0^{\alpha+\alpha'}.
    \end{align*}
\end{proof}

We are now in position to identify $\Twist$ with the category of $\Tw$-coalgebras.

\begin{proof}
[Proof of \cref{thm:main-thm}]
    We first prove that $\forget \free = \Tw$ as comonads.
    According to \cref{thm:adjoint}, we already know that $\forget \free$ and $\Tw$ agree on objects.
    Moreover, both $\forget \free$ and $\Tw$ act on morphisms as $f \mapsto f \vee 1$.
    Finally, according to \cref{prop:comonad-structure}, the comonad structures of $\forget \free$ and $\Tw$ also agree.

    We now check that the adjunction is comonadic.
    According to the comonadicity theorem (see \cite[Sec.~3.5]{barr2000toposes}), it is enough to check that 
    \begin{enumerate}
        \item \label{item:conservative} $\forget$ reflects isomorphisms,
        \item \label{item:equalizers2} $\Twist$ has equalizers, and they are preserved by $\forget$. 
    \end{enumerate}
    
    We first check (\ref{item:conservative}). Consider a morphism $f : (\rmQ, \rmQ_{0,0}) \to (\rmQ', \rmQ'_{0,0})$ in $\Twist$ such that $\forget f : \rmQ_{0,0} \to \rmQ'_{0,0}$ is an isomorphism.
    Because $f$ is a morphism in $\Twist$, it preserves $\ell_0$ and $\ell_1$.
    We conclude from the fact that $\rmQ_{0,0} \vee [\ell_0] \vee [\ell_1] \to \rmQ$ and $\rmQ'_{0,0} \vee [\ell_0] \vee [\ell_1] \to \rmQ'$ are isomorphisms by Definition of $\Twist$.
    
    Regarding $(\ref{item:equalizers2})$, recall that equalizers in the category of operads are computed by the arity-wise equalizers of the underlying vector space of operations; the same evidently holds for $\cMult$. 
    It now suffices to observe that, in $\Twist$, 
    $$\rm{Eq}((\rmQ, \rmQ_{0,0}) \substack{\rightarrow\\[-1em] \rightarrow} (\rmQ', \rmQ_{0,0}')) = 
    \rm{Eq}(\rmQ_{0,0} \substack{\rightarrow\\[-1em] \rightarrow} \rmQ'_{0,0}) \vee [\ell_0] \vee [\ell_1]$$
    with differential on the right hand side induced by the differential on $\rmQ$. 
\end{proof}

Analogous results hold (with the same proofs) for the corresponding categories of non-symmetric operads by replacing $\Linf$ by $\Ainf$ and $\ell_n^\alpha$ by $\mu_n^\alpha$.


\section{Maurer-Cartan as the monad on Twist and uncurving}
\label{section:monad}

\subsection{Maurer-Cartan monad}

We now study the monad structure on the category $\Twist$ associated to the adjunction $\forget \dashv \free$ of \cref{thm:adjoint}.
Explicitly, it is given as follows:
\begin{enumerate}
    \item On objects, $\free \forget(\rmQ, \rmQ_{0,0}) = (\cTw(\rmQ_{0,0})^+, \Tw(\rmQ_{0,0}))$, with differential 
    \[
    d_{| \rmQ_{0,0}} = d_{\rmQ_{0,0}}-[\diff,-], 
    \quad 
    d \alpha = \kurv- [\diff,\alpha],
    \quad 
    d\kurv = -\diff(\kurv),
    \quad
    d\diff = -m \circ_1 m,
    \]
    and structural morphism $\cLinf \to \cTw(\rmQ_{0,0})^+$ sending 
    \[
    \ell_0 \mapsto \kurv + \ell_0^\alpha, 
    \quad 
    \ell_1 \mapsto \diff + \ell_1^\alpha,
    \quad \text{and} \quad
    \ell_n \mapsto \ell_n^\alpha
    \quad \text{for } n \geqslant 2.
    \]
    \item On morphisms, $\free \forget$ sends $f$ to $f \vee 1_{\alpha,\kurv,\diff}$.
    \item If we write $(\rmQ_{0,0} \, \hat{\vee} \, [\alpha] \vee [\kurv] \vee [\diff])\, \hat{\vee} \, [\beta] \vee [\kurv'] \vee [\diff']$ for the underlying operad of $(\free \forget \circ \free \forget)(\rmQ, \rmQ_{0,0})$ and $\eta$ the unit of the adjunction $\forget \dashv \free$, then the monad structure is given by 
    \[\begin{array}{ccccccc}
    \free \forget \circ \free \forget & \to & \free \forget, & & \id & \to & \free \forget \\
    \nu & \mapsto & \nu & & \nu & \mapsto & \eta(\nu). \\
    \alpha & \mapsto & 0 & &  &  & \\
    \beta & \mapsto & \alpha & &  &  & \\
    \end{array}
    \]
\end{enumerate}

We will now introduce a new monad~$\MC$ on~$\Twist$, and show that it is isomorphic to~$\free \forget$.
This new monad allows one to trade off the complicated $\cLinf$ structure, but simple differential on~$\free \forget(\rmQ,\rmQ_{0,0})$ for a simple~$\cLinf$ structure, but complicated differential on~$\MC(\rmQ,\rmQ_{0,0})$.
Moreover, on can replace a generally complicated unit with the trivial one (!), the missing information being encoded in the isomorphism~$\MC \cong \free \forget$.

\begin{samepage}
\begin{definition}
\label{definition:cMC}
    We denote by $\MC : \Twist \to \Twist$ the monad defined as follows:
    \begin{enumerate}
        \item On objects, $\MC(\rmQ, \rmQ_{0,0}) \eqdef (\rmQ \, \hat{\vee} \, [\alpha], \rmQ_{0,0} \, \hat{\vee} \, [\alpha])$, with differential
        \[d_{| \rmQ} \eqdef d_\rmQ, \quad d \alpha \eqdef - \ell_0^\alpha \]
        and morphism $\cLinf \to (\rmQ \, \hat{\vee} \, [\alpha], d)$ sending $\ell_n \in \cLinf$ to $\ell_n \in \rmQ$.
        \item On morphisms, $\MC$ sends $f$ to $f \vee 1_\alpha$.
        \item If we write $\rmQ \, \hat{\vee} \, [\alpha] \, \hat{\vee} \, [\beta]$ for the underlying operad of $(\MC \circ \MC)(\rmQ, \rmQ_{0,0})$, then the monad structure of $\MC$ is given by 
        \[\begin{array}{ccccccc}
        \MC \circ \MC & \to & \MC, & & \id & \to & \MC \\
        \nu & \mapsto & \nu & & \nu & \mapsto & \nu. \\
        \alpha & \mapsto & \alpha & &  &  & \\
        \beta & \mapsto & \alpha & &  &  & \\
        \end{array}\]
    \end{enumerate}
\end{definition}
\end{samepage}

\medskip

\begin{theorem}
\label{thm:monad}
    Let $\eta$ be the unit of the adjunction $\forget \dashv \free$.
    Then $\eta \vee (-\id_\alpha) : \MC  \to \free \forget$ is an isomorphism of monads.
\end{theorem}

First need the following identity: 

\begin{lemma}
\label{lem:alternating}
    Let $\rmQ$ be an operad in $\cMult$. 
    Then $\sum_{k\geq 0}  \frac{(-1)^k}{k!} \ell_k^\alpha(\alpha^k)=\ell_0$ in~$\rmQ \, \hat \vee \, [\alpha]$.
\end{lemma}

\begin{proof}
We compute
    \begin{align*}
    \sum_{k\geq 0} \frac{(-1)^k}{k!} \ell_k^\alpha(\alpha^k) 
    &= \sum_{k\geq 0} \frac{(-1)^k}{k!}\sum_{r \geq 0} \frac{1}{r!}\ell_{k+r}(\alpha^{r},-)(\alpha^k) \\
    &= \sum_{n\geq 0} \sum_{k=0}^{n} \frac{(-1)^k}{n!} \binom{n}{k}\ell_n(\alpha^n) = \ell_0 + \sum_{n > 0 } \frac{(1-1)^n}{n!} \ell_n(\alpha^n) 
    =\ell_0.
    \end{align*}
\end{proof}

\begin{proof}
    Let $\calQ=(\rmQ,\rmQ_{0,0})$ be an object in $\Twist$.
    We first need to check that the morphism $\eta \vee (-\id_\alpha) : \MC(\calQ) \to \free \forget(\calQ)$ is in $\Twist$.
    The only non-trivial thing to check is that it commutes with the differentials. 
    This is immediate for elements of 
    $\rmQ$; let us check it for $\alpha$ as well.
    We want to show that 
    \begin{equation}
    \label{eqn:diff-check}
        d(-\alpha)=(\eta \vee -\id_\alpha)(d_{\MC(\calQ)}\alpha).
    \end{equation}
    On the one hand, we have $d(-\alpha)=[m,\alpha]-\kurv$.
    On the other hand, using the explicit description of $\eta$ on the image of $\cLinf$, we have 
    \begin{align*}
    (\eta \vee -\id_\alpha)(-\ell_0^\alpha)
    & =(\eta \vee -\id_\alpha)\left(-\kurv -m(\alpha)-\sum_{n\geqslant 2} \frac{1}{n!}\ell_n(\alpha^n)\right) \\
    & =-(\kurv+\ell_0^\alpha)+(m+\ell_1^\alpha)(\alpha)-\sum_{n\geq 2}  \frac{(-1)^n}{n!}\ell_n^\alpha(\alpha^n).
    \end{align*}
    The equality \eqref{eqn:diff-check} then follows from \cref{lem:alternating}.

    We now prove that $\eta \vee (-\id_\alpha) : \MC(\calQ) \to \free \forget(\calQ)$ is an isomorphism.
    In order to do that, we consider on both sides the (decreasing) filtration by the number of $\alpha$ and show that the morphism induced on the associated graded is an isomorphism.
    Writing $\eta = \sum_{k \geq 0} \eta^k$ according to the grading by the number of $\alpha$, the induced morphism between associated graded is $\eta^0 \vee (-\id_\alpha)$.
    It remains to show that $\eta^0 = \id_\calQ$.
    Since $\ell_0^\alpha,\ell_1^\alpha \in \free \forget(\calQ)$ are in positive $\alpha$-filtration, we have that $\eta^0(\kurv)=\kurv$ and $\eta^0(m)=m$.
    Moreover, by \cref{thm:adjoint}, the counit $\varepsilon$ of the adjunction $\forget \dashv \free$ is given by $\varepsilon(\alpha)=0$ and~$\varepsilon(\nu)=\nu$.
    Now, writing $\calQ=(\rmQ,\rmQ_{0,0})$, we have $\forget(\calQ)=\rmQ_{0,0}$ with the induced differential. 
    The first of the unit-counit identities then implies that $\id_{\rmP}=\varepsilon_{\rmQ_{0,0}}\forget(\eta)=\varepsilon_{\rmQ_{0,0}}\forget(\eta^0)$.
    We deduce that $\eta^0(\nu)=\nu$ for any $\nu \in \rmQ_{0,0}$, which completes the proof that $\eta^0=\id_\calQ$, as well as the proof that $\eta \vee (-\id_\alpha) : \MC(\calQ) \to \free \forget(\calQ)$ is an isomorphism.

    Let us check that $\eta \vee 1_\alpha$ is a natural transformation. 
    Let $f : \calQ \to \calQ'$ be a map in $\Twist$.
    We have $\MC(f)=f \vee 1_\alpha$ and $\free \forget(f)=f_{|\rmQ_{0,0}} \vee 1_{\alpha,\kurv,\diff}$.
    We thus need to check that the following diagram commutes:
    \begin{center}
    \begin{tikzcd}
    \MC(\calQ) \arrow[rr,"\eta \, \vee \,  (-\id_\alpha)"]  \arrow[d,"f \, \vee \, 1_\alpha"] & &\free \forget(\calQ) \arrow[d,"f_{|\rmQ_{0,0}} \vee \, 1_{\alpha,\kurv,\diff}"] \\
    \MC(\calQ') \arrow[rr,"\eta \, \vee \, (-\id_\alpha)"]       &    & \free \forget(\calQ') . 
    \end{tikzcd}
    \end{center}
    This can be seen directly by considering the images of $\alpha, \kurv, \diff$ and $\nu \in \rmP$ separately.
    
    It remains to show that $\eta \vee (-\id_\alpha)$ respects the monad structures. 
    This amounts to showing that the following two diagrams commute:
    \begin{center}
    \begin{tikzcd}
    \MC\circ \MC \arrow[rr,"\MC(\eta \vee -\id_\alpha)"] \arrow[d] &  & \MC \circ \free \forget \arrow[rr,"(\eta \vee -\id_{\beta})\free \forget"] &  & \free \forget \circ \free \forget \arrow[d,"\mu"'] &  &              & \id \arrow[ld,"\iota"'] \arrow[rd,"\eta"] &   \\
    \MC \arrow[rrrr,"\eta \vee -\id_\alpha"]         &  &              &  & \free \forget           &  & \MC \arrow[rr,"\eta \vee -\id_\alpha"] &                         & \free \forget
    \end{tikzcd}
    \end{center}
    The triangle on the right, where $\iota : \calQ \to \MC(\calQ)$ is the inclusion, is easily seen to commute.
    The multiplication $\mu : \free \forget \circ \free \forget \to \free \forget$ in the left pentagon is given by $\mu \eqdef \free \varepsilon (\forget)$.
    The lower composite sends $\nu \mapsto \eta(\nu)$ and $\alpha,\beta \mapsto -\alpha$.
    The upper composite does the same: for~$\beta$, this is immediate from the definitions; for~$\alpha$, this is easy to see if one remembers from~\cref{prop:comonad-structure} that $\eta_{\free}(\alpha)=\alpha+\beta$.
    For $\nu \in \rmQ$, one just has to observe that, since the counit $\varepsilon$ is part of the multiplication~$\mu$ and $\varepsilon(\alpha)=0$, the first application of $\eta$ in the composite amounts to the application of $\eta^0=\id_\rmQ$.
    We conclude that both diagrams commute, and the proof is complete.
\end{proof}

Let us now recall the operadic incarnation of the procedure of twisting by a Maurer-Cartan element (see e.g.\ \cite[Section~5.3]{DotsenkoShadrinVallette23}). 

\begin{definition} 
\label{maurer cartan endofunctor}
    Given an object $\rmP \in \Mult$, we write $\MC(\rmP)$ for the operad 
    $\rmP \, \hat \vee \, [\alpha]$ with differential characterized by $d_{\MC(\rmP)}|_{\rmP} \eqdef d_{\rmP}$ and $d_{\MC(\rmP)} \alpha \eqdef - \sum_{k \geq 0} \frac{1}{k!}\ell_{k} (\alpha^k) =  -\ell_0^\alpha$.

    We regard $\MC$ as a functor, acting on morphisms by $f \mapsto f \vee \mathrm{id}_{[\alpha]}$.  
    By the structure map~$\Linf \to \rmP \to \rmP \hat \vee [\alpha]$, we regard $\MC$ as an endofunctor of $\Mult$.  
\end{definition}

\begin{remark}
    In general, the assignment $\ell_n \mapsto \ell_n^\alpha$ does not define a morphism $\Linf \to \MC(\rmP)$.
    This problem is remedied when passing from $\MC$ to $\Tw$.
\end{remark}

The following important observation is immediate from the definitions.

\begin{lemma}
\label{remark:relation-MC} The forgetful functor $\forget$ intertwines the $\MC$ functors of Definitions~\ref{maurer cartan endofunctor} and~\ref{definition:cMC}~: we have
$\forget \circ \MC = \MC \circ \forget$. 
\end{lemma}

\begin{corollary}
\label{corollary:MC-isomorphic-Tw}
    Let $\rmP \in \Mult$ be an operad under $\Linf$.
    If $\rmP$ admits a $\Tw$-coalgebra structure $\eta_\rmP : \rmP \to \Tw(\rmP)$, then the morphism
    \[\eta_\rmP \vee (-\id_\alpha) : \MC(\rmP) \overset{\cong}{\longrightarrow} \Tw(\rmP) \]
    is an isomorphism.
\end{corollary}

\begin{proof}
    Given $\calQ$ in $\Twist$, by \cref{remark:relation-MC} and \cref{thm:main-thm}, applying $\forget$ to the isomorphism $\MC(\calQ) \cong \free \forget(\calQ)$ from \cref{thm:monad}, we get an isomorphism 
$\MC(\forget \calQ) \overset{\cong}{\longrightarrow} \Tw(\forget \calQ)$.
\end{proof}

\begin{remark}
    Taking $\rmP=\Ainf$, the formula for $\eta_\rmP \vee (-\id_\alpha)$ is the same as in \cite[Thm.~4.7]{ChuangLazarev13}.
\end{remark}

Analogous results hold (with the same proofs) for the corresponding categories of non-symmetric operads by replacing $\Linf$ by $\Ainf$ and $\ell_n^\alpha$ by $\mu_n^\alpha$ and using the following non-symmetric version of~\cref{lem:alternating}.

\begin{lemma}
    Let $\rmQ$ be an operad in $\nscMult$. 
    Then $\sum_{k\geq 0} (-1)^k \mu_k^\alpha(\alpha^k) = \mu_0$ in~$\rmQ \, \hat \vee \, [\alpha]$.
\end{lemma}
\begin{proof}
We compute
    \begin{align*}
    \sum_{k\geq 0} (-1)^k \mu_k^\alpha(\alpha^k) 
    &= \sum_{k\geq 0} (-1)^k \sum_{r_0,\ldots,r_k \geq 0}\mu_{k+r_0+\cdots + r_k}(\alpha^{r_0},-,\alpha^{r_2},\ldots,-,\alpha^{r_k})(\alpha^k) \\
    &= \sum_{n\geq 0} \sum_{k=0}^{n} (-1)^k \binom{n}{k}\mu_n(\alpha^n) = \mu_0 + \sum_{n > 0} (1-1)^n \mu_n(\alpha^n) 
    =\mu_0.
    \end{align*}
\end{proof}

\subsection{Uncurving a morphism of curved $\Linf$-operads}

By \emph{uncurving} a morphism $\rmQ \to \rmP$ in $\cMult$, we mean the procedure that inputs a Maurer-Cartan element $a$ in $\rmP(0)$ and outputs a morphism $\rmQ/(\ell_0, \ell_1) \to \rmP^{\ell_1^a}$ in $\Mult$.
We will show that there is a meaningful uncurving procedure when $\rmQ$ is in the image of the forgetful functor $\calL \calI : \Twist \to \cMult$.

We will need the following result. 

\begin{lemma}
\label{G is monadic}
    The functor $\calG : \Mult \to \Twist$ is monadic, with associated monad $\MC : \Twist \to \Twist$.
\end{lemma}
\begin{proof}
    We already know that $\calG$ has a left adjoint $\calF$, and that the monad $\calG \calF$ is isomorphic to $\MC$ (\cref{thm:monad}).
    
    We now want to apply the crude monadicity theorem. 
    We need to check that 
    \begin{enumerate}
        \item $\calG$ reflects isomorphisms,
        \item $\Mult$ has coequalizers, and $\calG$ preserves them.
    \end{enumerate}

    We first check that $\calG$ reflects isomorphisms. Let $f : \rmP \to \rmP'$ be a morphism in $\cMult$ such that
    \[\calG(f) = f \vee \id_{[\alpha] \vee [\kappa] \vee[m]} : (\cTw(\rmP)^+, \Tw(\rmP)) \to (\cTw(\rmP')^+, \Tw(\rmP')) \]
    is an isomorphism.
    In particular, $f \vee \id_{[\alpha]} : \Tw(\rmP) \to \Tw(\rmP')$ is an isomorphism.
    Therefore it induces an isomorphism between the underlying non-dg operads $f \vee \id_{[\alpha]} : \rmP \vee [\alpha] \to \rmP' \vee [\alpha]$.
    This implies that $f$ induces an isomorphism $\rmP \to \rmP'$ between the underlying non-dg operads.
    We conclude that $f$ is an isomorphism.

    Observe that, since the category of operads has colimits, this is also the case of the category $\Mult$.
    Therefore it remains to check that $\calG$ preserves co-equalizers.
    Note that, since the forgetful functor from dg operads to non-dg operads reflects isomorphisms and preserves coequalizers, it is enough to check that $\calG$ preserves co-equalizers at the level of non-dg operads.
    Now the underlying pair of non-dg operads of $\calG(\rmP) = (\cTw(\rmP)^+, \Tw(\rmP))$ is $(\rmP \vee [\alpha] \vee [\kappa] \vee [m], \rmP \vee [\alpha])$, and $\calG$ acts as $f \mapsto f \vee \id_{[\alpha] \vee [\kappa] \vee[m]}$ on morphisms.
    Therefore~$\calG$ preserves coequalizers, which finishes the proof.
\end{proof}

Recall that, given $\rmP \in \cMult$, the underlying operad of $(\calJ \calR)(\rmP)$ is $(\rmP \vee \rmT)^+$, where~$\rmT$ is the operad generated by two operations $\alpha$ and $\kappa$ of arity $0$ with $d_\rmT \alpha = \kappa$.
Moreover, the morphism $\cLinf \to (\calJ \calR)(\rmP)$ sends $\ell_0 \mapsto \kappa + \ell_0^\alpha$.
In this case we write $\MC((\calJ \calR)(\rmP)) = (\calJ \calR)(\rmP) \vee [\beta]$ with $d_{\MC} \beta = -(\kappa + \ell_0^\alpha)^\beta$.

\begin{prop}
\label{MC element = MC-algebra structures}
    Let $\rmP$ be an object of $\cMult$. 
    There is a bijection
    \begin{equation}
    \label{bijection-MC}
        \begin{matrix}
            \{a \in \calF_1 \rmP(0)_0 \mid d_\rmP a + \ell_0^a = 0\} & \longrightarrow & \{ \MC\text{-algebra structure on } (\calJ \calR)(\rmP) \} \\
            a & \mapsto & \id \vee (\beta \mapsto a - \alpha).
        \end{matrix}
    \end{equation}

    Moreover, given a Maurer-Cartan element $a$ in $\calF_1 \rmP(0)$, the corresponding $\Linf$-operad given by \cref{G is monadic} is $\rmP^{\ell_1^\alpha}$, with morphism $\Linf \to \rmP^{\ell_1^a}$ sending $\ell_n \mapsto \ell_n^a$.
\end{prop}
\begin{proof}
    First observe that the map is well defined. 
    Indeed, the image of a given $a$ is a morphism of dg operads, since $d \beta = -(\kappa + \ell_0^\alpha)^\beta = -(\kappa + m(\beta) + \ell_0^{\alpha+\beta})$ is sent to the element $-(\kappa + m(a - \alpha) + \ell_0^a)$, which is exactly the differential of $(a - \alpha)$ in $(\calJ \calR)(\rmP)$.
    The two diagrams defining the $\MC$-algebra structure (compatibility with the unit and the multiplication of $\MC$) are then readily seen to be satisfied. 
    Remarkably, compatibility with multiplication is automatic once compatibility with the unit is satisfied.

    It is straightforward to see that the map is injective, so it remains to check that it is surjective.
    Let $f : \MC((\calJ \calR)(\rmP)) \to (\calJ \calR)(\rmP)$ be a $\MC$-algebra structure on $(\calJ \calR)(\rmP)$, and consider the element $a \eqdef \alpha + f(\beta)$ in $(\calJ \calR)(\rmP)$.
    The precomposition of this morphism with the inclusion $(\calJ \calR)(\rmP) \hookrightarrow \MC((\calJ \calR)(\rmP))$ is the identity, which implies $f(\alpha)=\alpha$.
    Moreover, since $f$ is a morphism of dg operads, we have
    \begin{align*}
        d_{(\calJ \calR)(\rmP)}(a) 
        &=d_{(\calJ \calR)(\rmP)}(\alpha + f(\beta)) \\
        &= \kurv -m(\alpha) + f(d_\MC\beta) \\
        &= \kurv -m(\alpha)+f(-(\kurv+m(\beta)+\ell_0^{\alpha+\beta}))\\
        &= -m(a)-\ell_0^{a}.
    \end{align*}
    Therefore we have $d_{\rmP \vee \rmT \vee [\diff]}(a)=-\ell_0^a$, and by freeness in $\alpha,\kurv,\diff$, we conclude that $a \in \rmP(0)$ is a Maurer--Cartan element.
    Therefore that the assignment \eqref{bijection-MC} is also surjective, and the proof is complete.
\end{proof}

\begin{lemma}
\label{Linf-operad for MC element}
    Let $a$ be a Maurer-Cartan element $a$ in $F_1 \rmP(0)$.
    The $\Linf$-operad corresponding to the $\MC$-algebra structure $\id \vee (\beta \mapsto a - \alpha)$ on $(\calJ \calR)(\rmP)$ via \cref{G is monadic} is $\rmP^{\ell_1^a}$, with morphism $\Linf \to \rmP^{\ell_1^a}$ sending $\ell_n \mapsto \ell_n^a$.
\end{lemma}
\begin{proof}
    We first compute the morphism 
    \begin{equation*}
    \label{GF-algebra structure}
        f \colon (\calG \calF)((\calJ \calR)(\rmP)) \to (\calJ \calR)(\rmP)
    \end{equation*}
    induced by the $\MC$-algebra structure on $(\calJ \calR)(\rmP)$ and the isomorphism $\MC \xrightarrow{\sim} \calG \calF$ of \cref{thm:monad}. 
    If we let $\alpha$, respectively $\beta$, be the extra operation added by the functor $\calR$, respectively by the functors $\MC$ and $\calG$, then the isomorphism 
    \[\MC((\calJ \calR)(\rmP)) \xrightarrow{\sim} (\calG \calF)((\calJ \calR)(\rmP)) \]
    sends $\beta \mapsto - \beta$, $\alpha \mapsto (\alpha + \beta)$, and acts as the identity otherwise.
    Therefore the morphism $f$ sends $\beta \mapsto (\alpha - a)$, $\alpha \mapsto a$, and acts as the identity otherwise.

    We now need to compute the coequalizer in $\Mult$ of the diagram
    \[ 
    \begin{tikzcd}
        (\calF \calG \calF)((\calJ \calR)(\rmP)) \ar[r, bend left, "\calF(f)"] \ar[r, bend right, "\varepsilon^{\calF \calG}"] & \calF((\calJ \calR)(\rmP)).
    \end{tikzcd}
    \]
    The latter is given by the quotient of $\calF((\calJ \calR)(\rmP)) = (\rmP \vee \rmT \vee [m])^m / (\kappa + \ell_0^\alpha, m + \ell_1^\alpha)$ by the operadic ideal generated by the image of $(\calF(f) - \varepsilon^{\calF \calG})$.
    Recall that $\varepsilon^{\calF \calG}$ sends $\beta \mapsto 0$ and acts as the identity otherwise (\cref{prop:comonad-structure}). 
    The result is therefore the quotient of $(\rmP \vee \rmT \vee [m])^m / (\kappa + \ell_0^\alpha, m + \ell_1^\alpha)$ by the operadic ideal generated by $(\alpha-a)$, which is exactly $\rmP^{\ell_1^a}$. 
\end{proof}

Recall that, given an object $\calQ$ in $\Twist$, the object $\MC(\calQ)$ has a natural structure of $\MC$-algebra given by the monadic multiplication $\MC \circ \MC \to \MC$.

\begin{lemma}
\label{Linf-operad for MC}
    Given an object $\calQ$ of $\Twist$, the $\Linf$-operad corresponding to the $\MC$-algebra $\MC(\calQ)$ through \cref{G is monadic} is $\calF(\calQ)$.
\end{lemma}
\begin{proof}
    We first compute the morphism 
    \begin{equation*}
    \label{GF-algebra structure}
        f \colon (\calG \calF)(\MC(\calQ)) \to \MC(\calQ)
    \end{equation*}
    induced by the $\MC$-algebra structure on $\MC(\calQ)$ and the isomorphism $\MC \xrightarrow{\sim} \calG \calF$ of \cref{thm:monad}. 
    If we let $\alpha$, respectively $\beta$, be the extra operation added by the functor $\MC$, respectively by the functors $\calG$ and a second application of $\MC$, then the isomorphism 
    \[\MC(\MC(\calQ)) \xrightarrow{\sim} (\calG \calF)(\MC(\calQ)) \]
    sends $\beta \mapsto - \beta$ and acts as the identity otherwise.
    Therefore the morphism $f$ sends $\beta \mapsto - \beta$ and acts as the identity otherwise.

    We now need to compute the coequalizer in $\Mult$ of the diagram
    \[ 
    \begin{tikzcd}
        (\calF \calG \calF)(\MC(\calQ)) \ar[r, bend left, "\calF(f)"] \ar[r, bend right, "\varepsilon^{\calF \calG}"] & \calF(\MC(\calQ)).
    \end{tikzcd}
    \]
    The latter is given by the quotient of $\calF(\MC(\calQ)) = \MC(\rmQ) / (\mu_0)$ by the operadic ideal generated by the image of $(\calF(f) - \varepsilon^{\calF \calG})$.
    Since $\varepsilon^{\calF \calG}$ sends $\beta \mapsto 0$ and acts as the identity otherwise (\cref{prop:comonad-structure}), the result is the quotient of $\MC(\rmQ) / (\mu_0)$ by the operadic ideal generated by $\beta$, which is exactly $\calF(\calQ)$. 
\end{proof}

We can now explain the uncurving procedure, e.g.\ the fact that ``given a Maurer--Cartan element $a$ in a curved $\Linf$-algebra $A$, one can twist by $a$ to get a non-curved $\Linf$-algebra''.

\begin{prop}
\label{bijection-MC-elements}
    Let $\calQ$ be an object of $\Twist$, let $\rmP$ be an object of $\cMult$, and consider a morphism $(\calL \calI)(\calQ) \to \rmP$.
    Then, any Maurer-Cartan element $a$ in $\rmP(0)$ induces a morphism of $\Linf$-operads $\calF(\calQ) \to \rmP^{\ell_1^a}$.
\end{prop}
\begin{proof}
    By adjunction, we have a morphism $\calQ \to (\calJ \calR)(\rmP)$.
    Applying $\MC$ to the latter, we get a morphism $\MC(\calQ) \to \MC((\calJ \calR)(\rmP))$.
    Note that the latter is a morphism of $\MC$-algebras as the image of a morphism under $\MC$.

    Now, according to \cref{MC element = MC-algebra structures}, we know that the Maurer-Cartan element $a$ in $\rmP(0)$ corresponds to a $\MC$-algebra structure $\MC((\calJ \calR)(\rmP)) \to (\calJ \calR)(\rmP)$.
    Note that the latter morphism is also a morphism of $\MC$-algebras since it is the structural morphism of a $\MC$-algebra.
    
    Therefore we get a morphism of $\MC$-algebras $\MC(\calQ) \to (\calJ \calR)(\rmP)$.
    Using the equivalences between $\MC$-algebras and $\Linf$-operads from \cref{G is monadic}, we get a morphism from the $\Linf$-operad corresponding to $\MC(\calQ)$ to the one corresponding to $(\calJ \calR)(\rmP)$ with $\MC$-algebra structures induced by $a$.
    Using the respective characterizations of these operads from \cref{Linf-operad for MC} and \cref{Linf-operad for MC element}, we get a morphism $\calF(\calQ) \to \rmP^{\ell_1^a}$.
    This concludes the proof.
\end{proof}

\section{Homotopy fixed points of Tw}
\label{homotopy-fixed-points}

\subsection{Homotopy fixed points of the Maurer-Cartan endofunctor}





We give a criterion under which an operad is a homotopy fixed point of the functor $\MC$ from \cref{maurer cartan endofunctor}.

We will need the following. 

\begin{lemma}
\label{lemma:homology-coproduct}
    Let $\rmT$ be the operad freely generated by arity zero elements $\alpha, \kurv$ where $\alpha$ has degree $0$ and $\kurv$ has degree $(-1)$, with differential $d_\rmT \alpha = \kurv$.
    Given an operad $(\rmP, d_\rmP)$, the map $(\rmP \, \hat{\vee} \,\rmT, d_\rmP + d_\rmT) \to (\rmP, d_\rmP)$ that sends $\kurv$ and $\alpha$ to $0$ is a quasi-isomorphism.
\end{lemma}

\begin{proof}
    Since $\rmT$ is generated by elements in arity $0$, the underlying 
    module of $\rmP \, \hat{\vee} \, \rmT$ is the composite product $\rmP \, \hat \circ \, \rmT$ (see \cite[Section 6.2.1]{LodayVallette12}).
    Using the operadic K\"unneth formula \cite[Proposition 6.2.3]{LodayVallette12}, we get that the underlying module of $H_*(\rmP \, \hat{\vee} \, \rmT, d_\rmT)$ is $H_*(\rmP \, \hat \circ \, \rmT, d_\rmT) \cong H_*(\rmP, 0) \, \hat \circ \, H_*(\rmT, d_\rmT)$. 
    Since $(\rmT, d_\rmT)$ is acyclic, the latter is equal to $\rmP$.
\end{proof}

\begin{prop}
\label{prop:MC-fixed-point}
    Let $\rmP$ be an object of $(\Linf \downarrow \dgOp)$.
    Assume that $\ell_2$ is not a right zero divisor in $H_*(\rmP)$, i.e.\ $d_\rmP \nu = 0$ and $\nu \circ_1 \ell_2 \in \im(d_\rmP)$ imply $\nu \in \im(d_\rmP)$.
    Then the natural map $\MC(\rmP) \to \rmP$ is a quasi-isomorphism.
\end{prop}

\begin{proof}
    Consider the dg operad $(\rmP \, \hat{\vee} \, \rmT, d)$, where $d = d_\rmP + d_\rmT$.
    Observe that there is a decomposition $d=d_0+d_1$ with respect to the $(\kurv + \ell_0^\alpha)$-grading:
    \[\left\{
    \begin{array}{lll}
    d_0 \kurv = 0,  & d_0 \alpha = -\ell_0^{\alpha}, & d_0 \nu = d_\rmP \nu, \\
    d_1 \kurv = 0,  & d_1 \alpha = \kurv + \ell_0^{\alpha}, & d_1 \nu = 0. \\  
    \end{array}
    \right. \]
    
    According to \cref{lemma:homology-coproduct}, the map $(\rmP \hat{\vee} \rmT, d) \to (\rmP, d_\rmP)$ that sends $\kurv$ and $\alpha$ to $0$ is a quasi-isomorphism.
    Let $I$ be the operadic ideal generated by $(\kurv + \ell_0^\alpha)$.
    Note that $I$ is preserved by $d$, since $d(\kurv + \ell_0^\alpha) = - \ell_1^\alpha(\kurv + \ell_0^\alpha)$.
    Observe that the quotient of $(\rmP \hat{\vee} \rmT, d)$ by $I$ is $\MC(\rmP, d_\rmP)$.
    Therefore it is enough to show that $I$ is acyclic.

    Let $w$ be a closed element in $I$. 
    We will prove recursively that for all $k \geq 0$, there exists $w^0_I \in I$ and $t^0, \dots, t^k, w^k_\alpha \in \rmP \hat{\vee} [\alpha]$ such that 
    \begin{enumerate}
        \item $w^k_\alpha$ is in $\alpha$-filtration level $k$,
        \item $t^j$ is in $\alpha$-filtration level $(j-2)$ for all $j \in \{0, \dots, k\}$,
        \item $w= d(w^0_I + (\sum_{j=0}^k t^j) \circ_1 (\kurv + \ell_0^\alpha) + w^k_\alpha)$.
    \end{enumerate}
    Once this is proved, we will have $w= d(w^0_I + (\sum_{j \geq 0} t^j) \circ_1 (\kurv + \ell_0^\alpha)) \in d(I)$.

    We start with the $k=0$ case.
    Since the quasi-isomorphism $(\rmP \hat{\vee} \rmT, d) \to (\rmP, d_\rmP)$ sends~$I$ to~$0$, we know that there exists $w^0$ such that $w = d(w^0)$. 
    We decompose $w^0$ according to the $(\kurv + \ell_0^\alpha)$-grading: $w^0 = w^0_I + w^0_\alpha$ with $w^0_I \in I$ and $w^0_\alpha \in \rmP \hat{\vee} [\alpha]$ (here $t^0 = 0$). 

    Assume now that the result holds for some $k \geq 0$.
    Observe that we can write
    \[w^k_\alpha = \nu_k(\alpha^k, -) + \nu_{k+1}(\alpha^{k+1},-) + \cdots \]
    where $\nu_i \in \rmP$ is $\Sym_i$ invariant.
    Using the decomposition $d=d_0+d_1$ and the fact that $w, d_1(w^k_\alpha)$ and $d(w^k_I)$ are in the ideal~$I$ while $d_0(w^k_\alpha)$ is not, we deduce that $d_0(w^k_\alpha)=0$.
    This implies that the term in $d_0(w^k_\alpha)$ with the lowest number of $\alpha$ vanishes, i.e.\ we have $(d_\rmP \nu_k)(\alpha^k, -) = 0$ which, by freeness of $\alpha$, implies that $d_\rmP \nu_k = 0$. 
    If $k=0$, we can therefore set $w^1_\alpha := w^0_\alpha - \nu_0$ and $t_1 = t_0 = 0$.
    Assume now that $k \geq 1$.
    We use that the term in $d_0(w^k_\alpha)$ with the second lowest number of $\alpha$ vanishes, i.e.\ we have
    \[(-1)^{|\nu_k|} k (\nu_k \circ _1 \ell_2)(\alpha^{k+1}, -) + (d_\rmP \nu_{k+1})(\alpha^{k+1}, -) = 0. \]
    By freeness of $\alpha$, we get that $\nu_k \circ_1 \ell_2 \in \im(d_\rmP)$ and therefore, by assumption, that $\nu_k \in \im(d_\rmP)$: we write $\nu_k = d_\rmP(\tilde{\nu_k})$ with $\tilde{\nu_k} \in \rmP$ invariant under $\Sym_k$. 
    Now we set 
    \begin{align*}
        w^{k+1}_\alpha &:= w^k_\alpha - \nu_k(\alpha^k,-) + (-1)^{|\nu_k|} k \nu_k(\ell_0^\alpha, \alpha^{k-1}, -) \quad \text{and} \\
        t^{k+1} &:= -(-1)^{|\nu_k|} k \nu_k(-, \alpha^{k-1}, -)
    \end{align*}
    so that $t^{k+1} \circ_1 (\kurv + \ell_0^\alpha) + w^{k+1}_\alpha = w^k_\alpha - d(\tilde{\nu_k}(\alpha^k,-))$.
    This finishes the induction step.
\end{proof}

\subsection{Proof of \cref{thm:homotopy-fixed-point intro}}

When $\rmP$ admits a $\Tw$-coalgebra structure, $\MC(\rmP)$ is isomorphic to $\Tw(\rmP)$ by \cref{corollary:MC-isomorphic-Tw}.
Therefore, the result follows from \cref{prop:MC-fixed-point}. $\square$

\vspace{2mm}

\subsection{A useful criterion}

Finally, we give a useful result in order to apply \cref{thm:homotopy-fixed-point intro} in concrete examples.
We will apply this criterion repeatedly in \cref{examples}.

Recall that for any operad $\rmP$, one can form the free $\rmP$ algebra on $n$ generators (\cite[Section 5.2.5]{LodayVallette12}); we denote it $\rmP(x_0, \dots, x_{n-1})$. 

\begin{lemma}
\label{lemma:criteria-bracket-zero-divisor}
    Let $\rmP$ be a non-dg operad together with a morphism $\Lie \to \rmP$. 
    If the canonical $\rmP$-algebra map 
    \[
    \rmP(x_0, \dots, x_{n-1}) \to \rmP(u,v, x_0, \dots, x_{n-1}) / \left( \ell_2(u,v) - x_0 \right)
    \]
    is injective for every $n$, then $\ell_2$ is not a right zero divisor in $\rmP$.
\end{lemma}

\begin{proof}
    First, observe that an operation $\nu$ of arity $n$ in an operad $\rmP$ is zero if and only if the corresponding operation in the free $\rmP$-algebra on $n$ generators is zero.
    Let $\nu \in \rmP(n)$ such that $\nu \circ_1 \ell_2 = 0$.
    Consider the element $\nu(x_0, \dots, x_{n-1})$ in $\rmP(x_0, \dots, x_{n-1})$.
    Its image under 
    \[
    \rmP(x_0, \dots, x_{n-1}) \to \rmP(u,v, x_0, \dots, x_{n-1}) / \left( \ell_2(u,v) - x_0 \right)
    \]
    is zero by assumption.
    By injectivity, $\nu(x_0, \dots, x_{n-1})$ vanishes in $\rmP(x_0, \dots, x_{n-1})$, which implies that $\nu$ vanishes in $\rmP$.
    This finishes the proof.
\end{proof}


\subsection{Creating Tw-coalgebra morphisms}

We give an alternative proof of the result \cite[Thm. 10.1]{DolgushevWillwacher15} of Dolgushev and Willwacher.

\begin{theorem}
\label{create Tw morphism}
    Let $\rmP,\rmQ \in \Mult$, and let $f: \rmP \to \rmQ$ be a morphism in the same category.
    Suppose that $\rmP$ and $\rmQ$ are $\Tw$-coalgebras, hence have lifts to $\Twist$, which we denote by $\cP$ and $\cQ$.  Assume in addition that $\rmQ$ is a $\Tw$ homotopy fixed point, and $\cQ$ is a $\calG \calF$-algebra with structure morphism $a : \calG\calF(\cQ) \to \cQ$.
    Then, the morphism
    \begin{equation}
    \label{DW-morphism}
    \calF(a) \circ \Tw(f) \circ \calF(\eta_{\cP}) : \rmP \to \rmQ
    \end{equation}
    is a morphism of $\Tw$-coalgebras which is homotopy equivalent to $f$.
\end{theorem}
\begin{proof}
We show that going from top to bottom and then back in the following diagram is homotopic to the identity
\begin{equation*}
\begin{tikzcd}
\hom_{\MultSym}(\rmP,\rmQ) 
& = & 
\hom_{\MultSym}(\calF(\cP),\calF(\cQ)) \arrow[dd, "\calG(-)\,\circ\,\eta_{\cP}"', bend right]  
\\
  &   &                                                               \\
  &  \cong & \hom_{\Twist}(\cP,\calG\calF(\cQ)) \arrow[uu, "\varepsilon_{\rmQ}\,\circ\,\calF(-)"', bend right] \arrow[dd, "a\,\circ\,(-)"', bend right] \\
  &   &                                                               \\
  &   & \hom_{\Twist}(\cP,\cQ). \arrow[uu, "\eta_{\cQ}\,\circ\,(-)"', bend right]                         
\end{tikzcd}
\end{equation*}
First observe that the composition
\begin{equation}
\label{dive-composition}
\varepsilon_{\rmQ}
\,\circ\,
\calF(\eta_{\cQ})
\,\circ\,
\calF(a)
\,\circ\,
(\calF\calG)(-)
\,\circ\,
\calF(\eta_{\cP})
\end{equation}
is indeed equal to the morphism~\eqref{DW-morphism} since the first two terms cancel in virtue of the adjunction $\calF \dashv \calG$.
Now we claim that the composite $\calF(\eta_{\cQ}) \circ \calF(a)$ is homotopic to the identity.
This will prove the theorem, since~\eqref{dive-composition} will then be homotopic to
\begin{equation*}
    \varepsilon_{\rmQ}
\,\circ\,
(\calF\calG)(-)
\,\circ\,
\calF(\eta_{\cP})
=
(-)
\,\circ\,
\varepsilon_{\rmP}
\,\circ\,
\calF(\eta_{\cP})
=
(-)
\end{equation*}
since $\varepsilon$ is a natural transformation.
It thus remains to prove that $\calF(\eta_{\cQ}) \circ \calF(a)$ is homotopic to the identity.
Consider the following diagram
\begin{equation*}
\begin{tikzcd}
 \calF(\cQ) \arrow[r, "\calF(\eta_{\cQ})"] \arrow[rd, "\id"', bend right] & \calF\calG\calF(\cQ) \arrow[d, "\calF(a)"] \\     
 & \calF(\cQ)               
\end{tikzcd}
\end{equation*}
It is commutative since by definition of a $\calG \calF$-algebra we have $\calF(a\, \circ \, \eta_{\cQ})=\id_\rmQ$.
Since we assumed that $\rmQ$ is a $\Tw$ homotopy fixed point, the map $\calF(\eta_{\cQ})$ is a quasi-isomorphism. 
Therefore $\calF(a)$ is a quasi-isomorphism as well.
We conclude that the composite $\calF(\eta_{\cQ})
\,\circ\,
\calF(a)$ is homotopic to the identity.
This finishes the proof.
\end{proof}

\begin{remark} 
\label{comparison DW}
    The original  \cite[Theorem 10.1]{DolgushevWillwacher15} assumed instead that $\rmQ = \Tw(\rmQ')$.  We recover it as follows: the lift $\calG(\rmQ') \in \Twist$ of $\Tw(\rmQ')$ always admits a $\calG \calF$-algebra structure given by $\calG(\varepsilon_{\rmQ'}) : \calG \calF \calG(\rmQ') \to \calG(\rmQ')$.   
    Our formulation at first may appear slightly more general; in fact this is an illusion since $\calG$ is monadic (\cref{G is monadic}) and so any $\calG \calF$-algebra has underlying operad of the form $\Tw(\rmQ')$.
\end{remark}

\begin{remark}
    In \cite[Theorem~10.1]{DolgushevWillwacher15}, it is further assumed that~$\rmP$ is a homotopy fixed point for~$\Tw$, but this assumption is not used in their proof (or ours). 
\end{remark}


\section{Koszul preduals}
\label{Koszul preduals}

\subsection{Exposition of the results}

We denote by $\coCom$ the category of cooperads
under the linear dual of the operad $\uCom$ encoding unital commutative algebra.
Recall that~$\uCom(n)=\mathbf{k} \mu_n$ is the trivial $\Sym_n$ representation for all $n \geqslant 0$, and operadic composition is given by $\mu_k \circ_i \mu_l \eqdef \mu_{k+l-1}$ for all $k,l\geqslant 0$.

Let $\rmO$ be the $\Sym$-module defined by $\rmO(0) \eqdef \mathbf{k}$ and $\rmO(n) \eqdef 0$ for $n \geq 1$.

\begin{definition}
\label{coCurv}
    We denote by $\coCurv$ the category defined as follows
    \begin{itemize}[leftmargin=*]
        \item[--] 
        Objects are pairs $(\rmC,\form)$ where $\rmC \in \coCom$, and $\form : \rmC \to \rmO$ is a $\Sym$-module map of degree $0$ such that $\form (\mu_0^*) = 1$.
        \item[--] 
        Morphisms are maps $f: \rmC \to \rmC'$ in $\coCom$ such that $\form' \circ f = \form$.
    \end{itemize}
\end{definition}

We first prove that the forgetful functor $\coCurv \to \coCom$ admits a right adjoint, and we characterize the unit of this  adjunction (\cref{cooperadic adjunction}).

Then we show (\cref{cobar of coCurv}) that the cobar construction induces a well defined functor $\tilde \Omega : \coCurv \to \Curv, \, (\rmC, \form) \mapsto (\Omega \rmC, \Omega \ker(\form))$.
As a result, we can give an explicit formula for the unit of the adjunction $\Curv \leftrightarrow \cMult$ on the image of $\tilde \Omega$.

As in \cite[Section 5]{LodayVallette12}, we view a $\Sym$-module as a functor from the groupoid of finite sets to the category of chain complexes, and a non-counital cooperad as a coalgebra over the reduced tree comonad $\overline \T$ on the category of $\Sym$-modules. 
Given a non-unital cooperad~$\rmC$, we will denote by $\delta_\rmC : \rmC \to \overline \T \rmC$ the structural morphism.
We now introduce some notations.

\begin{notation}
\label{notation coproduct}
    Given a finite subset $A \subset X$ and an operation $v \in \rmP(X)$ in an operad $\rmP$, we denote by $v^A \in (\rmP \hat \vee [\alpha])(X \setminus A)$ the operation obtained from $v$ by grafting $\alpha$ at the entries labelled by elements in $A$. 
\end{notation}

\begin{notation}
\label{notation two-levels}
    Given a non-empty finite subset $A \subset X$, a cooperation $w \in \rmC(X)$ in a cooperad $\rmC$, and a family $\underline x = (x_a)_{a \in A}$ of elements in $\rmC(\emptyset)$, we denote by $w(\underline x) \in (\overline \T \rmC)(X \setminus A)$ the element obtained as follow: first take the tree made of the corolla with leaves labelled by~$X$ together with ``plugs'' (a.k.a\ ``corks'') at the leaves labelled by $A$, then decorate the corolla with $w$ and each plug labelled by $a \in A$ with $x_a$.  
\end{notation}

\begin{notation}
\label{notation delta +}
    Given a cooperation $v$ in a cooperad $\rmC$, we denote by $\delta_\rmC^+(v) \in \overline \T \rmC$ the summand of $\delta_\rmC(v)$ consisting of trees which have at least one cooperation of non-zero arity outside of the first level.
\end{notation}

Let $(\rmC,\form) \in \coCurv$ and $v \in \rmC$. 
Observe that there always exists a set of triples $(w, A, \underline x)$ where $w \in \rmC(X)$, $A \subset X$, $\underline x \in \rmC(\emptyset)^A$, such that 
\[ \delta_\rmC(v) = v + \sum_{w, A, \underline x} w(\underline x) + \delta_\rmC^+(v). \]

\begin{thm}
\label{explicit formula unit intro}
    The unit $\left( \Omega \rmC, \Omega \ker(\form) \right) \xrightarrow{\eta_{\tilde \Omega (\rmC, \form)}} \left( \cTw( \Omega \rmC), \Omega (\rmC) \vee [\alpha] \right)$ of the adjunction $\Curv \leftrightarrow \cMult$ is explicitly given, for $v \in \rmC$, by
    \[\eta_{\tilde \Omega (\rmC, \form)}(s^{-1}v) = s^{-1}v + \form(v) \kappa + \form(d_\rmC v) \alpha + \sum_{w, A, \underline x} \left(\prod_{a \in A} \form(x_a)\right) (s^{-1} w)^A. \]
\end{thm}

\vspace{2mm}

We can then deduce corresponding results for the adjunction $\Twist \leftrightarrow \Mult$.
Let $\rmI$ be the $\Sym$-module defined by $\rmI(1) \eqdef \mathbf{k}$ and $\rmI(n) \eqdef 0$ for $n \ne 1$.
Note that the product of $\Sym$-modules $\rmO \oplus \rmI$ is the $\Sym$-module given by $(\rmO \oplus \rmI)(0) = (\rmO \oplus \rmI)(1) = \mathbf{k}$ and $(\rmO \oplus \rmI)(n) = 0$ for $n \geq 2$.
Recall that, given a cooperad $\rmC$, a map $\xi : \rmC \to \rmI$ is called a counit if $(\xi \hat \circ \id_\rmC) \circ \Delta = (\id_\rmC \hat \circ \xi) \circ \Delta = \id_\rmC$, where $\hat \circ$ is the completed composite product and $\Delta$ is the full cocomposition in $\rmC$. 

\begin{definition}
\label{coTwist} 
    We denote by $\coTwist$ the category defined as follows
    \begin{itemize}[leftmargin=*]
    \item[--] 
    Objects are triples $(\rmC, \counit, \form)$ where $(\rmC, \form) \in \coCurv$ and $\counit : \rmC \to \rmI$ is a $\Sym$-module map of degree $0$ such that
    \begin{enumerate}
        \item we have $\counit (\mu_1^*) = 1$,
        \item\label{decomposition condition} the map $\counit : \rmC \to \rmI$ is a counit, 
        and $\ker(\counit)$ is preserved by the differential $d_\rmC$.
    \end{enumerate}
    \item[--] 
    Morphisms are maps $f: (\rmC, \form) \to (\rmC', \form')$ in $\coCurv$ such that $\counit' \circ f = \counit$.
    \end{itemize}
\end{definition}

We show (\cref{cobar of coTwist}) that the cobar construction induces a well defined functor $\overline \Omega : \coTwist \to \Twist, \, (\rmC, \counit, \form) \mapsto (\Omega \rmC, \Omega \ker(\form \oplus \counit))$.

Let $(\rmC, \counit, \form)$ be an object of $\coTwist$ and $v$ be an element of $\ker(\form \oplus \counit)$. 
If $v$ is of arity $0$, we write $\delta_\rmC(v) = v + \mu_1^*(v) + \sum_{w, \underline x} w(\underline x)$, where every $w$ appearing in the sum is in $\ker(\counit)$.
If $v$ is of positive arity, we write $\delta_\rmC(v) = v + \sum_{w, A, \underline x} w(\underline x) + \delta_\rmC^+(v)$, where every $w$ appearing in the sum is in $\ker(\counit)$. 
Note that $\delta_\rmC(v)$ is of this form by assumption on objects of $\coTwist$. 

\begin{thm}
\label{explicit formula unit Twist intro}
    The unit $\left( \Omega \rmC, \Omega \ker(\form \oplus \counit) \right) \xrightarrow{\eta_{\overline \Omega (\rmC, \counit, \form)}} \left( \cTw(\Omega \ker(\form \oplus \counit))^+, \Tw(\Omega \ker(\form \oplus \counit)) \right)$
    of the adjunction $\Twist \leftrightarrow \Mult$ is explicitly given, for $v \in \ker(\form \oplus \counit)$, by
    \[
    \eta_{\overline \Omega (\rmC, \counit, \form)}(s^{-1}v) = s^{-1}v + \form(d_\rmC v) \alpha + \sum_{w, A, \underline x} 
    \left(\prod_{a \in A} \form(x_a)\right) (s^{-1} w)^A. 
    \]
\end{thm}

We will show in \cref{section:extendable-operads,sec:non-homotopy-algebras} that the ``extendable'' operads of Dotsenko--Shadrin--Vallette \cite[Section~4.5]{DotsenkoShadrinVallette23} correspond to objects of $\coTwist$, hence determine $\Tw$-coalgebra structures by the above explicit formula (in fact, recovering a formula which appeared already in \cite{DotsenkoShadrinVallette23}, though not there identified as a $\Tw$-coalgebra structure).  This class of operads includes many examples of interest, as we explore in Section \ref{examples}. 
 

\subsection{The adjunction $\coCurv \leftrightarrow \coCom$}

Since the category of cooperads is complete (as the category of coalgebras for the comonad $\overline \T$ on the complete category of $\Sym$-modules), we can talk about the product of two cooperads $\rmC_-$ and $\rmC_+$.
The underlying $\Sym$-module of $\rmC_- \times \rmC_+$ is the sum of trees whose vertices are labelled by elements of either $\rmC_-$ or $\rmC_+$, with the property that if a vertex is labelled by an operation in $\rmC_\pm$ then its neighbours are labelled by operations in $\rmC_\mp$. 
See \cite[Sec.~1.5]{grignou2025} for a detailed presentation of the product of two cooperads.

\vspace{2mm}

Recall that $\rmT$ denotes the operad freely generated by arity zero elements $\alpha, \kurv$ where $\alpha$ has degree $0$ and $\kurv$ has degree $(-1)$, with differential $d_\rmT \alpha = \kurv$.

\begin{notation}
    Given $\rmC \in \coCom$, we denote by $\mu_n^*$ the operation in $\rmC$ which is the image of the generator $\mu_n^*$ in $\uComd$ under the structural morphism $\uComd \to \rmC$. 
    We then consider the following elements in $\rmC \times \rmT^*$
    \[
    \mu_n^{\alpha^*} \eqdef \sum_{k \geq 0} \frac{1}{k!}(\mu_{n+k}^*;(\alpha^*)^k,-).
    \]
\end{notation}

\begin{lemma}
\label{strucural morphism cooperadic adjoint}
    Let $\rmC \in \coCom$ with structural morphism $\psi: \uComd \to \rmC$.
    The unique morphism $\Psi : \uComd \to \rmC \times \rmT^*$ satisfying $\pi_\rmC \circ \Psi = \psi$ and $\pi_{\rmT^*} \circ \Psi = \ev_{\mu_0}(-) \alpha^*$ is given by $\mu_0^* \mapsto (\alpha^* + \mu_0^{\alpha^*})$ and $\mu_n^* \mapsto \mu_n^{\alpha^*}$ for $n \geq 1$.
\end{lemma}
\begin{proof}  
    As explained in \cite[Section 1.5]{grignou2025}, the complete formula for $\Psi$ is obtained as follows: first apply the $\overline \T$-coalgebra structural morphism $\delta_{\uCom^*}$ to get an element of $\overline \T \uCom^*$, then take the sum of the elements obtained by
    \begin{enumerate}
        \item applying $\pi_\rmC \circ \Psi$ to the root vertex, then $\pi_{\rmT^*} \circ \Psi$ to the vertices just above, then $\pi_\rmC \circ \Psi$ to the vertices just above the previous ones, etc,
        \item applying $\pi_{\rmT^*} \circ \Psi$ to the root vertex, then $\pi_\rmC \circ \Psi$ to the vertices just above, then $\pi_{\rmT^*} \circ \Psi$ to the vertices just above the previous ones, etc.
    \end{enumerate}
    Since $\pi_{\rmT^*} \circ \Psi$ is zero on all operations of positive arity, it is enough to consider elements in the image of $\delta_{\uCom^*}$ which have only arity zero operations outside of the root vertex.
    The result follows since $\delta_{\uCom^*}(\mu_n^*) = \sum_{k \geq 0} \frac{1}{k!} \mu_{k+n}^* \otimes ((\mu_0^*)^{\otimes k}, -)+\delta_{\uCom^*}^+(\mu_n^*)$.
\end{proof}

\begin{lemma}
    Given $\rmC \in \coCom$, the pair $(\rmC \times \rmT^*,\ev_{\alpha})$ is an object of $\coCurv$.
\end{lemma}
\begin{proof}
    Using notations of \cref{strucural morphism cooperadic adjoint}, we have to check that $\ev_{\alpha} \circ \Psi = \ev_{\mu_0}$.
    This follows since $\ev_{\alpha}$ is only non-zero on $\rmT^*$ and $\pi_{\rmT^*} \circ \Psi = \ev_{\mu_0}(-) \alpha^*$.
\end{proof}

\begin{theorem}
\label{cooperadic adjunction}
    The forgetful functor $F : \coCurv \to \coCom$ admits a right adjoint~$G$, given on objects by $\rmC \mapsto (\rmC \times \rmT^*, \ev_{\alpha})$ and on morphisms by $f \mapsto f \times 1_{\rmT^*}$.
    The counit $\varepsilon : FG \to 1$ is the projection $\rmC \times \rmT^* \to \rmC$.
    The unit $\eta : 1 \to GF$ is characterized by 
    \begin{align*}
        \pi_\rmC \circ \eta & \eqdef \id_\rmC, \\
        \pi_{\rmT^*} \circ \eta & \eqdef \form(-) \alpha^* + (\form \circ d_\rmC)(-) \kappa^*.
    \end{align*}
\end{theorem}

\begin{proof}
    Checking the triangle identities is straightforward: the composition 
    \begin{equation*}
        \begin{matrix}
            \rmC & \overset{F(\eta_{(\rmC, \form)})}{\longrightarrow} & (\rmC \times \rmT^*, \ev_{\alpha}) & \overset{\varepsilon_{F(\rmC, \form)}}{\longrightarrow} & \rmC
        \end{matrix}
    \end{equation*}
    is the identity on $F(\rmC, \form)$, while the composition
    \begin{equation*}
        \begin{matrix}
            (\rmC \times T_0^*, \ev_{\alpha_0}) & \overset{\eta_{G(\rmC)}}{\longrightarrow} & (\rmC \times \rmT_0^* \times \rmT_1^*, \ev_{\alpha_1}) & \overset{G\left(\varepsilon_{\rmC}\right)}{\longrightarrow} & (\rmC \times \rmT_1^*, \ev_{\alpha_1})
        \end{matrix}
    \end{equation*}
    is the identity on $G(\rmC)$.
\end{proof}

\vspace{2mm}

\subsection{Proof of \cref{explicit formula unit intro}}

\begin{lemma-def}
\label{cobar of coCurv}
    The cobar construction induces a well defined functor
    \[\tilde \Omega : \coCurv \to \Curv, \quad (\rmC, \form) \mapsto (\Omega \rmC, \Omega \ker(\form)). \]
\end{lemma-def}

\begin{proof}
    Given $(\rmC, \form) \in \coCurv$ with structural morphism $\uCom^* \xrightarrow{\psi} \rmC$, we have the natural morphism $\cLinf = \Omega \uCom^* \xrightarrow{\Omega \psi} \Omega \rmC$. 
    Moreover, we have that
    \begin{itemize}[leftmargin=*]
        \item[--] 
        The natural map $[\ell_0] \vee \Omega \ker(\form) \to \Omega \rmC$ (given by $\Omega \psi$ on $\ell_0$) is an isomorphism of non-dg operads.
        Indeed we have that $\form(\psi(\mu_0^*)) = 1 \ne 0$ by assumption on objects of $\coCurv$, and therefore the map $(\psi(\mu_0^*) \oplus \ker(\form)) \to \rmC$ is an isomorphism of $\Sym$-modules.
        The claim follows by definition of the cobar construction.
        \item[--] 
        $\Omega \ker(\form)$ contains the image of the $\ell_i$ for $i > 0$; this follows by assumption on $\form$. 
        \item[--] 
        $d_{\Omega \rmC}(\Omega \ker(\form))$ consists of elements of degree at most one in $\ell_0$.
        Indeed, the differential $d_{\Omega \rmC}$ is the sum of the differential $d_{\rmC}$ of $\rmC$, and the cobar differential $d_{\Delta}$ which is given by the infinitesimal decomposition map in $\rmC$.
    \end{itemize} 
    Moreover, for a morphism $(\rmC, \form) \to (\rmC', \form')$ in $\coCurv$, the morphism $\Omega \rmC \to \Omega \rmC'$ is in $\Curv$, since it is a morphism over $\cLinf$ carrying $\Omega \ker(\form) \to \Omega \ker(\form')$. 
\end{proof}

\begin{remark}
    In \cref{bar-cobar}, we discuss when the bar construction of a $\cLinf$-operad can be lifted to an object of $\coCurv$.
\end{remark}

\begin{lemma}
\label{commutativity forgetful cobar}
    The forgetful functors $F : \coCurv \to \coCom$ and $\calL : \Curv \to \cMult$ satisfy $\Omega F = \calL \tilde \Omega$.
\end{lemma}
\begin{proof}
    This is a straightforward verification.
\end{proof}

Our goal is to give a formula for the unit  $\eta^{\calR \calL}_{\tilde \Omega} : \tilde \Omega \to \calR \calL \tilde \Omega$ in $\Curv$ for objects in the image of the cobar construction.

\begin{lemma}
\label{formula unit}
    We have $\eta^{\calR \calL}_{\tilde \Omega} = (\calR \Omega) (\varepsilon^{F G}_{F}) \circ \eta^{\calR \calL}_{\tilde \Omega G F} \circ \tilde \Omega (\eta^{GF})$.
\end{lemma}
\begin{proof}
    We compute
    \begin{align*}
        (\calR \Omega) (\varepsilon^{F G}_{F}) \circ \eta^{\calR \calL}_{\tilde \Omega G F} \circ \tilde \Omega (\eta^{GF}) 
        & = (\calR \Omega) (\varepsilon^{F G}_{F}) \circ (\calR \calL \tilde \Omega) (\eta^{GF}) \circ \eta^{\calR \calL}_{\tilde \Omega} \\
        & = (\calR \Omega) (\varepsilon^{F G}_{F}) \circ (\calR \Omega F) (\eta^{GF}) \circ \eta^{\calR \calL}_{\tilde \Omega} = \eta^{\calR \calL}_{\tilde \Omega}.
    \end{align*}
    Here, the first equality follows from the fact that $\eta^{\calR \calL} : 1 \to \calR \calL$ is a natural transformation, the second equality follows from the equality $\calL \tilde \Omega = \Omega F$ (\cref{commutativity forgetful cobar}), and the last equality follows from the unit-counit relation $\varepsilon^{F G}_{F} \circ F (\eta^{GF}) = \id_{F}$.
\end{proof}

We introduce additional notations similar to \cref{notation coproduct} and \cref{notation two-levels}.

\begin{notation}
    Given disjoint subsets $A, K$ of a set $X$, and an operation $v \in \rmP(X)$ in an operad $\rmP$, we denote by $v^{A, K} \in (\rmP \vee \rmT)(X \setminus (A \cup K))$ the operation obtained from $v$ by grafting $\alpha$, respectively $\kappa$, at the entries labelled by elements in $A$, respectively in $K$. 
    In particular, $v^{A,\emptyset} = v^A$.
\end{notation}

\begin{notation}
    Given disjoint subsets $A, K$ in a set $X$, and a cooperation $w \in \rmC(X)$ in a cooperad $\rmC$, we denote by $w_{A, K} \in (\rmC \times \rmT^*)(X \setminus (A \cup K))$ the cooperation obtained from $w$ by grafting $\alpha^*$, respectively $\kappa^*$, at the entries labelled by elements in $A$, respectively in $K$. 
\end{notation}

\begin{lemma}
\label{unit specific case}
    Given $\rmC \in \coCom$, the morphism in $\Curv$
    \[\eta = \eta^{\calR \calL}_{\tilde \Omega G(\rmC)} : \left( \Omega (\rmC \times \rmT_0^*), \Omega \ker(\ev_{\alpha_0}) \right) \to \left( \cTw(\Omega (\rmC \times \rmT_0^*)), \Omega (\rmC \times \rmT_0^*) \vee [\alpha_1] \right) \]
    is given for $w \in \rmC(n)$ and $A_0, K_0 \subset \{1, \dots, n\}$ disjoint by 
    \[\eta(s^{-1}\alpha_0^*) = s^{-1}\alpha_0^* + \kappa_1, \, \eta(s^{-1}\kappa_0^*) = s^{-1}\kappa_0^* + \alpha_1, \text{ and } \eta(s^{-1} w_{A_0, K_0}) = \sum_{A \subset A_0} (s^{-1} w_{A, K_0})^{A_0 \setminus A}. \]
\end{lemma}

\begin{proof}
    We have to check the following conditions:
    \begin{enumerate}
        \item\label{commutes with differentials} $\eta$ commutes with the differentials,
        \item\label{intertwines mor from cLinf} $\eta$ intertwines the morphisms from $\cLinf$,
        \item\label{preserves sub-operads} $\eta$ preserves the distinguished sub-operads,
        \item\label{composition counit} the composition of $\calL(\eta)$ with the counit is the identity, i.e.\ $\varepsilon^{\calL \calR}_{\calL \tilde \Omega G \rmC} \circ \calL(\eta) = \id_{\calL \tilde \Omega G \rmC}$.
    \end{enumerate}
    Conditions (\ref{commutes with differentials}), (\ref{intertwines mor from cLinf}) and (\ref{preserves sub-operads}) guarantee that $\eta$ is a morphism in $\Curv$, while Condition~(\ref{composition counit}) guarantees that $\eta$ is the unit of the adjunction. 

    Observe that~(\ref{composition counit}) is satisfied because the counit $\varepsilon^{\calL \calR}_{\rmQ} : \rmQ \vee \rmT_1 \to \rmQ$ is the identity on $\rmQ$ and sends $\rmT_1$ to $0$ (see \cref{thm:adjunction-Curv}). 
    Moreover, (\ref{preserves sub-operads}) is satisfied because $\eta$ sends $s^{-1} \ker(\ev_{\alpha_0})$ to $\Omega(\rmC \times \rmT_0^*) \vee [\alpha_1]$.
    Therefore it remains to check (\ref{commutes with differentials}) and (\ref{intertwines mor from cLinf}). 

    We check (\ref{commutes with differentials}).
    On $s^{-1} \rmT_0^*$, it suffices to observe that by definition we have
    \[d_{\rmT_0^*}(\kappa_0^*)=\alpha_0^*, \quad d_{\rmT_1}(\alpha_1)=\kappa_1. \] 
    Now let $w \in \rmC(X)$ and $A_0, K_0 \subset X$ disjoint.
    On the one hand we have
    \[
    d_{\Omega(\rmC) \vee \rmT_1} \left( \eta \left(s^{-1} w_{A_0,K_0} \right) \right)
    = \sum_{A \subset A_0} \left(
    d_{\Omega(\rmC)} (s^{-1} w_{A,K_0})^{A_0 \setminus A, \emptyset}
    + \sum_{i \in A_0 \setminus A} (s^{-1}w_{A,K_0})^{A_0 \setminus (A \cup \{i\}), \{i\}}
    \right)
    \]
    We compute
    \begin{align*}
        d_{\Omega(\rmC)}(s^{-1} w_{A,K_0}) = & \, s^{-1} d_{\rmC \times \rmT_0^*}(w_{A,K_0}) + \sum s^{-1} w^{(1)}_{A^{(1)},K_0^{(1)}} \otimes s^{-1} w^{(2)}_{A^{(2)},K_0^{(2)}} \\
        & + \sum_{i \in A} s^{-1} w_{A \setminus \{i \}, K_0} \otimes_i s^{-1} \alpha_0^* + \sum_{i \in K_0} s^{-1} w_{A, K_0 \setminus \{i \}} \otimes_i s^{-1} \kappa_0^*
    \end{align*}
    (where we used Sweedler's notation for the infinitesimal cocomposition in $\rmC$) and
    \[
    d_{\rmC \times \rmT_0^*}(w_{A,K_0}) = d_\rmC(w)_{A,K_0}
    + \sum_{i \in K_0} w_{A \cup \{i\}, K_0 \setminus \{i\}}.
    \]
    Therefore, we have
    \begin{align}
        d_{\Omega(\rmC) \vee \rmT_1} \left( \eta \left(s^{-1} w_{A_0,K_0} \right) \right) = & \sum_{A \subset A_0} (s^{-1} d_\rmC(w)_{A,K_0})^{A_0 \setminus A, \emptyset} + \sum_{A \subset A_0} \sum_{i \in K_0} (s^{-1} w_{A \cup \{i\}, K_0 \setminus \{i\}})^{A_0 \setminus A, \emptyset} \label{1} \\
        & + \sum_{A \subset A_0} \sum (s^{-1} w^{(1)}_{A^{(1)},K_0^{(1)}} \otimes s^{-1} w^{(2)}_{A^{(2)},K_0^{(2)}})^{A_0 \setminus A, \emptyset} \label{2} \\
        & + \sum_{A \subset A_0} \sum_{i \in A} (s^{-1} w_{A \setminus \{i \}, K_0} \otimes_i s^{-1} \alpha_0^*)^{A_0 \setminus A, \emptyset} \label{3} \\
        & + \sum_{A \subset A_0} \sum_{i \in K_0} (s^{-1} w_{A, K_0 \setminus \{i \}} \otimes_i s^{-1} \kappa_0^*)^{A_0 \setminus A, \emptyset} \label{4} \\
        & + \sum_{A \subset A_0} \sum_{i \in A_0 \setminus A} (s^{-1}w_{A,K_0})^{A_0 \setminus (A \cup \{i\}), \{i\}}. \label{5} 
    \end{align}
    
    On the other hand we have to compute $\eta(d_{\Omega(\rmC)}(s^{-1}w_{A_0, K_0}))$.
    Observe that we computed $d_{\Omega(\rmC)}(s^{-1}w_{A_0, K_0})$ above and that it consists of four terms.
    Applying $\eta$ to each of these four terms we get
    \begin{align}
        \eta \left( s^{-1} d_{\rmC \times \rmT_0^*}(w_{A,K_0}) \right) & = \eta \left(s^{-1}d_\rmC(w)_{A_0,K_0} \right) + \eta \left( \sum_{i \in K_0} s^{-1} w_{A_0 \cup \{i\}, K_0 \setminus \{i\}} \right) \nonumber \\
        & = \sum_{A \subset A_0} \left(s^{-1} d_{\rmC}(w)_{A,K_0} \right)^{A_0 \setminus A, \emptyset} + \sum_{i \in K_0} \sum_{A' \subset A_0 \cup \{i \}} \left(s^{-1}w_{A',K_0 \setminus \{i \}} \right)^{(A_0 \cup \{i \}) \setminus A', \emptyset} \label{1bis}
    \end{align}
    and 
    \begin{align}
        \eta & \left( \sum s^{-1} w^{(1)}_{A_0^{(1)},K_0^{(1)}} \otimes s^{-1} w^{(2)}_{A_0^{(2)},K_0^{(2)}} \right) \nonumber \\
        & = \sum \left( \sum_{A^{(1)} \subset A_0^{(1)}} (s^{-1} w^{(1)}_{A^{(1)},K_0^{(1)}})^{A_0^{(1)} \setminus A^{(1)}} \right) \otimes \left( \sum_{A^{(2)} \subset A_0^{(2)}} (s^{-1} w^{(2)}_{A^{(2)},K_0^{(2)}})^{A_0^{(2)} \setminus A^{(2)}} \right) \label{2bis}
    \end{align}
    and
    \begin{align}
        \eta \left(\sum_{i \in A_0} s^{-1}w_{A_0 \setminus \{i \}, K_0} \otimes_i s^{-1} \alpha_0^* \right) = & \sum_{i \in A_0} \sum_{A'' \subset A_0 \setminus \{i \}} 
        \left(s^{-1}w_{A'',K_0} \right)^{(A_0 \setminus \{i \}) \setminus A'', \emptyset} \otimes_i s^{-1} \alpha_0^* \label{3bis} \\
        & + \sum_{i \in A_0} \sum_{A'' \subset A_0 \setminus \{i \}} \left(s^{-1} w_{A'',K_0 \setminus \{i\}} \right)^{(A_0 \setminus \{i \}) \setminus A'', \{i \}} \label{5bis}
    \end{align}
    and
    \begin{align}
        \eta \left(\sum_{i \in K_0} s^{-1}w_{A_0, K_0 \setminus \{i\}} \otimes_i s^{-1} \kappa_0^* \right) = & \sum_{i \in K_0} \sum_{A \subset A_0} 
        \left(s^{-1}w_{A,K_0 \setminus \{i\}}\right)^{A_0 \setminus A, \emptyset} \otimes_i s^{-1} \kappa_0^* \label{4bis} \\
        & + \sum_{i \in K_0} \sum_{A \subset A_0} \left(s^{-1}w_{A,K_0 \setminus \{i\}}\right)^{(A_0 \setminus A) \cup \{i \}, \emptyset}. \label{1bisbis}
    \end{align}
    
    To check that $d_{\Omega(\rmC) \vee \rmT_1} \left( \eta \left(s^{-1} w_{A_0,K_0} \right) \right) = \eta(d_{\Omega(\rmC)}(s^{-1}w_{A_0, K_0}))$, it suffices to observe the following identifications: (\ref{1}) with~(\ref{1bis}) and~(\ref{1bisbis}), (\ref{2}) with~(\ref{2bis}), (\ref{3}) with~(\ref{3bis}), (\ref{4}) with~(\ref{4bis}), and (\ref{5}) with~(\ref{5bis}). 
    Therefore (\ref{commutes with differentials}) is satisfied.

    We now check $(\ref{intertwines mor from cLinf})$.
    Using the notations of \cref{strucural morphism cooperadic adjoint} and letting $\lambda_n \eqdef s^{-1} \mu_n^{\alpha_0^*}$, we have to show that 
    \[\eta(s^{-1} \alpha_0^* + \lambda_0) = \kappa_1 + s^{-1} \alpha_0^* + \lambda_0^{\alpha_1} \text{ and } \eta(\lambda_n) = \lambda_n^{\alpha_1} \]
    for $n \geq 1$.
    We compute for $n \geq 0$: 
    \begin{align*}
        \eta(\lambda_n) =  \sum_{k \geq 0} \frac{1}{k!} \eta(s^{-1} (\mu_{n+k}^*)_{\{1, \dots, k\}, \emptyset})
        & = \sum_{k \geq 0} \frac{1}{k!} \sum_{A \subset \{1, \dots, k\}} (s^{-1} (\mu_{n+k}^*)_{A, \emptyset})^{\{1, \dots, k\} \setminus A, \emptyset} \\
        & = \sum_{i, j \geq 0} \frac{1}{(i+j)!} \sum_{\substack{A \subset \{1, \dots, i+j\} \\ |A|=i}} (s^{-1} (\mu_{n+i+j}^*)_{A, \emptyset})^{\{1, \dots, i+j\} \setminus A, \emptyset} \\ 
        & = \sum_{i, j \geq 0} \frac{1}{(i+j)!} \binom{i+j}{i} (s^{-1} (\mu_{n+i+j}^*)_{\{1, \dots, i\}, \emptyset})^{\{i+1, \dots, i+j\}, \emptyset} \\ 
        & = \sum_{j \geq 0} \frac{1}{j!} \left( \sum_{i \geq 0} \frac{1}{i!} s^{-1} (\mu_{n+i+j}^*)_{\{1, \dots, i\}, \emptyset} \right)^{\{i+1, \dots, i+j\}, \emptyset} \\
        & = \sum_{j \geq 0} \frac{1}{j!} \lambda_{j+n}(\alpha_1^j, -) = \lambda_n^{\alpha_1}.
    \end{align*}
    This concludes the proof.
\end{proof}

In the following Lemma, we will denote by $\delta_{(-)}$ the Kronecker delta function: $\delta_{(-)} = 1$ if the condition $(-)$ is satisfied, and $\delta_{(-)} = 0$ otherwise. It should not be confused with the structural morphism $\delta_\rmC : \rmC \to \overline{\mathbb{T}}\rmC$.

\begin{lemma}
\label{composition counit unit}
    Given $\rmC \in \coCom$, the morphism in $\Curv$
    \[f_\rmC \eqdef (\calR \Omega) (\varepsilon^{F G}_{\rmC}) \circ \eta^{\calR \calL}_{\tilde \Omega G \rmC} : \left( \Omega (\rmC \times \rmT_0^*), \Omega \ker(\ev_{\alpha_0}) \right) \to \left( \cTw(\Omega \rmC), \Omega(\rmC) \vee [\alpha_1] \right)\]
    is given for $w \in \rmC(X)$ and $A, K \subset X$ disjoint by 
     \[f_\rmC(s^{-1}\alpha_0^*) = \kappa_1, \, f_\rmC(s^{-1}\kappa_0^*) = \alpha_1, \text{ and } f_\rmC(s^{-1} w_{A, K}) = (s^{-1} w)^A \delta_{K = \emptyset}. \]
\end{lemma}
\begin{proof}
    This follows from \cref{unit specific case} and the fact that $\varepsilon^{F G}_{\rmC} : \rmC \times \rmT_0^* \to \rmC$ is the projection (see \cref{cooperadic adjunction}).
\end{proof}

We are now ready to compute the unit of the adjunction $\Curv \leftrightarrow \cMult$ on the image of $\tilde \Omega : \coCurv \to \Curv$.

\begin{proof}
[Proof of \cref{explicit formula unit intro}]
    According to \cref{formula unit} and \cref{composition counit unit}, we have the formula $\eta^{\calR \calL}_{\tilde \Omega \rmC} = f_\rmC \circ \tilde \Omega(\eta^{GF}_{(\rmC, \form)})$, where
    \[f_\rmC(s^{-1}\alpha_0^*) = \kappa_1, \, f_\rmC(s^{-1}\kappa_0^*) = \alpha_1, \text{ and } f_\rmC(s^{-1} w_{A, K}) = (s^{-1} w)^A \delta_{K = \emptyset} \]
    for $w \in \rmC(X)$ and $A, K \subset X$ disjoint.

    Moreover, according to \cref{cooperadic adjunction}, the unit $\eta^{GF}_{(\rmC, \form)} : (\rmC, \form) \to (\rmC \times \rmT_0^*, \ev_{\alpha_0})$ is characterized by 
    \begin{align*}
        \pi_\rmC \circ \eta^{GF}_{(\rmC, \form)} & \eqdef \id_\rmC, \\
        \pi_{\rmT_0^*} \circ \eta^{GF}_{(\rmC, \form)} & \eqdef \form(-) \alpha_0^* + (\form \circ d_\rmC)(-) \kappa_0^*.
    \end{align*}
    The complete formula for $\eta^{GF}_{(\rmC, \form)}$ is obtained as follows: first apply the $\overline \T$-coalgebra structural morphism $\delta_\rmC : \rmC \to \overline \T \rmC$, then take the sum of 
    \begin{enumerate}
        \item applying $\pi_\rmC \circ \eta^{GF}_{(\rmC, \form)}$ to the root vertex, then $\pi_{\rmT_0^*} \circ \eta^{GF}_{(\rmC, \form)}$ to the vertices just above, then $\pi_\rmC \circ \eta^{GF}_{(\rmC, \form)}$ to the vertices just above the previous ones, etc,
        \item applying $\pi_{\rmT_0^*} \circ \eta^{GF}_{(\rmC, \form)}$ to the root vertex, then $\pi_\rmC \circ \eta^{GF}_{(\rmC, \form)}$ to the vertices just above, then $\pi_{\rmT_0^*} \circ \eta^{GF}_{(\rmC, \form)}$ to the vertices just above the previous ones, etc.
    \end{enumerate}
    Since $\pi_{\rmT_0^*} \circ \eta^{GF}_{(\rmC, \form)}$ is zero on all operations of positive arity, it is enough to consider elements in the image of $\delta_{\rmC}$ which have only arity zero operations outside of the root vertex.

    Given $v \in \rmC$ with $\delta_\rmC(v) = v + \sum_{w, A, \underline x} w(\underline x) + \delta_\rmC^+(v)$, we then compute 
    \begin{align*}
        \eta&^{\calR \calL}_{\tilde \Omega \rmC}(s^{-1}v) = (f_\rmC \circ \tilde \Omega(\eta^{GF}_{(\rmC, \form)}))(s^{-1}v) \\
        & = f_\rmC \left( s^{-1}v + \form(v) s^{-1} \alpha_0^* + \form(d_\rmC v) s^{-1} \kappa_0^* + \sum_{w, A, \underline x} \sum_{K \subset A} \left(\prod_{a \in A \setminus K, b \in K} \form(x_a) \form(d_\rmC x_b)\right) (s^{-1} w_{A,K}) \right) \\
        & = s^{-1}v + \form(v) \kappa_1 + \form(d_\rmC v) \alpha_1 + \sum_{w, A, \underline x} \left(\prod_{a \in A} \form(x_a)\right) (s^{-1} w)^A.
    \end{align*}
    This completes the proof.
\end{proof}


\vspace{2mm}

\subsection{Proof of \cref{explicit formula unit Twist intro}}

\begin{lemma-def}
\label{cobar of coTwist}
    The cobar construction induces a well defined functor
    \[\overline \Omega : \coTwist \to \Twist, \quad (\rmC, \counit, \form) \mapsto (\Omega \rmC, \Omega \ker(\form \oplus \counit)). \]
\end{lemma-def}

\begin{proof}
    Given $(\rmC, \counit, \form) \in \coTwist$ with structural morphism $\uCom^* \xrightarrow{\psi} \rmC$, we have the natural morphism $\cLinf = \Omega \uCom^* \xrightarrow{\Omega \psi} \Omega \rmC$. 
    Moreover we have that
    \begin{itemize}[leftmargin=*]
        \item [--]
        The natural map $[\ell_0] \vee [\ell_1] \vee \Omega \ker(\form \oplus \counit) \to \Omega \rmC$ (given by $\Omega \psi$ on $\ell_0$ and $\ell_1$) is an isomorphism of (non-dg) operads.
        Indeed we have that $\form(\psi(\mu_0^*)) = \counit(\psi(\mu_0^*)) = 1 \ne 0$ by assumption on objects of $\coTwist$, and therefore the map 
        \[(\psi(\mu_0^*) \oplus \psi(\mu_1^*) \oplus \ker(\form \oplus \counit)) \to \rmC \]
        is an isomorphism of $\Sym$-modules.
        The claim follows by definition of the cobar construction.
        \item[--] 
        $\Omega \ker(\form \oplus \counit)$ contains the image of the $\ell_i$ for $i > 1$; this follows by assumption on $\form$ and $\counit$. 
        \item[--]  
        $(d_{\Omega \rmC} + [\ell_1,-])(\Omega \ker(\form \oplus \counit))$ consists of elements of degree zero in $\ell_1$ and at most one in $\ell_0$.
        This follows from condition~(\ref{decomposition condition}) in \cref{coTwist} since the differential $d_{\Omega \rmC}$ is the sum of the differential $d_{\rmC}$ of $\rmC$ and the cobar differential which is given by the infinitesimal decomposition map in $\rmC$.
    \end{itemize} 
    Moreover, for a morphism $(\rmC, \counit, \form) \to (\rmC', \counit', \form')$ in $\coTwist$, the morphism $\Omega \rmC \to \Omega \rmC'$ is in $\Twist$, since it is a morphism over $\cLinf$ carrying $\Omega \ker(\form \oplus \counit) \to \Omega \ker(\form' \oplus \counit')$. 
\end{proof}

\begin{lemma}
\label{forgetful functors commute with cobar}
    The forgetful functors $I : \coTwist \to \coCurv$ and $\calI : \Twist \to \Curv$ satisfy $\tilde \Omega I = \calI \overline \Omega$.
\end{lemma}
\begin{proof}
    This is a straightforward verification.
\end{proof}

Recall (\cref{thm:adjoint}) that the adjunction $(\forget \dashv \free)$ between $\Twist$ and $\Mult$ comes from composition of the following three adjunctions 
\[ 
\begin{tikzcd}
    \Twist \ar[r, bend left, "\calI"] & \Curv \ar[r, bend left, "\calL"] \ar[l, bend left, "\calJ"] & \cMult \ar[l, bend left, "\calR"] \ar[r, bend left, "\pi_!"] & \Mult \ar[l, bend left, "\pi^!"].
\end{tikzcd}
\]

We are now ready to compute the unit of the adjunction $\Twist \leftrightarrow \Mult$ on the image of $\overline \Omega : \coTwist \to \Twist$.

\begin{proof}
[Proof of \cref{explicit formula unit Twist intro}]
    The formula for the unit of a composition of adjunctions gives 
    \[\eta^{\free \forget} = (\calJ \calR)(\eta^{\pi^! \pi_!}_{\calL \calI}) \circ \calJ (\eta^{\calR \calL}_\calI) \circ \eta^{\calJ \calI}. \]
    In particular, using \cref{forgetful functors commute with cobar}, we get $\eta^{\free \forget}_{\overline \Omega} = (\calJ \calR)(\eta^{\pi^! \pi_!}_{\Omega F I}) \circ \calJ (\eta^{\calR \calL}_{\tilde \Omega I}) \circ \eta^{\calJ \calI}_{\overline \Omega}$.
    
    According to \cref{thm:adjunction-Twist-Curv}, the unit $\eta^{\calJ \calI} : (\rmQ, \rmQ_{0,0}) \to (\rmQ^+, \rmQ_{0,0} \vee [\ell_1])$ is given by 
    \[\eta^{\calJ \calI}_{|\rmQ_{0,0}} = \id_{\rmQ_{0,0}}, \quad \eta^{\calJ \calI}(\ell_0) = \ell_0, \quad \eta^{\calJ \calI}(\ell_1) = \diff + \ell_1. \]
    Moreover the unit $\eta^{\pi^! \pi_!}_\rmQ : \rmQ \to \rmQ / (\ell_0, \ell_1)$ is the projection $\pi_\rmQ : \rmQ \to \rmQ / (\ell_0, \ell_1)$.
    The result follows from the explicit formula for the unit $\eta_{\tilde \Omega}^{\calR \calL}$ (\cref{explicit formula unit intro}).
\end{proof}


\subsection{Extendable operads}
\label{section:extendable-operads}

Let us recall from \cite[Sec.~4.5]{DotsenkoShadrinVallette23} the definition of an extendable operad. 
Let $\rmP=\rmP(E,R)$ be an operad with arity-wise finite dimensional $\Sym$-module of generators~$E$ such that $E(0)=0$ and quadratic-linear relations~$R$.
We assume it has~$0$ differential.

Suppose given a non-trivial $\Sym_2$-equivariant map $\chi: E(2) \to \mathbf{k}$, where the target is endowed with the trivial action.   
Consider the space of relations $R_\chi \subset \calT(E \oplus \mathbf{k} u)$  generated by 
\begin{enumerate}
\item $\mu \circ_1 u -\chi(\mu) \id$ and $\mu \circ_2 u -\chi(\mu) \id$ for all $\mu \in E(2)$, and
\item all the composites $\mu \circ_i u$ of elements $\mu$ of $E(n)$ with at least one element $u$, for $n \neq 2$.
\end{enumerate}
Then, we form the operad
\begin{equation}
\label{eq:unital-quotient}
    u_\chi \rmP \eqdef (\rmP \vee [u])/(R_\chi).
\end{equation}

\begin{definition}
    We say that $u_\chi \rmP$ is a \emph{unital extension} of $\rmP$ if the natural inclusion $\rmP \to u_\chi \rmP$ is a monomorphism. 
    In this case we say that $\rmP$ is \emph{extendable} and we write $\rmP \hookrightarrow u_\chi \rmP$ for the unital extension.
\end{definition}

An essential property of extendable operads is that they live over the unital commutative operad~$\uCom$.

\begin{lemma}
\label{extendable-over-uCom}
    Let $\rmP \hookrightarrow u_\chi \rmP$ be a unital extension of an extendable operad~$\rmP$.
    Then, the morphism of operads $\psi : u_\chi \rmP \to \uCom$ defined on generators by
    \[u \mapsto \mu_0, \quad \id \mapsto \mu_1, \quad \nu \mapsto \chi(\nu) \mu_2, \quad \nu' \mapsto 0 \]
    for $\nu \in E(2)$ and $\nu' \in E(n)$ with $n \geqslant 3$ is well-defined.
\end{lemma} 

\begin{proof}
    We need to show that for any relation $r=0$ in $u_\chi \rmP$, we have $\psi(r)=0$ in $\uCom$.
    For the relations $R_\chi$, this is clear. 
    Now consider a relation $r=0$ in $R$.
    Without loss of generality, we can assume that it contains only generators of arity $2$; otherwise, the associated monomials are sent to $0$ by $\psi$. 
    Since $\rmP$ is quadratic-linear, $r=0$ is of the form $\sum_{\sigma}(\mu \circ_1 \nu)^{\sigma} + \sum_{\tau}(\alpha \circ_2 \beta)^{\tau} =0$
    for some $\mu,\nu,\alpha,\beta \in E(2)$ and shuffle permutations $\sigma,\tau$ in~$\Sym_3$.
    Composing this relation with~$u$ twice gives 
    \begin{align*}
    (r \circ_1 u)\circ_1 u
    &=
    \sum_{\sigma}\left(((\mu \circ_1 \nu)\circ_i u)\circ_j u \right) + 
    \sum_{\tau}\left(((\alpha \circ_2 \beta)\circ_k u)\circ_l u \right)\\
    &=\sum_{\sigma}\chi(\mu)\chi(\nu)\id+\sum_{\tau}\chi(\alpha)\chi(\beta)\id=0.
    \end{align*}
    Requiring $\rmP$ to be extendable is equivalent to requiring that the underlying $\Sym$-module of~$u_\chi\rmP$ is isomorphic to $\rmP \oplus \mathbf{k}u$.
    Therefore, imposing the relations $R_\chi$ in the definition of $u_\chi \rmP$~\eqref{eq:unital-quotient} do not create any new relations between elements of $\rmP$. 
    In particular, $u_\chi \rmP$ is not the trivial operad, and we must have $\sum_{\sigma}\chi(\mu)\chi(\nu)+\sum_{\tau}\chi(\alpha)\chi(\beta)=0$.
    The image of the relation $r$ under the map~$\psi$ then becomes
    \begin{eqnarray*}
        \psi(r) 
        &=& \sum_{\sigma}\psi(\mu \circ_1 \nu)^{\sigma} + \sum_{\tau}\psi(\alpha \circ_2 \beta)^{\tau}  \\
        &=& \sum_{\sigma}(\psi(\mu) \circ_1 \psi(\nu))^{\sigma} + \sum_{\tau}(\psi(\alpha) \circ_2 \psi(\beta))^{\tau}  \\
        &=& \sum_{\sigma}(\chi(\mu)\chi(\nu)\mu_2 \circ_1 \mu_2)^{\sigma} + \sum_{\tau}(\chi(\alpha)\chi(\beta)\mu_2\circ_2 \mu_2)^{\tau}  \\
        &=&\sum_{\sigma}\chi(\mu)\chi(\nu)\mu_2 \circ_1 \mu_2 + \sum_{\tau}\chi(\alpha)\chi(\beta)\mu_2\circ_2 \mu_2  \\
        &=& \left(\sum_{\sigma}\chi(\mu)\chi(\nu)+\sum_{\tau}\chi(\alpha)\chi(\beta)\right)\mu_2 \circ_1 \mu_2  = 0,
    \end{eqnarray*}
    as desired.
\end{proof}

This immediately gives a lift of $(u_\chi \rmP)^{*}$ to $\coCurv$ and $\coTwist$.

\begin{samepage}
\begin{theorem}
\label{dual of Ext is in coCurv intro}
    Any unital extension $\rmP \hookrightarrow u_\chi \rmP$ of an extendable operad~$\rmP$ defines objects $((u_\chi \rmP)^*, \ev_u)$ in $\coCurv$ and $((u_\chi \rmP)^*, \ev_{\id}, \ev_u)$ in $\coTwist$, with both structural morphisms being the linear dual of the operad morphism $u_\chi \rmP \to \uCom$.
\end{theorem}
\end{samepage}

Note that the cooperad $(u_\chi \rmP)^*$ is generated in arity $0$ by a single element $u^*$; therefore given a cooperation $v \in (u_\chi \rmP)^*$, we can write $\delta_{(u_\chi \rmP)^*}(v) = v + \sum_{w, A} w(u^*_A) + \delta_{(u_\chi \rmP)^*}^+(v)$, where $u^*_A$ is the constant family $a \mapsto u^*$ and $\delta_{(u_\chi \rmP)^*}^+(v) \in \overline \T \rmC$ is a sum of trees which have at least one cooperation of non-zero arity outside of the first level (\cref{notation delta +}). 

\begin{corollary}
\label{DSV formula intro}
    Let $\rmP$ be a Koszul operad whose Koszul dual $\rmP^!$ admits a unital extension $\rmP^! \hookrightarrow u_\chi \rmP^!$, and set $\cPinf^\chi \eqdef \Omega (u_\chi \rmP^!)^*$.
    Then, the lift $((u_\chi \rmP^!)^*, \ev_u)$ of $(u_\chi \rmP^!)^*$ to $\coCurv$ induces 
    \begin{enumerate}
        \item a morphism $\cLinf \xrightarrow{\Omega(\psi^*)} \cPinf^\chi$, where $\psi : u_\chi \rmP^! \to \uCom$ is given in \cref{extendable-over-uCom},
        \item \label{cTw-coalgebra-extendable} a $\cTw$-coalgebra structure 
        \[ \eta_{\cPinf^\chi} : \cPinf^\chi \to \cTw(\cPinf^\chi), \quad s^{-1}v \mapsto s^{-1}v + \ev_u(v) \kappa + \sum_{w, A} (s^{-1} w)^A. \]
    \end{enumerate} 
\end{corollary}

\begin{proof}
    According to \cref{dual of Ext is in coCurv intro}, a unital extension $\rmP^! \hookrightarrow u_\chi \rmP^!$ defines, in particular, an object $((u_\chi \rmP^!)^*, \ev_u) \in \coCurv$ with an explicit morphism $\psi^* : \uCom^* \to (u_\chi \rmP^!)^*$ given by~$\chi$.
    Using \cref{cobar of coCurv}, we get an object $(\Omega (u_\chi \rmP^!)^*, \ker(\ev_u)) \in \Curv$, with morphism $\Omega(\psi^*) : \cLinf \to \Omega (u_\chi \rmP^!)^* = \cPinf^\chi$.
    Now, \cref{cTw comonadicity} implies that the unit of the adjunction $\Curv \leftrightarrow \cMult$ for $(\cPinf^\chi, \ker(\ev_u))$ induces a $\cTw$-coalgebra structure $\eta_{\cPinf^\chi} : \cPinf^\chi \to \cTw(\cPinf^\chi)$.
    The explicit formula for $\eta_{\cPinf^\chi}$ is given by \cref{explicit formula unit intro}.
\end{proof}

Recall that a quadratic operad $\rmP$ is naturally augmented, therefore $\rmP^*$ is coaugmented and we can consider its coaugmentation ideal $\overline{\rmP^*}$.
Given a unital extension $\rmP \hookrightarrow u_\chi \rmP$, we have~$\overline{\rmP^*} \cong \ker(\ev_u \oplus \ev_{\id})$. 
Hence the following result.

\begin{samepage}
\begin{corollary}
\label{extendable operads are Tw-coalgebras intro}
    Let $\rmP$ be a Koszul operad whose Koszul dual $\rmP^!$ admits a unital extension $\rmP^! \hookrightarrow u_\chi \rmP^!$ and set $\Pinf \eqdef \Omega \overline{(\rmP^!)^*}$. 
    Then, the lift $((u_\chi \rmP^!)^*, \ev_u, \ev_{\id})$ of $(u_\chi \rmP^!)^*$ to $\coTwist$ induces 
    \begin{enumerate}
        \item\label{morphism from Linf} a morphism $\Linf \xrightarrow{\Omega(\psi^*)} \Pinf$, where $\psi : \overline{\rmP^!} \to \overline{\Com}$ is given on generators by $\nu \mapsto \chi(\nu) \mu_2$ for $\nu \in E(2)$ and $\nu' \mapsto 0$ for $\nu' \in E(n)$ with $n \geq 3$,
        \item\label{Tw-coalgebra structure} a $\Tw$-coalgebra structure
        \[ \eta_{\Pinf} : \Pinf \to \Tw(\Pinf), \quad s^{-1}v \mapsto s^{-1}v + \sum_{w, A} (s^{-1} w)^A. \]
    \end{enumerate}
\end{corollary}
\end{samepage}

\begin{proof}
    The proof is the same as the proof of \cref{DSV formula intro}, using \cref{cobar of coTwist} instead of \cref{cobar of coCurv}, \cref{thm:main-thm} instead of \cref{cTw comonadicity}, and \cref{explicit formula unit Twist intro} instead of \cref{explicit formula unit intro}.
\end{proof}

In \cite[Section~4.5]{DotsenkoShadrinVallette23}, a twisting procedure for $\cPinf^\chi$-algebras is given via the action of a ``deformation gauge group'' \cite[Definition~3.21]{DotsenkoShadrinVallette23} on the convolution (Pre-)Lie algebra 
\begin{equation}
\label{convolution-Lie-algebra}
    \hom_{\Sym}((u_\chi \rmP^!)^*, \End_A).
\end{equation}
Here, $\End_A$ is the endomorphism operad of a complete filtered vector space and $\hom_\Sym$ denote morphisms of $\Sym$-modules.
Recall that Maurer--Cartan elements~$\beta$ in this Lie algebra are in bijection with morphisms of operads $\phi_\beta : \cPinf^\chi \to \End_A, \, s^{-1} v \mapsto \beta(v)$ \cite[Proposition~3.18]{DotsenkoShadrinVallette23}.
According to \cite[Lemma~4.22]{DotsenkoShadrinVallette23}, the fact that $u_\chi \rmP$ is a unital extension implies that the Lie algebra \eqref{convolution-Lie-algebra} is isomorphic to $A \times \hom_{\Sym}((\rmP^{!})^{*}, \End_A)$.
In particular, we have that degree one elements of~$A$ are part of the gauge group, and thus act on the set of Maurer-Cartan elements of the Lie algebra~\eqref{convolution-Lie-algebra}, as follows (see~\cite{DotsenkoShadrinVallette23}, after Definition~3.23).
An element $a \in \calF_1 A$ acts on a Maurer-Cartan element $\beta \in \hom_{\Sym}((u_\chi \rmP^{!})^*, \End_A)$ via the composition 
\begin{equation}
\label{eq:gauge-group}
    (u_\chi \rmP^{!})^* \xrightarrow{\Delta} (u_\chi \rmP^{!})^* \, \hat{\circ} \, (u_\chi \rmP^{!})^* \xrightarrow{\beta \,\hat \circ \, (1+a)} \End_A \hat \circ \End_A \xrightarrow{\gamma} \End_A,
\end{equation}
where $\hat \circ$ denotes the completed composite product, $\Delta$ is the full decomposition of the cooperad~$(u_\chi \rmP^{!})^*$ and $\gamma$ is the full composition of the operad~$\End_A$; this gives a new Maurer--Cartan element $a \cdot \beta$ in $\hom_{\Sym}((u_\chi \rmP^{!})^*, \End_A)$, i.e.\ a twisted $\cPinf^\chi$-algebra structure on~$A$.
Moreover, this new $\cPinf^\chi$-algebra is a $\Pinf$-algebra if and only if $(a \cdot \beta)(u^*)=0$ \cite[Theorem~4.23]{DotsenkoShadrinVallette23}.

Consider the cartesian category $\calC$ of representable functors $\Op \to \Vect$, and the monoid object~$G$ in $\calC$ which sends $\rmP$ to the space of degree $0$ elements in $F_1 \rmP(0)$.
This functor~$G$ is indeed representable by the operad~$\rmT$ generated by two arity $0$ operations $\alpha$ and $\kappa$ in~$F_1$ of degrees $0$ and $(-1)$ respectively, with $d_\rmT \alpha = \kappa$.
In \cref{sec:group-action}, we use the Yoneda Lemma to show that for any $\cLinf$-operad $\rmQ$, a $\cTw$-coalgebra structure on $\rmQ$ is equivalent to a $G$-action on $\hom_{\Op}(\rmQ,-)$ (\cref{Yoneda-power}).
Applying this result to $\cPinf^\chi$, we get an action of~$\calF_1 A$ on $\hom_{\Op}(\cPinf^\chi,\End_A)$ for each complete filtered vector space~$A$. 

\begin{prop}
\label{rem:DSV-formula} 
Let $A$ be a $\cPinf^\chi$-algebra.
    Then, the actions of $\calF_1 A$ on $A$ given by the $\cTw$-coalgebra structure of $\cPinf^\chi$ and the Dotsenko--Shadrin--Vallette gauge group coincide.
    The same holds for $\Pinf$-algebras.
\end{prop}

\begin{proof}
    The action of degree one elements $\calF_1 A$ of a complete filtered vector space $A$ provided by the $\cTw$-coalgebra structure on $\cPinf^\chi$ is given by precomposition with the unit $\eta_{\cPinf^\chi}$ of the adjunction $\Curv \leftrightarrow \cMultSym$ 
    \begin{equation*}
    \label{composition-gauge-group}
        \calF_1 A \times 
        \hom_{\Op}(\cPinf^\chi, \End_A) 
        \to 
        \hom_{\Op} (\cTw(\cPinf^\chi), \End_A) \xrightarrow{(-) \circ \eta_{\cPinf^\chi}} \hom_{\Op}(\cPinf^\chi, \End_A),
    \end{equation*}
    sending $(a, \phi_\beta)$ to $(\phi_\beta \vee (\alpha \mapsto a)) \circ \eta_{\cPinf^\chi}$.
    One can then check from the explicit description of $\eta_{\cPinf^\chi}$ in Point~\eqref{cTw-coalgebra-extendable} of \cref{DSV formula intro}, that this twisted $\cPinf^\chi$-algebra structure agrees with the one given by~\eqref{eq:gauge-group}, i.e.\ that 
    \[(\phi_\beta \vee (\alpha \mapsto a)) \circ \eta_{\cPinf^\chi}(s^{-1} v) = (a \cdot \beta)(v) \]
    for every $v \in (u_\chi \rmP^{!})^*$. 
    For the case of $\Pinf$-algebras, it follows analogously from the comparison with the formula in Point~(2) of \cref{extendable operads are Tw-coalgebras intro}.
\end{proof}

\begin{remark}
In the language of \cite[Definition~4.24]{DotsenkoShadrinVallette23},  Point~(2) of  \cref{extendable operads are Tw-coalgebras intro}
asserts that the operad~$\Pinf$ admits a $\Tw$-coalgebra structure if the category of $\Pinf$-algebras is twistable.
This fact is suggested by the order of chapters in  \cite{DotsenkoShadrinVallette23}, but does not appear to be stated or proven therein.
However, once one has recognized the existence of a structural morphism~$\Linf \to \Pinf$ (which is provided here by~\cref{extendable-over-uCom}), the result can be derived from the properties of the gauge group action~\eqref{eq:gauge-group} developed in loc.\ cit., rather than the methods of the present article; we do so in \cref{sec:twistable-Tw-stable}.
\end{remark}


\subsection{Non-homotopy algebras}
\label{sec:non-homotopy-algebras}
The preceding two \cref{DSV formula intro,extendable operads are Tw-coalgebras intro} admit ``non-homotopy'' counterparts.
As before, let $\rmP$ be a Koszul operad whose Koszul dual admits a unital extension $\rmP \hookrightarrow u_\chi \rmP$, and let $\phi :  \rmP^{\as}\cong (\rmP^{!})^{*} \to \rmP$ be the Koszul twisting morphism associated to $\rmP$ \cite[Sec.~7.4.1]{LodayVallette12}.
One then has a notion of \emph{curved $\rmP$-algebras} encoded by the following operad (see also \cite[Sec.~2.5.4]{CalaqueCamposNuiten21}).

\begin{definition}
    The operad $\cP^{\chi}$ is the quotient of $\cPinf^{\chi} = \Omega(u_\chi \rmP^{!})^{*}$ by the operadic ideal generated by the kernel $\ker(\phi) \subset \rmP^{\as} \subset \mathbf{k}u^{*} \oplus \rmP^{\as} \cong (u_\chi \rmP^{!})^{*}$.
\end{definition}

\begin{lemma}
\label{cP-in-curv}
    Let $\rmP$ be a Koszul operad such that its Koszul dual operad $\rmP^{!}$ admits a unital extension $\rmP^{!} \hookrightarrow u_\chi \rmP^{!}$. 
    Then, the curved variant $(\cP^{\chi},\rmP)$ is in $\Curv$.
    Moreover the projection morphism $(\cPinf^{\chi},\Pinf) \to (\cP^{\chi},\rmP)$ is in $\Curv$. 
\end{lemma}

\begin{proof}
    We first prove that $(\cP^{\chi}, \rmP)$ is in $\Curv$.
    We have a morphism $\cLinf \to \cP^\chi$ given by the composition $\cLinf \to \cPinf^\chi \to \cP^\chi$.
    Since we had an isomorphism $\Omega(\rmP^!)^* \vee [\ell_0] \xrightarrow{\sim} \cPinf^\chi$ and $\cP^\chi$ is the quotient of $\cPinf^\chi$ by the operadic ideal generated by $s^{-1} \ker(\phi) \subset \Omega(\rmP^!)^*$, we get an isomorphism $\rmP \vee [\ell_0] \xrightarrow{\sim} \cP^\chi$.
    The fact that $d_{\cP^\chi}(\rmP)$ is of degree at most $1$ in $\ell_0$ follows from the corresponding fact for $d_{\cPinf}(\Pinf)$. 
    The projection $(\cPinf^{\chi},\Pinf) \to (\cP^{\chi},\rmP)$ is in $\Curv$ since it sends $\Pinf$ to $\rmP$.
\end{proof}

\begin{lemma}
\label{lem:cP-factors-through-cLie}
    The structural morphism $\cLinf \to \cP^{\chi}$ factors through $\cLie$, in the sense that the following diagram commutes in $\Curv$
    \begin{center}
    \begin{tikzcd}
    \cLinf \arrow[r] \arrow[d] & \cPinf^{\chi} \arrow[d] \\
    \cLie \arrow[r]                 & \cP^{\chi}    
    \end{tikzcd}
    \end{center}
\end{lemma}

\begin{proof}
    This follows from the fact that the canonical twisting morphism $\phi$ defines a natural transformation $(-)^{\as} \to \id$ between the endofunctors $(-)^\as$ and $\id$ on the category of Koszul operads.
\end{proof}

\begin{corollary}
\label{cor:non-homotopy-cTw}
    Let $\rmP$ be a Koszul operad whose Koszul dual $\rmP^!$ admits a unital extension $\rmP^! \hookrightarrow u_\chi \rmP^!$. 
    Then, the lift $((u_\chi \rmP^!)^*, \ev_u)$ of $(u_\chi \rmP^!)^*$ to $\coCurv$ induces a morphism $\cLie \to \cP^\chi$ and a $\cTw$-coalgebra structure $\eta_{\cP} : \cP^\chi \to \cTw(\cP^\chi)$ given by the projections of the formulas from \cref{DSV formula intro}.
\end{corollary}

\begin{proof}
From \cref{cP-in-curv} we know that $\cPinf \to \cP$ is a morphism in Curv, which factors through $\cLie$ by \cref{lem:cP-factors-through-cLie}. 
    Now, since the unit $\eta$ is a natural transformation between the endofunctors $\id$ and $\cTw(-)$, we have that the following diagram commutes:
    \begin{center}
    \begin{tikzcd}
    \cPinf \arrow[r,"\eta_{\cPinf}"] \arrow[d] & \cTw(\cPinf) \arrow[d] \\
    \cP \arrow[r,"\eta_{\cP}"]                 & \cTw(\cP)     
    \end{tikzcd}
    \end{center}
    This finishes the proof. 
\end{proof}

A similar proof, using $\Twist$ instead of $\Curv$, leads to the following analogous result.

\begin{corollary}
\label{cor:non-homotopy-Tw}
    Let $\rmP$ be a Koszul operad whose Koszul dual $\rmP^!$ admits a unital extension $\rmP^! \hookrightarrow u_\chi \rmP^!$. 
    Then, the lift $((u_\chi \rmP^!)^*, \ev_u, \ev_{\id})$ of $(u_\chi \rmP^!)^*$ to $\coTwist$ induces a morphism $\Lie \to \rmP$ and a $\Tw$-coalgebra structure $\eta_{\rmP} : \rmP \to \Tw(\rmP)$ given by the projections of the formulas from \cref{extendable operads are Tw-coalgebras intro}.
\end{corollary}

Analogous results hold for the analogous categories of non-symmetric operads.
The proofs are the same as in the symmetric case.


\section{Examples}
\label{examples}

\subsection{The associative operad}
The associative operad~$\Ass$ is the free operad on one arity~$2$ generator~$m_2$ generating the regular representation of $\Sym_2$, modulo the operadic ideal generated by the associativity relation $m_2 \circ_1 m_2 = m_2 \circ_2 m_2$.
The underlying $\Sym$-module~$\Ass(n)=\mathbf{k}[\Sym_n]$ is the regular representation of the symmetric group~$\Sym_n$.
See \cite[Sec.~9.1.3]{LodayVallette12} for more details.
It is Koszul self-dual, i.e.\ its Koszul dual operad is $\Ass^! \cong \Ass$.

Every associative algebra has an underlying Lie algebra structure given by $[a,b] \eqdef ab - ba$; this corresponds to a morphism $\Lie \to \Ass$, and there is a corresponding morphism $\Linf \to \Ainf$ given explicitly by $\ell_n \mapsto \sum_{\sigma \in \Sym_n} \mu_n^\sigma$.
Induction along this morphism  gives
the notion of universal enveloping $\Linf$-algebra  \cite{LadaMarkl95}.  It is known that $\Ainf$ and $\As$ are $\Tw$-coalgebras and $\Tw$ homotopy fixed points \cite[Cor.~5.17~\&~5.18]{DolgushevWillwacher15}.

We recover all these facts as follows: 

\begin{prop}
    The operad $\Ainf$ admits 
    \begin{enumerate}
        \item a morphism $\Linf \to \Ainf$ given by $\ell_n \mapsto \sum_{\sigma \in \Sym_n} \mu_n^\sigma$,
        \item a $\Tw$-coalgebra structure $\Ainf \to \Tw(\Ainf)$ given by $\mu_n \mapsto \mu_n^{\alpha}$ (see (\ref{mu-n-alpha})). 
    \end{enumerate}
    The same holds when replacing $\Ainf$ and $\Linf$ by $\Ass$ and $\Lie$, respectively.
\end{prop}

\begin{proof}
    The Koszul dual of the associative operad $\Ass^! \cong \Ass$ admits a unital extension $\Ass \hookrightarrow u_\chi \Ass \defeq \uAss$ given by $\chi(m_2) = \chi(m_2^\sigma) \eqdef 1$.
    Point~(\ref{morphism from Linf}) in \cref{extendable operads are Tw-coalgebras intro} gives explicitly the morphism $\Linf \to \Ainf$.
    In order to apply Point~(\ref{Tw-coalgebra structure}), one computes explicitly the structural morphism $\delta_{\uAss^{*}}: \uAss^{*} \to \overline{\mathbb{T}}(\uAss^{*})$, which is given on the generators by
    \[
    \delta_{\uAss^{*}}(m_n^{*})
    =
    \sum_{k \geq 0} m_{k+n}^{*}((m_0^{*})^k,-,\ldots,-)+ \delta_{\uAss^{*}}^+(m_n^{*}),
    \]
    where $\delta_{\uAss^{*}}^+(m_n^{*})$ is a sum of trees which have at least one cooperation of non-zero arity outside of the first level (\cref{notation delta +}).
    The analogous results for $\Ass$ then follow from \cref{cor:non-homotopy-Tw}.
\end{proof}

We now prove that $\Ainf$ and $\Ass$ are homotopy fixed points of $\Tw$. 
Let $\Ass(x_0,\ldots,x_{n-1})$ denote the free associative algebra on~$n$ generators.

\begin{lemma}
\label{lemma:bracket-not-zero-divisor-in-Ass}
    The canonical associative algebra map 
    \[
    \Ass(x_0, \dots, x_{n-1}) \to \Ass(u,v, x_0, \dots, x_{n-1}) / \left( (uv-vu) - x_0 \right)
    \]
    is injective.
\end{lemma}

\begin{proof}
    Consider a non-commutative polynomial $p(x_0, \dots, x_{n-1})$ in $\Ass(x_0, \dots, x_{n-1})$, and suppose that its image under the map above is $0$.
    This means that replacing some occurrences of $x_0$ in $p$ by $uv-vu$ gives $0$.
    But since the only equivalence relation between monomials in the free associative algebra is the associativity relation, no two monomials in the resulting polynomial cancel each other, therefore they must all vanish and we have that~$p=0$. 
\end{proof}

\begin{corollary}
    The operads $\Ainf$ and $\Ass$ are homotopy fixed points for $\Tw$.    
\end{corollary}

\begin{proof}
    Combining \cref{lemma:criteria-bracket-zero-divisor} with \cref{lemma:bracket-not-zero-divisor-in-Ass}, we get that the Lie bracket is not a right zero divisor in $\Ass$.
    We conclude with \cref{thm:homotopy-fixed-point intro}.
\end{proof}


\subsection{The Lie operad}
The Lie operad is the free operad on one arity~$2$ generator~$m_2$ generating the sign representation, modulo the operadic ideal generated by the Jacobi relation $m_2\circ_1 m_2+(m_2\circ_1 m_2)^\tau+(m_2\circ_1 m_2)^{\tau^2}=0$,
where $\tau=(123)$ is the 3-cycle in~$\Sym_3$.
The Koszul dual operad of $\Lie$ is the commutative operad~$\Com$, which is the quotient of the free operad on one arity~$2$ operation~$\mu_2$ on which $\Sym_2$ acts trivially, by the operadic ideal generated by the associativity relation $\mu_2 \circ_1 \mu_2 = \mu_2 \circ_2 \mu_2$.
The underlying $\Sym$-module of $\Com$ is in each arity the trivial representation $\Com(n)=\mathbf{k}\mu_n$ of~$\Sym_n$.

\begin{prop}
    The operad $\Linf$ admits 
    \begin{enumerate}
        \item a morphism $\Linf \to \Linf$ given by $\ell_n \mapsto \ell_n$,
        \item a $\Tw$-coalgebra structure $\Linf \to \Tw(\Linf)$ given by $\ell_n \mapsto \ell_n^{\alpha}$.
    \end{enumerate}
    The same holds for the operad~$\Lie$.
\end{prop}

\begin{proof}
    The Koszul dual of the Lie operad $\Lie^! \cong \Com$ admits a natural unital extension $\Com \hookrightarrow u_\chi \Com \defeq \uCom$ defined by $\chi(\mu_2) \eqdef 1$.
    Point~(\ref{morphism from Linf}) in \cref{extendable operads are Tw-coalgebras intro} gives explicitly the morphism $\Linf \to \Linf$.
    In order to apply Point~(\ref{Tw-coalgebra structure}), one computes explicitly the structural morphism $\delta_{\uCom^{*}}: \uCom^{*} \to \overline{\mathbb{T}}(\uCom^{*})$, which is given on the generators by
    \[
    \delta_{\uCom^{*}}(\mu_n^{*})
    =
    \sum_{k \geq 0} \frac{1}{k!}\mu_{k+n}^{*}((\mu_0^{*})^k,-,\ldots,-)+ \delta_{\uCom^{*}}^+(\mu_n^{*}).
    \]
    The analogous results for $\Lie$ then follow from \cref{cor:non-homotopy-Tw}.
\end{proof}

We now prove that $\Linf$ and $\Lie$ are homotopy fixed points of $\Tw$. 
Let us denote by $\Lie(x_0, \dots, x_{n-1})$ the free Lie algebra on~$n$ generators.

\begin{lemma}
\label{lemma:bracket-not-zero-divisor-in-Lie}
    The canonical Lie algebra map 
    \[
    \Lie(x_0, \dots, x_{n-1}) \to \Lie(u,v, x_0, \dots, x_{n-1}) / \left( \ell_2(u,v) - x_0 \right)
    \]
    is injective.
\end{lemma}
\begin{proof}
    Recall the following consequences of the Poincaré-Birkhoff-Witt theorem:
    \begin{enumerate}[leftmargin=*]
        \item The natural Lie-algebra morphism from a Lie algebra to its universal enveloping algebra (with Lie bracket given by the commutator) is injective.
        \item The universal enveloping algebra of the free Lie algebra on $n$ generators is the free associative algebra on $n$ generators.
        \item The universal enveloping algebra of the quotient of a free Lie algebra by a Lie ideal is the quotient of the universal enveloping algebra of the free Lie algebra by the corresponding associative ideal.
    \end{enumerate}
    Using these facts, we get the following commutative diagram of Lie algebras
    \begin{center}
        \begin{tikzcd}
        \Lie(x_0, \dots, x_{n-1}) \ar[r] \ar[d, hook] & \Lie(u,v, x_0, \dots, x_{n-1}) / \left( \ell_2(u,v) - x_0 \right) \ar[d, hook] \\
        \Ass(x_0, \dots, x_{n-1}) \ar[r] & \Ass(u,v, x_0, \dots, x_{n-1}) / \left( \ell_2(u,v) - x_0 \right). 
    \end{tikzcd}
    \end{center}
    According to ~\cref{lemma:bracket-not-zero-divisor-in-Ass}, the bottom horizontal map is injective.
    Therefore, the top horizontal map is also injective.
\end{proof}

\begin{corollary}
    The operads $\Linf$ and $\Lie$ are homotopy fixed points for $\Tw$.
\end{corollary}
\begin{proof}
    This follows from combining \cref{lemma:bracket-not-zero-divisor-in-Lie,lemma:criteria-bracket-zero-divisor}, and applying \cref{thm:homotopy-fixed-point intro}.
\end{proof}

The present methods should be compared to the proofs of \cite[Cor.~5.12~\&~5.13]{DolgushevWillwacher15}.


\subsection{Operads obtained by distributive law}
A distributive law between two operads~$\rmP$ and~$\rmQ$ is a morphism of $\Sym$-modules $\rmQ \circ \rmP \to \rmP \circ \rmQ$ satisfying certain properties which allow one to endow the composite product $\rmP \circ \rmQ$ with an operad structure~\cite[Section~8.6]{LodayVallette12}.
We will consider in this section the special cases where $\rmQ$ is the (curved) Lie operad.

\begin{definition}
\label{def:distributive-law}
    We say that an operad $\rmO$ is obtained from an operad $\rmP$ with zero differential and $\Lie$ by \emph{distributive law} if $\rmO = \rmP \circ \Lie$ as $\Sym$-modules, where the symbol~$\circ$ denotes the circle product (aka.\ plethysm), and the operad structure is given by the respective operad structures, together with the relations
    \begin{equation}
        \label{eq:distributive-law}
        \ell_2 \circ_2 \nu = (-1)^{|\nu|} \sum_{i=1}^{n}(\nu \circ_i \ell_2)^{\sigma_{1,i}} \quad \forall \ \nu \in \rmP(n),
    \end{equation}
    where $\sigma_{1,i}$ is the cycle $(12\cdots i)$ in $\Sym_{n+1}$.
\end{definition}

\begin{def-prop}
\label{def:curved-distributive-law}
    We say that an operad $\rmO$ is obtained from an operad $\rmP$ with zero differential and $\cLie$ by \emph{curved distributive law} if
    \begin{enumerate}
        \item\label{underlying module} the underlying non-dg operad is $(\rmP \circ \Lie) \vee [\ell_0] \vee [\ell_1]$,
        \item\label{differential2} the differential $d$ is given by $d_{\cLie}$ on $\cLie = \Lie \vee [\ell_0] \vee [\ell_1]$ and by $-[\ell_1, \,\cdot\,]$ on $\rmP$.
    \end{enumerate} 
\end{def-prop}

\begin{proof}
    We have to check that $d^2=0$.
    We already know that $d_{\cLie}^2 = 0$, so we have to show that $d^2 \nu = 0$ for $\nu \in \rmP$. 

    We first compute 
    \begin{align*}
        d^2(\nu)  = d [-\ell_1, \nu] 
         = -[d \ell_1, \nu] + [\ell_1, d\nu] 
        & = [\ell_2 \circ_1 \ell_0 + \ell_1 \circ_1 \ell_1, \nu] - [\ell_1, [\ell_1, \nu]] \\
        & = [\ell_2 \circ_1 \ell_0, \nu] + [\ell_1 \circ_1 \ell_1, \nu] - [\ell_1, [\ell_1, \nu]] \\
        & = [\ell_2 \circ_1 \ell_0, \nu] + \frac{1}{2} [[\ell_1, \ell_1], \nu] - [\ell_1, [\ell_1, \nu]] \\ 
        & = [\ell_2 \circ_1 \ell_0, \nu] 
    \end{align*}
    where we have used the Jacobi identity to cancel the term $\frac{1}{2} [[\ell_1, \ell_1], \nu] - [\ell_1, [\ell_1, \nu]]$. 
    We finish the computation:
    \begin{align*}
        d^2(\nu) = [\ell_2 \circ_1 \ell_0, \nu] & = (\ell_2 \circ_1 \ell_0) \circ_1 \nu - \sum_{i=1}^n \nu \circ_i (\ell_2 \circ_1 \ell_0) \\
        & = (-1)^{|\nu|} (\ell_2 \circ_2 \nu) \circ_1 \ell_0 - \sum_{i=1}^n (\nu \circ_i \ell_2) \circ_i \ell_0 \\
        & = (-1)^{|\nu|} \left( \ell_2 \circ_2 \nu - \sum_{i=1}^n (\nu \circ_i \ell_2)^{\sigma_{1,i}} \right) \circ_1 \ell_0 = 0,
    \end{align*}
    where the last equality follows from the distributive law identity~\eqref{eq:distributive-law}.
\end{proof}

\begin{prop}
\label{distributive-coalgebra}
    Any operad $\rmO$ which can be written as a distributive law $\rmO = \rmP \circ \Lie$ with the Lie operad admits a lift to $\Twist$.
    Equivalently, it admits a $\Tw$-coalgebra structure.
\end{prop}
\begin{proof}
    A lift to $\Twist$ is given in \cref{def:curved-distributive-law}.
    The fact that this is equivalent to the existence of a $\Tw$-coalgebra structure follows from \cref{thm:main-thm}.
\end{proof}

\begin{remark}
    The analogue for non-symmetric operads of \cref{def:curved-distributive-law} with $\Lie$ replaced by $\As$ does not work, because $d^2\nu \neq 0$ in this case.  In fact, the analogue of \cref{distributive-coalgebra} is false. 
    Indeed, consider the non-commutative version $\ncGer$ of the Gerstenhaber operad (defined in \cite[Sec.~3.1.2]{DotsenkoShadrinVallette-toric}).
    While $\ncGer$ can be obtained via a distributive law with the associative operad \cite[Prop.~3.1.5]{DotsenkoShadrinVallette-toric}, it is shown in \cite[Prop.~6.1]{DotsenkoShadrinVallette23} that it does not admit a $\Tw$-coalgebra structure.
\end{remark}

\begin{remark}
    More generally, given a morphism $\Lie \to \rmP$, there is a natural candidate for a lift to $\Twist$.
    Consider the non-dg operad $\cP \eqdef \rmP \vee [\kurv] \vee [\diff]$, with derivation given by
    \[d_{\cP} \kurv = - \diff \circ_1 \kurv, \quad d_{\cP} \diff = - \diff \circ_1 \diff + \ell_2 \circ_1 \kurv, \quad \text{and } d_{\cP} \nu = d_\rmP \nu - [\diff, \nu] \]
    for $\nu \in \rmP$.
    Then, we have $d_{\cP}^2(\nu)=0$ for all $\nu \in \rmP$ if and only if $[\ell_2 \circ_1 \kurv,\nu]=0$.
    In this case, it is straightforward to check that $(\cP, \rmP)$ is in $\Twist$, and therefore $\rmP$ has a $\Tw$-coalgebra structure.
    This recovers the criterion in~\cite[Prop.~5.29]{DotsenkoShadrinVallette23} guaranteeing that $\rmP$ is a $\Tw$-coalgebra.
\end{remark}

In the case where $\rmP$ is Koszul, we have that $\rmO$ is Koszul as well \cite[Thm.~4.5]{Markl96}.
One can then ask if the Koszul dual $\rmO^{!}$ admits a unital extension $\rmO^{!} \hookrightarrow u_\chi \rmO^{!}$.

\begin{prop}
\label{Tw-and-distributive}
    Let $\rmO=\rmP \circ \Lie$ be an operad obtained as a distributive law of a Koszul operad $\rmP$ with the Lie operad.
    Then, the Koszul dual operad $\rmO^{!}$ is extendable, and its minimal resolution $\rmO_\infty$ is a $\Tw$-coalgebra.
\end{prop}

\begin{proof}
    By \cite[Lem.~4.3]{Markl96}, we have that $\rmO^{!} \cong \Lie^{!} \circ \rmP^{!} \cong \Com \circ \rmP^{!}$.
    One can then extend trivially to $\rmO^{!}$ the map $\chi$ associated to $\Com$.
    The $\Tw$-coalgebra structure is then given by \cref{extendable operads are Tw-coalgebras intro}.
\end{proof}

We now prove that $\rmO_\infty$ and $\rmO$ are homotopy fixed points of $\Tw$.  
The following result appears with a different proof and an extra hypothesis (a certain bound on degrees) in  \cite[Theorem 6.1]{DolgushevWillwacher15}.

\begin{prop}
\label{fixed-point-distributive-law}
    Any operad $\rmO$ which can be written as a distributive law $\rmO = \rmP \circ \Lie$ with the Lie operad is a homotopy fixed point for $\Tw$.
    If further $\rmP$ is Koszul, then the minimal resolution $\rmO_\infty$ is also a homotopy fixed point for $\Tw$.
\end{prop}

\begin{proof}
    Combining \cref{lemma:bracket-not-zero-divisor-in-Lie} and \cref{lemma:criteria-bracket-zero-divisor}, we know that $\ell_2$ is not a right zero divisor in~$\Lie$.
    Further, it is clear from the definition of the relations in a distributive law~\eqref{eq:distributive-law} that~$\ell_2$ is not a zero divisor in $\rmO = \rmP \circ \Lie$ either.
    The result then follows from combining \cref{Tw-and-distributive} with \cref{thm:homotopy-fixed-point intro}.
\end{proof}

\begin{remark}
    Checking that the Lie bracket is not a right zero divisor in an operad seems to be related to the existence of a Poincaré-Birkhoff-Witt (PBW) theorem for this operad.
    While we don't have a precise understanding of this relationship, we note that we used the standard PBW theorem in the proof of \cref{lemma:bracket-not-zero-divisor-in-Lie}.
    Moreover, one can read in the introduction of \cite{DotsenkoTamaroff21} that the PBW property is satisfied by operads obtained by means of distributive laws, for which the Lie bracket is also not a right zero divisor.
\end{remark}


\subsection{The Gerstenhaber operad}

The Gerstenhaber operad~$\Ger$ is obtained via distributive law between the commutative and Lie operads $\Ger \eqdef \Com \circ \Lie$ (\cref{def:distributive-law}).
It is Koszul self-dual, i.e.\ we have $\Ger^! \cong \Lie^! \circ \Com^! \cong \Com \circ \Lie \cong \Ger$~\cite[Lem.~4.3]{Markl96}.

Recall from \cite[Sec.~4.3]{DolgushevWillwacher15} that its minimal resolution $\Ginf$ is quasi-free on generators given in arity~$n$ by a family $\{g_n^I\}$ indexed by ordered partitions $I=I_1 \sqcup \cdots \sqcup I_k$ of $\{1,\ldots,n\}$ such that $\max(I_1)< \cdots < \max(I_k)$, for $1 \leq k \leq n$.
The generators of $\Ger(n)$ are indexed by the same partitions~$I$; we will denote them by~$w_n^I$.
Finally, we will write $I_n \eqdef \{1\}\sqcup \cdots \sqcup \{n\}$ for the unique increasing ordered partition of size $n$.

\begin{prop}
\label{Gerst-coalgebra}
    The operad $\Gerinf$ admits 
    \begin{enumerate}
        \item a morphism $\Linf \to \Gerinf$ given by $\ell_n \mapsto g_n^{I_n}$,
        \item a $\Tw$-coalgebra structure $\Gerinf \to \Tw(\Gerinf)$ given by $g_n^{I} \mapsto \sum_{r \geqslant 0} \frac{1}{r!} g_{r+n}^{I_r \sqcup I}(\alpha^r,-)$.
    \end{enumerate}
    The same holds for the operad $\Ger$.
\end{prop}

\begin{proof}
    This follows from \cref{Tw-and-distributive}, whose proof relies on \cref{extendable operads are Tw-coalgebras intro}.
    The explicit $\Tw$-coalgebra structure comes from Point~(\ref{Tw-coalgebra structure}) in \cref{extendable operads are Tw-coalgebras intro}, where one computes explicitly the structural morphism $\delta_{\uGer^{*}}: \uGer^{*} \to \overline{\mathbb{T}}(\uGer^{*})$, which in the present case is given on the generators by
    \[
    \delta_{\uGer^{*}}((w_n^I)^{*})
    =
    \sum_{r \geq 0} \frac{1}{r!}(w_{r+n}^{I_r \sqcup I})^{*}((w_0^{*})^r,-,\ldots,-)+ \delta_{\uGer^{*}}^+((w_n^I)^{*}).
    \]
    The analogous results for $\Ger$ then follow either from \cref{distributive-coalgebra} or \cref{cor:non-homotopy-Tw}.
    The shortest path to the explicit formula for the $\Tw$-coalgebra structure is Point~(2) in \cref{cor:non-homotopy-Tw}.
\end{proof}

\begin{prop}
\label{Gerst-homotopy-fixed-point}
    The operads $\Gerinf$ and $\Ger$ are homotopy fixed points for $\Tw$.
\end{prop}

\begin{proof}
This in an immediate consequence of \cref{fixed-point-distributive-law}.
\end{proof}

The present methods should be compared to the proofs of \cite[Cor.~5.12~\&~5.13]{DolgushevWillwacher15}.

\begin{remark} 
    Let us recall (following \cite{Willwacher-KGT}) how some of the above structures appear in the context of Kontsevich's work on formality of the little disks operad $\mathbb{E}_2$ \cite{Kontsevich99}.
    The homology of~$\mathbb{E}_2$ is known to be the Gerstenhaber operad $\Ger$.
    Kontsevich uses the semi-algebraic forms on the Fulton--MacPherson operad $\Omega_{PA}(\FM_2)$ as a model for cochains on $\mathbb{E}_2$.  The main content of Kontsevich's proof is the introduction of a certain operad $\Gra$ of graphs and, by studying certain integrals, a quasi-isomorphism of operads
    $$\Omega_{PA}(\FM_2)^* \to \Tw(\Gra)^c$$
    where $\Tw(\Gra)^c$ is a quotient of $\Tw(\Gra)$ by certain disconnected graphs.  

    The translation to a statement about $\Ger$ is the following.
    There is a map $\Ger \to \Gra$ given by sending the product to the graph with 2 vertices and the bracket to the graph with 2 vertices and an edge between them.  Applying $\Tw$ and precomposing with the Tw-coalgebra structure morphism (recovered above in \cref{Gerst-coalgebra}), one gets maps $\Ger~\to~\Tw(\Ger) \to \Tw(\Gra)$.
    One composes further by the quotient to $\Tw(\Gra)^c$, giving finally a map 
    \begin{equation}
    \label{eq:Ger-Graph-morphism}
    \Ger \to \Tw(\Graph)^c
    \end{equation}
    One checks that this is a quasi-isomorphism.
    (Using the fact, recovered above in \cref{Gerst-homotopy-fixed-point}, that $\Ger$ is a $\Tw$ homotopy fixed point, one could instead check that $\Tw(\Ger) \to \Tw(\Graph)^c$ is a quasi-isomorphism.)
\end{remark}

\begin{remark} 
\label{Dolgushev-Willwacher}
The Deligne conjecture asserts that the $\mathbb{E}_2$ operad acts on the Hochschild chains of an $\mathbb{E}_1$ algebra; it was proven by Kontsevich and Soibelman \cite{Kontsevich-Soibelman-deligneconjecture} by constructing a morphism from $\Omega_{PA}(\FM_2)^*$ to an operad which acts on Hochschild chains.
We follow Dolgushev and Willwacher \cite{DolgushevWillwacher15} in denoting said operad as $\Br$; in fact, loc.\ cit.\ introduces another operad $\BT$ such that $\Tw(\BT)$ acts on Hochschild chains and contains $\Br$ as a quasi-isomorphic sub-operad.  One may thus view `the space of solutions to the Deligne conjecture' as the quasi-isomorphisms $\Gerinf \to \Tw(\BT)$.  

The main concern of \cite{DolgushevWillwacher15} is to show that any solution to the Deligne conjecture is homotopic to one compatible with the $\Tw$-coalgebra structure.  We have now recovered all the ingredients they used to show this result: 
Fix a quasi-isomorphism $f : \Gerinf \xrightarrow{\sim} \Tw(\BT)$.  
Now, $\Gerinf$ has a $\Tw$-coalgebra structure $\eta_{\Gerinf} : \Gerinf \to \Tw(\Gerinf)$ (\cref{Gerst-coalgebra})
and is a homotopy fixed point for $\Tw$ (\cref{Gerst-homotopy-fixed-point}).  By the latter fact, $\Tw(\BT) \cong \Gerinf$ is  also a homotopy fixed point for~$\Tw$.
By \cite[Thm~10.1]{DolgushevWillwacher15}, recovered above as \cref{create Tw morphism}, we get a morphism $\tilde f \eqdef \Tw(\varepsilon_{\BT} \circ f) \circ \eta_{\Gerinf} : \Gerinf \xrightarrow{\sim} \Tw(\BT)$ of $\Tw$-coalgebras which is homotopy equivalent to $f$.

Let us note however that \cite{DolgushevWillwacher15} showed a bit more, namely that if the original morphism $f$ factors through $\Br$, then so too does $\widetilde{f}$.  
\end{remark}


\subsection{The gravity operad}
The gravity operad $\Grav$ was introduced by Getzler in \cite{Getzler94,Getzler95}.
It is a generalization of the Lie operad given by the homology $H_{\bullet-1} (\mathcal{M}_{0,n+1})$ of the moduli space of genus zero curves with $(n+1)$ marked points \cite[Sec.~3.4]{Getzler95}.
It is the free operad on the $\Sym$-module $E$, where $E(n)$, $n\geq 2$ is the sign representation of $\Sym_n$ generated by one operation~$m_n$, modulo the operadic ideal generated by the relations given for $k > 2$ and $l \geq 0$ by
\[
\sum_{1 \leq i < j \leq k} (m_{k+l-1}\circ_1 m_2)^{\sigma_{ij}} 
=
\begin{cases}
    m_{l+1}\circ_1 m_k  & \text{if } l >0 ,\\
    0 & \text{else},
\end{cases}
\]
where $\sigma_{ij} \in \Sym_{k+l}$ takes $i \mapsto 1$ and $j \mapsto 2$ and preserves the order of all other $h \in \{1,\ldots,k+l\}\setminus\{i,j\}$. 
Note that for $k=3$ and $l=0$ this gives the Jacobi identity.

The Koszul dual of $\Grav$ is the operad $\HypCom$ \cite[Thm.~4.13]{Getzler95}, which is given by the homology $H_\bullet(\overline{\mathcal{M}}_{0,n+1})$ of the Deligne--Mumford compactification of the moduli space of genus zero curves with $n+1$ marked points \cite[Sec.~3.6]{Getzler95}.
It is the free operad on the $\Sym$-module $E$, where $E(n)$, $n\geq 2$ is the trivial representation of $\Sym_n$ generated by one operation~$\mu_n$, modulo the operadic ideal generated by the relations given for $m \geqslant 0$ by 
\[
\sum_{\substack{k+l=m \\ \sigma \in \Sh(k,l)}}(\mu_{k+2} \circ_1 \mu_{l+2})^{\sigma'}
=
\sum_{\substack{k+l=m \\ \sigma \in \Sh(k,l)}}(\mu_{k+2} \circ_2 \mu_{l+2})^{\tilde \sigma},
\]
where $\sigma'$ is the permutation of $\{1,\ldots,m+3\}$ which acts as the identity on $1,2$ and $l+3$ and by $\sigma$ elsewhere, while $\tilde \sigma$ acts as the identity on $1,2,3$ and by $\sigma$ elsewhere. 
Note that for~$m=0$ we get the usual associativity relation defining the operad~$\Com$.

\begin{prop}
    The operad $\Grainf$ admits 
    \begin{enumerate}
        \item a morphism $\Linf \to \Grainf$,
        \item a $\Tw$-coalgebra structure $\Grainf \to \Tw(\Grainf)$.
    \end{enumerate}
    The same results hold for the operad $\Grav$.
\end{prop}

\begin{proof}
    Observe that, as the commutative operad, the hypercommutative operad $\Grav^! \cong \HypCom$ admits a unital extension $\HypCom \hookrightarrow u_\chi \HypCom$ given by $\chi(\mu_2)\eqdef 1$.
    Therefore the result follows from \cref{extendable operads are Tw-coalgebras intro}.
    The analogous results for $\Grav$ are given by \cref{cor:non-homotopy-Tw}.
\end{proof}

\begin{remark}
    Given a choice of basis for $\HypCom(n)=H_\bullet(\overline{\mathcal{M}}_{0,n+1})$, one could further compute explicitly both the morphism $\Linf \to \Grainf$ and the $\Tw$-coalgebra structure via the structural morphism $\delta_{\HypCom^{*}}$.
    In order to do so, one could use the  Givental graphs of~\cite{KhoroshkinMarkarianShadrin13}, the cellular model of \cite{RossiSalvatore24}, chain models of the gravity operad \cite{DupontHorel18}, or the homology of Brown's moduli spaces \cite{DupontVallette17}.
\end{remark}

It is unknown to us whether or not the operads $\Grainf$ and $\Grav$ are homotopy fixed points for $\Tw$. 


\subsection{The permutative operad} \label{sec: permutative}

The permutative operad $\Perm$ was defined by Chapoton in \cite{Chapoton01}.
It is the free operad on one arity $2$ operation $\mu_2$ generating the regular representation of $\Sym_2$, modulo the operadic ideal generated by the associativity and permutative relations
\begin{equation}
\label{eq:perm-relations}
    \mu_2 \circ_1 \mu_2 = \mu_2 \circ_2 \mu_2 \quad \text{and} \quad
\mu_2 \circ_2 \mu_2 = \mu_2 \circ_2 \mu_2^{(12)}.
\end{equation}
Its Koszul dual $\Perm^! \cong \PreLie$ is the Pre-Lie operad \cite[Prop.~2.1]{ChapotonLivernet01}.
This operad is the free operad on one arity~$2$ operation $t_2$ generating the sign representation of $\Sym_2$, modulo the operadic ideal generated by the relation
\[
t_2 \circ_1 t_2 - t_2 \circ_2 t_2
= 
(t_2 \circ_1 t_2 - t_2 \circ_2 t_2)^{(23)}.
\]
As shown in \cite[Thm.~1.9]{ChapotonLivernet01}, $\PreLie$ is isomorphic to the operad of rooted trees~$t$. 
Let $t_n$ denote the linear rooted tree (ladder) with $n$ vertices.

We write $\tau \eqdef s^{-1}t^{*}$ for the generators of $\Perinf$, which we identify with their predual rooted trees~$t$.
Given a rooted tree $\tau$, one can consider its \emph{unital expansions}, which are bigger rooted trees~$\tilde \tau$ with additional ``black'' vertices, such that when $u \in \uPreLie$ (see below for the definition of $\uPreLie$) is substituted into them, one gets $\tau$ back, up to a coefficient $c_{\tilde \tau}$ \cite[Lem.~4.27]{DotsenkoShadrinVallette23}.
We denote by $r(\tilde \tau)$ the number of black vertices in $\tilde \tau$.

\begin{samepage}
    \begin{prop}
\label{Perm is a Tw-coalgebra}
    The operad $\Perinf$ admits 
    \begin{enumerate}
        \item a morphism $\Linf \to \Perinf$ sending $\ell_n \mapsto \sum_{\sigma \in \Sym_n} \tau_n^\sigma$,
        \item a $\Tw$-coalgebra structure $\Perinf \to \Tw(\Perinf)$ given by $\tau \mapsto \sum_{\tilde \tau} \frac{1}{c_{\tilde \tau}} {\tilde \tau}(\alpha^{r(\tilde \tau)},-)$.
    \end{enumerate}
    The same is true for the operad $\Perm$.
\end{prop}
\end{samepage}

This is the formula that appears in \cite[Prop.~4.28]{DotsenkoShadrinVallette23} (see \cref{rem:DSV-formula}).

\begin{proof}
    The Koszul dual of the permutative operad $\Perm^! \cong \PreLie$ admits a unital extension $\PreLie \hookrightarrow u_\chi \PreLie \defeq \uPreLie$ defined by $\chi(t_2) \eqdef 1$ \cite[Prop.~4.21]{DotsenkoShadrinVallette23}.
    Point~(\ref{morphism from Linf}) in \cref{extendable operads are Tw-coalgebras intro} gives explicitly the morphism $\Linf \to \Perinf$.
    In order to apply Point~(\ref{Tw-coalgebra structure}), one computes explicitly the structural morphism $\delta_{\uPreLie^{*}}: \uPreLie^{*} \to \overline{\mathbb{T}}(\uPreLie^{*})$, which is given on the generators by
    \[
    \delta_{\uPreLie^{*}}(\tau)
    =
    \sum_{\tilde \tau} \frac{1}{c_{\tilde \tau}}{\tilde \tau}((u^{*})^{r(\tilde \tau)},-,\ldots,-)+ \delta_{\uPreLie^{*}}^+(\tau).
    \]
    The analogous results for $\Perm$ then follow from \cref{cor:non-homotopy-Tw}.
\end{proof}

\begin{remark}
    The morphism $\Linf \to \Perinf$ factors through a morphism $\Ainf \to \Perinf$, which sends a generator $\mu_n^\sigma \mapsto \tau_n^\sigma$ to the the linear rooted tree with~$n$ vertices decorated by the same permutation $\sigma$.  
    This is Koszul dual to the canonical chain of morphisms $\uPreLie \to \uAss \to \uCom$ between the unital Pre-Lie, associative and commutative operads.
    The fact that the morphism $\uPreLie \to \uCom$ factors through $\uAss$ is due to the fact that non-linear trees give $0$ when composed with the element~$u$ in $\uPreLie$ \cite[Lem.~4.27]{DotsenkoShadrinVallette23}.
\end{remark}

Unlike our previous examples: 

\begin{prop}
    The operad $\Perm$ is not a homotopy fixed point of $\Tw$.
\end{prop}

\begin{proof}
    First note that $\nu \eqdef \mu_2 \circ_2 \alpha$ is a cycle in $\Tw(\Perm)$: we have indeed
    \[d_{\Tw(\Perm)} \nu = \frac{1}{2}(\mu_2 \circ_2 \ell_2)(-, \alpha, \alpha) = 0 \]
    since the second defining relation~\eqref{eq:perm-relations} in $\Perm$ asserts that
    \begin{equation} 
    \label{zero divisor} 
    \mu_2 \circ_2 \ell_2 = \mu_2 \circ_2 \left(\mu_2 - \mu_2^{(12)}\right) = 0.
    \end{equation}
    We claim that $\nu$ is not a boundary in $\Tw(\Perm)$.
    Observe that the quotient of $\Perm$ by the operadic ideal generated by $\ell_2$ is the commutative operad $\Com$, and denote by $p : \Perm \to \Com$ the projection.
    Now  the image of $\nu$ under $\Tw(p)$ is not a boundary in $\Tw(\Com)$ since $p(\mu_2) \ne 0$ and $d_{\Tw(\Com)} = 0$. 
    Therefore $\nu = \mu_2 \circ_2 \alpha$ is not a boundary in $\Tw(\Perm)$ despite being sent to $0$ by the counit $\Tw(\Perm) \to \Perm$.
    This completes the proof.
\end{proof}

Consistency with \cref{thm:homotopy-fixed-point intro} requires that the Lie bracket is a right zero divisor in $\Perm$; this is true and indeed appears in the  above calculation in \eqref{zero divisor}. 


\subsection{The pre-Lie operad} 
\label{sec: prelie}
The~$\PreLie$ operad was introduced as the Koszul dual of~$\Perm$ in the previous \cref{sec: permutative}.
We have the following result. 

\begin{lemma}
    The Lie bracket is not a right zero divisor in $\PreLie$.
\end{lemma}

\begin{proof}
    According to \cite[Proposition ~4.5]{DotsenkoTamaroff21}, the morphism $\PreLie \to \Lie$ satisfies the PBW property.
    Therefore, we can repeat the proof of \cref{lemma:bracket-not-zero-divisor-in-Lie} by replacing $(\Ass, \Lie)$ by $(\PreLie, \Lie)$ and use the result of \cref{lemma:bracket-not-zero-divisor-in-Lie} to show that the canonical pre-Lie algebra map 
    \[
    \PreLie(x_0, \dots, x_{n-1}) \to \PreLie(u,v, x_0, \dots, x_{n-1}) / \left( \ell_2(u,v) - x_0 \right)
    \]
    is injective.
    The result then follows from \cref{lemma:criteria-bracket-zero-divisor}.
\end{proof}

It is known that  that $\Tw(\PreLie)$ is quasi-isomorphic to $\Lie$  \cite[Thm.~5.1]{DotsenkoKoroshkin24}, and so $\PreLie$ is not a homotopy fixed point for $\Tw$.  It follows from 
\cref{thm:homotopy-fixed-point intro} that $\PreLie$ does not admit a $\Tw$-coalgbra structure, recovering 
 \cite[Prop.~5.2]{DotsenkoShadrinVallette23}. 


\appendix


\section{Reconstruction}
\label{sec:reconstruction}

The definition of $\Curv$ given in the introduction differs from the definition we gave in \cite{Laplante-Petr-Shende2025}, which was the following: 

\begin{definition}[{\cite[Def.~1.1]{Laplante-Petr-Shende2025}}]
\label{curv}    
    We denote by ${}^*\!\Curv$ the category whose objects are given by the data of tuples $(\rmQ, d_\rmQ, \kurv, \rmQ_0)$ where $(\rmQ, d_\rmQ)$ is a dg operad, $\kurv \in \rmQ$ is a arity zero element of (homological) degree $(-1)$, and $\rmQ_0$ is a sub-operad of the underlying graded (not dg) operad $\rmQ$.  
    They must satisfy the conditions:
    \begin{enumerate}
        \item \label{free weights} The natural morphism $\rmQ_0 \vee [\kurv] \to \rmQ$ is an isomorphism of non-dg operads.
        \item \label{strict mixed} Under the resulting $\kappa$-grading $\rmQ = \bigoplus \rmQ_i$ (the notation $\rmQ_0$ is not ambiguous), consider the splitting $d = \sum d_i$ where $d_i(\rmQ_j) \subset \rmQ_{i+j}$.  Then all $d_i$ vanish except $d_0, d_1$.  
        \item \label{d1 closed} $d_1 \kurv = 0$.
    \end{enumerate}    
    Morphisms are morphisms of tuples. 
    We write ${}^*\!\nsCurv$ for the corresponding category of non-symmetric operads. 
\end{definition}

There is an evident candidate morphism $\Curv \to {}^* \! \Curv$ given by setting $\kappa \in \rmQ$ to be the image of $\mu_0$ and then forgetting the map from $\cLinf$; it is straightforward from the properties of $\cLinf$ and $\Curv$ that the resulting tuples satisfy \eqref{strict mixed} and \eqref{d1 closed} above. The fact that this morphism $\Curv \to {}^* \! \Curv$ is an equivalence, i.e.\ that one can uniquely reconstruct the morphism from $\cLinf$, is  \cite[Thm. 1.2]{Laplante-Petr-Shende2025}.  

Here we give deduce from said theorem an result for $\Twist$. 

\begin{definition}
\label{def:twist}
    Let ${}^* \!\Twist$ be the category whose objects are given by the data of tuples $(\rmQ, d_\rmQ, \kurv, \diff, \rmQ_{0,0})$, where  $(\rmQ, d_\rmQ)$ is a dg operad, $\rmQ_{0,0} \subset \rmQ$ is a non-dg sub-operad, $\kurv \in \rmQ$ is an arity zero, degree $(-1)$ element, and
    $\diff \in \rmQ$ is an arity one, degree $(-1)$ element.      
    They must satisfy the conditions: 
    \begin{enumerate}
        \item\label{item:coproduct} The natural map $\rmQ_{0,0} \vee [\kurv] \vee [\diff] \to \rmQ$ is an isomorphism of non-dg operads.
        \item\label{item:graded-mixed} Under the resulting grading $\rmQ = \bigoplus \rmQ_{i,j}$ where $i$ is the number of $\kurv$ and $(j-i)$ is the number of $\diff$ (the notation $\rmQ_{0,0}$ is not ambiguous), consider the splitting $d_\rmQ = \sum d_{i,j}$ where $d_{i,j}(\rmQ_{i',j'}) \subset \rmQ_{i'+i,j'+j}$.
        Then all $d_{i,j}$ vanish except for $i,j \in \{0,1\}$.
        \item\label{item:kurv-differential} $d_{0,0} \kurv = 0 = d_{1,1} \kurv$. 
        \item\label{item:diff-differential} $d_{0,0} \diff = 0 = d_{1,1} \diff$. 
        \item\label{item:general-differential} 
        $d_{0,1} \diff + \diff \circ_1 \diff = 0$ 
        \item \label{new differential effect} $d_{0,1} + [\diff,-] = 0$ on $\kappa$ and $\rmQ_{0, 0}$.\footnote{
    An element $\diff \in \rmQ(1)$ is said to be an operadic Maurer-Cartan if $d \diff + \diff \circ_1 \diff = 0$. In this case  $d^\diff := d + [\diff, -]$ is a new differential on $\rmQ$ (see \cite[Section 5.1]{DotsenkoShadrinVallette23}).
    With this terminology, \cref{def:twist} \eqref{item:general-differential} says  that $\diff$ is an operadic Maurer-Cartan in $(\rmQ, d_{0,1})$ and \eqref{new differential effect} says that $d_{0,1}^\diff=0$ on $\kappa$ and  $\rmQ_{0, 0}$.} 
    \end{enumerate}
    Morphisms are morphisms of tuples.
    We write ${}^* \!\nsTwist$ for the corresponding category of non-symmetric operads.
\end{definition}

One feature of this definition is the following.

\begin{lemma} 
\label{Q00 restriction}
    Fix an object $(\rmQ, d_\rmQ, \kurv, \diff, \rmQ_{0,0}) \in {}^* \! \Twist$ and a  $(\rmQ, d_\rmQ)$-algebra with structure map $r : (\rmQ, d_\rmQ) \to \End_{(A,d_A)}$.  If $r(\kappa) = 0$, then $(d_A + r(\diff))^2 = 0$ and $r$ gives $(A, d_A + r(\diff))$ the structure of a $(\rmQ_{0,0}, d_{0,0})$-algebra.   
\end{lemma}

\begin{proof}
    Let us write $\kurv_A := r(\kurv) \in A$ and $\diff_A := r(\diff) : A \to A$.  
    Conditions \eqref{item:diff-differential} and \eqref{item:general-differential} imply that $(d_A + \diff_A)^2 = r(d_\rmQ \diff + \diff \circ_1 \diff) = r(d_{1,0} \diff)$ is a multiple of $\kappa_A$.  
    In particular, if $\kappa_A = 0$, then $(d_A + \diff_A)^2 = 0$.  
    Finally we check that $r|_{\rmQ_{0,0}}$ intertwines $d_{0,0}$ with $d_A + \diff_A$: 
    \begin{align*}
        r(d_{0,0} \nu) & = r(d_\rmQ \nu - d_{0,1} \nu - d_{1,0} \nu - d_{1,1} \nu) \\
        & = r(d_\rmQ \nu) - r(d_{0,1} \nu) \\
        & = d_A(r \nu) + r([\diff, \nu]) = d_A(r\nu) + [\diff_A, r\nu] = (d_A + \diff_A)(r \nu).
    \end{align*}
\end{proof}

\begin{prop}
\label{reconstruction}
    The map $\Twist \to {}^*\! \Twist$ given by sending $(\cLinf \to \rmQ, \rmQ_{0,0})$ to the tuple $(\rmQ, d_{\rmQ}, \ell_0, \ell_1, \rmQ_{0,0})$ is an equivalence. 
\end{prop}

\begin{proof}
    Equivalently, we show that any object $(\rmQ, d_{\rmQ}, \kurv, \diff, \rmQ_{0,0})$ in ${}^*\! \Twist$ receives a unique morphism from $(\cLinf, d_{\cLinf}, \ell_0, \ell_1, \Linf)$.
    Observe that there is a faithful functor  
    \[{}^*\! \Twist \to {}^*\! \Curv, \quad (\rmQ, d_{\rmQ}, \kurv, \diff, \rmQ_{0,0}) \mapsto (\rmQ, d_{\rmQ}, \kurv, \rmQ_{0,0} \vee [\diff]). \]
    The latter sends $(\cLinf, d_{\cLinf}, \ell_0, \ell_1, \Linf)$ to the initial object of ${}^*\! \Curv$ (\cite[Thm.~1.2]{Laplante-Petr-Shende2025}), so there exists a unique morphism
    \[f: (\cLinf, d_{\cLinf}, \ell_0, \Linf \vee [\ell_1]) \to (\rmQ, d_{\rmQ}, \kurv, \rmQ_{0,0} \vee [\diff]).\]
    Therefore it suffices to show that $f$ lifts to a morphism in $\Twist$, i.e.\ that $f(\ell_1) = \diff$ and $f(\Linf) \subset \rmQ_{0,0}$.

    We first show that $f(\ell_1) = \diff$.
    Since $f$ is a dg morphism, that $f(\ell_0) = \kurv$, and that 
    \[d_{\cLinf}(\ell_0) = -[\ell_1, \ell_0], \quad d_{\rmQ}(\kurv) = -[\diff, \kurv],\]
    we get that $[f(\ell_1), \kurv] = [\diff, \kurv]$.
    By freeness of $\kurv$ and $\diff$, and since $f(\ell_1) \in \rmQ_{0,0} \vee [\diff]$, we get $f(\ell_1) = \diff$.
     
    We now recursively check that $f(\ell_n) \in \rmQ_{0,0}$ for $n \geq 2$.
    For $n=2$ we have 
    \[-f(\ell_2) \circ_1 \kurv = -f(\ell_2 \circ_1 \mu_0) = f(d_1 \ell_1) = (d_{1,0} + d_{1,1}) (f \ell_1) = d_{1,0} \diff \in \rmQ_{0,0} \vee [\kurv], \]
    where the last equality holds since $d_{1,1} \diff = 0$ by definition of ${}^*\! \Twist$. This implies that $f(\ell_2) \in \rmQ_{0,0}$.
    Assume that the result holds for a fixed $n \geq 2$. Then
    \[-f(\ell_{n+1}) \circ_1 \kurv = -f(\ell_{n+1} \circ_1 \ell_0) = f(d_1 \ell_n) = (d_{1,0} + d_{1,1}) (f \ell_n) = d_{1,1} (f \ell_n) \in \rmQ_{0,0} \vee [\kurv], \]
    where the last equality holds since $f(\ell_n) \in \rmQ_{0,0}$ by assumption.
    This implies $f(\ell_{n+1}) \in \rmQ_{0,0}$.
    This finishes the proof.
\end{proof}



\section{Operadic twisting via distributive comonads}
\label{section distributivity}

We  discussed the comonad $\cTw$ on $\cMult$.
There is another endofunctor on $\cMult$ that consists in adding and twisting by an operadic Maurer-Cartan, which we denoted by $(-)^+$ (\cref{def:+}).
In fact, there is a natural map $\rmP^+ \to (\rmP^+)^+$ sending $m \mapsto m_1+m_2$ which, together with the map $\rmP^+ \to \rmP$ sending $m \mapsto 0$, determine a comonad.

Given two comonads on the same category, one can  ask if one distributes over the other. 
A distributive law of a comonad $C_2$ over another comonad $C_1$ is a natural transformation $C_2 C_1 \to C_1 C_2$ satisfying compatibility conditions with respect to the comonad structures.
Distributive laws of $C_2$ over $C_1$ are in bijection with lifts of $C_2$ to the category of $C_1$-coalgebras. 
See \cite[Section~9.2]{barr2000toposes}.
Our results imply the following. 

\begin{prop}
\label{distributivity}
    The comonad $(-)^+$ distributes over the comonad $\cTw(-)$, i.e.\ $(-)^+$ lifts to the category of $\cTw$-coalgebras.
    In particular the composition $\cTw(-)^+$ defines a comonad on $\cMult$.
\end{prop}

\begin{proof}
    This follows from the fact that $\Curv$ is the category of $\cTw$-coalgebras (\cref{cTw comonadicity}), and the fact that $\calI \calJ$ is a lift of $(-)^+$ to $\Curv$ (\cref{monad IJ}).
\end{proof}

Additionally to these two comonads, we also encountered a monad on $\cMult$ via the adjunction induced by $\pi : \cLinf \to \Linf$ (see beginning of \cref{adjunction twist}).

\begin{lemma}
\label{monadic adjunction}
    The adjunction
    \[\pi_!: \cMult \leftrightarrow \Mult: \pi^!\]
    induced by $\pi : \cLinf \to \Linf$ is monadic, i.e.\ $\Mult$ is the category of algebras over the monad $\pi^! \pi_! : \rmQ \mapsto \rmQ / (\mu_0, \mu_1)$.
\end{lemma}
\begin{proof}
    An algebra for $\pi^! \pi_!$ is a dg-operad $\rmQ \in \cMult$ together with a morphism $f : \rmQ / (\mu_0, \mu_1) \to \rmQ$ such that $f \circ \pi_\rmQ = \id_\rmQ$.
    This implies that $\mu_0 = \mu_1 = 0$ in $\rmQ$, and therefore $\rmQ = (\pi^! \pi_!)(\rmQ)$.
    This completes the proof.
\end{proof}

Recall now that, given a comonad $T$ and a monad $M$ on the same category, the composition $MT$ defines a comonad on the category of $M$-algebras (note $(MT)(A)$ has an $M$-algebra structure given by the multiplication $\gamma^M_{TA}$ of $M$) with counit and comultiplication given by $\varepsilon^{MT}_A \eqdef a \circ M(\varepsilon^T_A)$ and $\delta^{MT}_A \eqdef MT(\eta^M_{TA}) \circ M(\delta^T_A)$
for an $M$-algebra $A$ with structural morphism $a : MA \to A$.

Applying this observation to the comonad $T = \cTw(-)^+$ (\cref{distributivity}) and the monad $M = \pi^! \pi_!$, one gets a comonad $MT$ on the category of $M$-algebras, i.e.\ on $\Mult$ (\cref{monadic adjunction}). 

We observe that this recovers $\Tw$:

\begin{prop}
    $MT = \cTw(-)^+/(\kappa + \ell_0^\alpha, m + \ell_1^\alpha) = \Tw$ on $\Mult$.            
\end{prop}


\section{Lift of the bar construction to $\coCurv$}
\label{bar-cobar}

Denote by $\eta^{B \Omega} : 1 \to B \Omega$ and $\varepsilon^{\Omega B} : \Omega B \to 1$ the unit and counit of the bar-cobar adjunction.

Since $\cLinf = \Omega(\uCom^*)$, every morphism $\cLinf \xrightarrow{f} \rmP$ induces a map
\[\psi := B(f) \circ \eta^{B \Omega}_{\uCom^*} : \uCom^* \to B \rmP, \]
and therefore a map $\Omega(\psi) : \cLinf \to \Omega B \rmP$.
Observe that 
\[\varepsilon^{\Omega B}_\rmP \circ \Omega(\psi) = \varepsilon^{\Omega B}_\rmP \circ (\Omega B)(f) \circ \Omega(\eta^{B \Omega}_{\uCom^*}) = f \circ \varepsilon^{\Omega B}_{\cLinf} \circ \Omega(\eta^{B \Omega}_{\uCom^*}) =f. \]
The second equality above holds because $\varepsilon^{\Omega B} : \Omega B \to 1$ is a natural transformation, and the second equality follows from the counit-unit relation since $\cLinf = \Omega(\uCom^*)$.

\begin{prop}
    Consider $\cLinf \xrightarrow{f} \rmP$ with corresponding morphism $\uCom^* \xrightarrow{\psi} B \rmP$.
    If $f(\ell_0) \ne 0$, then there exists a lift of $\uCom^* \xrightarrow{\psi} B \rmP$ to $\coCurv$, and therefore a lift of $\cLinf \xrightarrow{\Omega(\psi)} \Omega B \rmP$ to $\Curv$ (\cref{cobar of coCurv}).
\end{prop}

\begin{proof}
    We have to find a linear form $\form : (B \rmP)(0) \to \mathrm{k}$ such that $(\uCom^* \xrightarrow{\psi} B \rmP, \form)$ is in $\coCurv$, i.e.\ such that $\form \circ B(f) \circ \eta^{B \Omega}_{\uCom^*} = \ev_{\mu_0}$.
    Now observe that, since $f(\ell_0) \ne 0$,
    \[\psi(\mu_0^*) = s f(\ell_0) + s f(\ell_1) \otimes s f(\ell_0) + s f(\ell_2) \otimes (s f(\ell_0) \otimes s f(\ell_0)) + \cdots \ne 0 \]
    It remains to choose $\form : (B \rmP)(0) \to \mathrm{k}$ such that $\form(\psi \mu_0^*) = 1$.
    This completes the proof.
\end{proof}


\section{From $\cTw$ to group action via Yoneda}
\label{sec:group-action}

In \cite{DotsenkoShadrinVallette23}, the authors describe the twisting procedure of a $\Qinf$-algebra $A$ as a group action of $F_1 A$ on the Maurer-Cartan elements in the convolution Lie algebra of $\rmQ^\as$ and $\End_A$, i.e. on $\hom(\Qinf, \End_A)$.
In this section, we give a conceptual relationship between this approach and the comonad $\cTw$.

\vspace{2mm}

This relationship is based on the following basic observation: given a monoid object $G$ in a cartesian category $(\calC, \times, I)$, there is a corresponding monad on $\calC$ that sends $X \mapsto G \times X$.
The unit $X \cong I \times X \to G \times X$ is $\eta_G \times \id_X$, where $\eta_G : I \to G$ is the unit of $G$. The multiplication $G \times G \times X \to G \times X$ is $\mu_G \times \id_X$, where $\mu_G : G \times G \to G$ is the multiplication of $G$.
Given $X \in \calC$, an algebra structure on $X$ for this monad is exactly an action of~$G$ on~$X$.

\vspace{2mm}

We will apply this observation to the cartesian category $\calC$ of representable functors $H~:~\Op \to \Vect$.
The cartesian structure on $\calC$ is inherited from the cartesian structure on $\Vect$, so the unit $I$ sends $\rmP \mapsto 0$.
 
We consider the monoid object $G$ in $\calC$ which sends $\rmP$ to the space of degree $0$ elements in $F_1 \rmP(0)$.
Note that the functor $G$ is representable by the operad $\rmT$ generated by two arity~$0$ operations $\alpha$ and $\kappa$ in $F_1$ of degrees $0$ and $(-1)$ respectively, with $d_\rmT \alpha = \kappa$.
The unit $\eta_G : I \to G$ is the zero morphism for each  $\rmP$, and the multiplication $G \times G \to G$ is given by the addition on $\rmP(0)$. 

\vspace{2mm}

We now modify slightly the situation. 
Consider the specific object $H_\infty := \hom_\op(\cLinf, -)$ in $\calC$ (note that the functor $H_\infty(\rmP)$ is alternatively described as the set of Maurer-Cartan elements in the totalization Lie algebra $\bigoplus_{n \geq 0} \rmP(n)$).
We denote by $(\calC \downarrow H_\infty)$ the category of objects $H$ in $\calC$ together with a natural transformation $H \to H_\infty$.
We proceed to show that the monad $G \times (-)$ on $\calC$ induces a monad $M_G$ on $(\calC \downarrow H_\infty)$.
For that we need to define a morphism $\psi :G \times H \to H_\infty$ for each morphism $\varphi: H \to H_\infty$ which sends $f \in H(\rmP)$ to $\varphi_P^f \in H_\infty(\rmP)$. 
Given an operad $\rmP$, a degree $0$ element $a \in F_1 \rmP(0)$, and an element $f \in H(\rmP)$ we define $\psi_{\rmP}^{a, f} \in H_\infty(\rmP)$ by
\[
\psi_{\rmP}^{a, f}(\ell_0) := d_\rmP a + \sum_{k \geq 0} \frac{1}{k!} \varphi_{\rmP}^f(\ell_{k})(a^k), 
\quad \text{and} \quad 
\psi_{\rmP}^{a, f}(\ell_n) := \sum_{k \geq 0} \frac{1}{k!} \varphi_{\rmP}^f(\ell_{n+k})(a^k, -) \quad \text{for } n \geqslant 1.
\]

\begin{lemma}
    The monad $G \times (-)$ on $\calC$ induces a monad $M_G$ on $(\calC \downarrow H_\infty)$ via the formula above.
\end{lemma}
\begin{proof}
    Given $H \to H_\infty$, we need to check that the natural transformations $H \to G \times H$ and $G \times G \times H \to G \times H$ coming from the monad structure on $G$ commute with the morphisms to $H_\infty$. 
    This amount to the usual facts that twisting by zero does nothing, while twisting by $a$ then $b$ or by $a+b$ is the same.
\end{proof}

Note that, since $M_G$ is a monad on $(\calC \downarrow H_\infty)$, it is defines a comonad on the opposite category $(H_\infty \downarrow \calC^{op})$.
We can now state and prove the main result of this section.

\begin{theorem}
\label{Yoneda-power}
    The Yoneda functor 
    \[\cMult \to (H_\infty \downarrow \calC^{op}), \quad \rmQ \mapsto \hom_{\Op}(\rmQ, -) \]
    induces an equivalence of comonads between $\cTw$ and $M_G$.
    In particular, $\cTw$-coalgebra structures on $\rmQ$ are equivalent to $G$-actions on $\hom_{\Op}(\rmQ, -)$.
\end{theorem}
\begin{proof}
    Observe that we have
    \[\hom_{\Op}(\cTw(\rmQ), \rmP)) \cong \hom_{\Op}(\rmT, \rmP) \times \hom_{\Op}(\rmQ, \rmP) = (G \times \hom_{\Op}(\rmQ, -))(\rmP), \]
    because $\cTw(\rmQ)$ is $\rmQ \vee \rmT$ with morphism $\cLinf \to \cTw(\rmQ)$ given by
    \[
    \ell_0 \mapsto d_\rmT \alpha + \sum_{k \geq 0} \ell_k(\alpha^k), \quad \text{and} \quad 
    \ell_n \mapsto \sum_{k \geq 0} \frac{1}{k!} \ell_{n+k}(\alpha^k, -) 
    \quad \text{for } n \geqslant 1.
    \]
\end{proof}


\section{Twistable operads are $\Tw$-stable}
\label{sec:twistable-Tw-stable}

Consider a Koszul operad $\rmP=\rmP(E,R)$ with finite dimensional $E(n)$ for every $n$ and whose Koszul dual admits a unital extension $\rmP^! \hookrightarrow u_\chi\rmP^!$.
In this case, there is a meaningful notion of \emph{unital $\rmP^{!}$-algebras} encoded by $u_\chi \rmP^{!}$, and dually a notion of \defn{curved homotopy $\rmP$-algebras}, or $\cPinf^\chi$-algebras, encoded by Maurer--Cartan elements of the Pre-Lie convolution algebra 
\[
\hom_\Sym((u_\chi \rmP^{!})^{*},\End_A).
\]
According to \cite[Lemma 4.22]{DotsenkoShadrinVallette23}, the convolution Lie algebra $\hom_{\Sym}((u_\chi \rmP^{!})^*, \End_A)$ is isomorphic to $A \times \hom_{\Sym}((\rmP^{!})^{*}, \End_A)$.
In particular, $\calF_1 A$ acts on the set of Maurer-Cartan elements in $\hom_{\Sym}((u_\chi \rmP^{!})^*, \End_A)$ which is in bijection with $\hom_{\Op}(\cPinf^\chi, \End_A)$.
Therefore, given an element $a$ in $A$, and a $\cPinf^\chi$-algebra $\beta \in \hom_\Sym((u_\chi \rmP^{!})^{*},\End_A)$ on $A$, one can produce a new $\cPinf^\chi$-algebra \cite[Thm.~4.23]{DotsenkoShadrinVallette23}:
\begin{equation}
\label{gauge-group-action}
\begin{matrix}
        \calF_1 A \times \hom_\Sym((u_\chi \rmP^{!})^{*}, \End_A) & \to & 
\hom_\Sym((u_\chi \rmP^{!})^{*}, \End_A) \\
(a,\beta) & \mapsto & a \cdot \beta .
\end{matrix}
\end{equation}
Our goal is to show that this twisting procedure gives $\cPinf^\chi$ (resp.\ $\Pinf$) a $\cTw$-coalgebra (resp.\ $\Tw$-coalgebra) structure using only the methods of \cite{DotsenkoShadrinVallette23}. 

\subsection{$\cTw$-coalgebra structure on $\cPinf$}
Let us fix $\chi$ for the rest of the section and write~$\uP^! \eqdef u_\chi \rmP^!$ and $\cPinf \eqdef \cPinf^\chi$ from now on.
As we have shown in \cref{extendable-over-uCom}, there is a morphism of operads $\uP^! \to \uCom$.
Since $\hom_\Sym(-,-)$ is a bifunctor (with respect to strict morphisms), this gives a morphism of Pre-Lie algebras 
\begin{equation}
\label{cLinf-to-cPinf}
    \hom_\Sym((\uP^{!})^{*},\End_A) 
    \to 
    \hom_\Sym(\uCom^{*},\End_A).
\end{equation}
This is equivalent to a morphism of operads $\cLinf \to \cPinf$. 
Now, we claim that the gauge group action~\eqref{gauge-group-action}
gives a a morphism of operads
\begin{equation}
\label{cTw-coalgebra-map}
    \eta_{\cPinf} : \cPinf \to \cTw(\cPinf)
\end{equation}
which satisfies the $\cTw$-coalgebra axioms.

\medskip

Let us write $\rmC \eqdef (\uP^{!})^{*}$.
Given $a \in \calF_1 A$ and a Maurer-Cartan element $\beta \in \hom_{\Sym}(\rmC, \End_A)$, the result of the action is given by the composite
\[
a \cdot \beta : \rmC 
\xrightarrow{\Delta_\rmC} \rmC \, \hat \circ \, \rmC 
\xrightarrow{\beta \hat \circ (1+a)} \End_A \hat \circ \End_A 
\xrightarrow{\gamma} \End_A ,
\]
where $\Delta_\rmC$ is the full decomposition map of $\rmC$ and $\gamma$ is the full composition map of $\End_A$ (see \cite[after Definition~3.23]{DotsenkoShadrinVallette23}).
Explicitly,  choose a basis of $\uP^{!}$, and pick a basis element~$\mu$ of arity~$n$.
The image of $\mu^{*}$ under the decomposition map $\Delta_{\rmC}$ is given by the sum over all composites of basis elements in $\uP^{!}$ equal to $\mu$, permuted by some shuffles.
Let us denote by~$\omega_k^{*}((u^{*})^k,-)$ denote the sum of terms in $\Delta_{\rmC}(\mu^{*})$ in which $u^{*}$ appears exactly $k$ times.
Then, as explained in the proof of~\cite[Prop.~4.7]{DotsenkoShadrinVallette23}, the operation $(a \cdot \beta)(\mu^{*})$ takes the form
\begin{equation}
\label{eq:gauge-twisting}
    (a \cdot \beta)(\mu^{*}) = \beta(\mu^{*})+\sum_{k \geqslant 1}\beta(\omega_k^{*})(a^k,-).
\end{equation}
This translates immediately into a map $\eta_{\cPinf} : \cPinf \to  \cTw(\cPinf)$ defined by
\begin{equation}
\label{cTw-coalgebra-structure-morphism}
        s^{-1}\mu^* \mapsto s^{-1}\mu^* + \sum_{k \geqslant 1} (s^{-1}\omega_k^*)(\alpha^k,-),
\end{equation}
which we take as our definition for~\eqref{cTw-coalgebra-map}.
Here by $s^{-1}\omega_k^*$ we mean the sum $\omega_k^*$ viewed in $\cPinf$, i.e.\ with the $-1$ shift induced by the cobar construction.
The fact that this is a well-defined morphism of operads follows directly from the fact that $a \cdot \beta$ is a Maurer--Cartan element.
It remains to show that $\eta_{\cPinf}$ endows $\cPinf$ with a $\cTw$-coalgebra structure.  
The fact that post-composition with the counit $\cTw(\cPinf) \to \cPinf$, which sends $\alpha,\kurv \mapsto 0$, is the identity on $\cPinf$ is clear from~\eqref{cTw-coalgebra-structure-morphism}.
It remains to check that the following diagram commutes
\begin{equation}
\label{cTw-diagram}
\begin{tikzcd}
\cPinf \arrow[r, "\eta_{\cPinf}"] \arrow[d, "\eta_{\cPinf}"'] & \cTw(\cPinf) \arrow[d, "\cTw(\eta_{\cPinf})"] \\
\cTw(\cPinf) \arrow[r, "\eta(\cPinf)"]                 & \cTw(\cTw(\cPinf))     
\end{tikzcd}
\end{equation}
The bottom morphism is given by the comonad structure on $\cTw$.
The upper composite corresponds to twisting a $\cPinf$-algebra structure $\beta$ on $A$ by a Maurer--Cartan element $a$ and then by a Maurer--Cartan element $b$ to get the $\cPinf$-algebra structure $b\cdot (a \cdot \beta)$. The lower composite corresponds to twisting by the Maurer--Cartan element $b+a$ to get the $\cPinf$-algebra structure $(b+a) \cdot \beta$. 
The definition of the gauge group \cite[Def.~3.9]{DotsenkoShadrinVallette23} implies $b \cdot ( a \cdot \beta)=(b+a) \cdot \beta$.
Thus \eqref{cTw-diagram} commutes, and  $\eta_{\cPinf}$ defines a $\cTw$-coalgebra structure on $\cPinf$.

\subsection{$\Tw$-coalgebra structure on $\Pinf$}
One can see directly from the equation~\eqref{eq:gauge-twisting} that the Maurer--Cartan equation $(a \cdot \beta)(u^{*})=0$ coincides with the Maurer--Cartan equation coming from the $\cLinf$-algebra structure on $A$.
One can also see that $(a \cdot \beta)(\id^{*})$ coincides with the twisting  of the image of the $\cLinf$-operation $\ell_1$ on $A$ by the element $a$ : the only way  to obtain $\id$ as composite of operations in $\uP^{!}$ with a certain number of~$u$'s is to take these operations to be made strictly from generators in~$E(2)$. 
Therefore, whenever $(a \cdot \beta)(u^{*})=0$, the gauge action~\eqref{gauge-group-action} produces a $\Pinf^+$-algebra structure on $(A,d)$, or equivalently a $\Pinf$-algebra structure on $(A,d+\ell_1^a)$.
This is precisely the datum of a $\Tw$-coalgebra structure on~$\Pinf$, whose structure morphism $\eta_{\Pinf}$ is given by~\eqref{cTw-coalgebra-structure-morphism}.
